\documentclass[a4paper,10pt]{amsart}
\usepackage{pb-diagram}
\usepackage{xypic}
\usepackage{amscd}
\usepackage [curve]{xy}
\usepackage{bm}
\usepackage{slashbox}

\usepackage{amsmath,amssymb,amscd,bbm,bbding,amsthm,mathrsfs,hyperref}

\usepackage{latexsym}
\usepackage{amsfonts}
\usepackage{amsthm}
\usepackage{indentfirst}

 \newtheorem{thm}{Theorem}[section]
 \newtheorem{cor}[thm]{Corollary}
 \newtheorem{lem}[thm]{Lemma}
 \newtheorem{prop}[thm]{Proposition}
 \newtheorem{defn}[thm]{Definition}
 \newtheorem{ex}[thm]{Example}
 \newtheorem{rem}[thm]{Remark}

\def\k{\mathbbm{k}}

 \newcommand{\Hom}{\mathrm{Hom}}

\title{Homological properties of homologically smooth connected cochain DGAs}
\author{X.-F. Mao}
\address{Department of Mathematics, Shanghai University, Shanghai, China, 200444}
\address{Newtouch Center for Mathematics of Shanghai University, Shanghai, China, 200444}
\email{xuefengmao@shu.edu.cn}

\date{}

\begin{document}
 \def\abstactname{abstract}
\begin{abstract}
 Assume that $\mathscr{A}$ is a connected cochain DG algebra. We show that $\mathscr{A}$ is homologically smooth and Gorenstein if and only if its $\mathrm{Ext}$-algebra $H(R\Hom_{\mathscr{A}}(\mathbbm{k},\mathbbm{k}))$ is a Frobenius graded algebra. Moreover, $\mathscr{A}$ is Calabi-Yau
if and only if the $\mathrm{Ext}$-algebra $H(R\Hom_{\mathscr{A}}(\mathbbm{k},\mathbbm{k}))$ is a symmetric Frobenius graded algebra.
These generalize the corresponding results in \cite{HW1} and \cite{HM}, where the additional Koszul hypothesis is needed. 
\end{abstract}


%16E45 Differential graded algebras and applications
%16E10 Homological dimension

%%% ----------------------------------------------------------------------
\maketitle
%%% ----------------------------------------------------------------------
\section*{introduction}
The theory of differential graded algebras (DG algebras, for short) and their modules has numerous applications in classic ring theory, representation theory, and noncommutative algebraic geometry. In turn, the development of DG homological algebra usually benefits from the insights gained from the rich theories in these research fields. In DG homological algebra, the homologically smoothness of a DG algebra plays a similar important role as the regularity of an algebra does in the homological ring theory. Recall that a connected cochain DG algebra $\mathscr{A}$ is called homologically smooth if ${}_{\mathscr{A}^e}\mathscr{A}$ admits a minimal semi-free resolution $F$ which has a finite semi-free basis, or equivalently, ${}_{\mathscr{A}}\k$ admits a minimal semi-free resolution $K$ which has a finite semi-free basis (cf. \cite[Lemma 3.3]{MYH} or\cite[Lemma 2.16]{Mao}). The research on this fundamental property of DG algebras have attracted many people's interests. The author and Wu \cite{MW2} proved that any homologically
smooth connected cochain DG algebra $\mathscr{A}$ is cohomologically unbounded
unless $\mathscr{A}$ is quasi-isomorphic to the simple algebra $k$. And it was
proved that the $\mathrm{Ext}$-algebra
 of a homologically smooth DG algebra $\mathscr{A}$
is Frobenius if and only if both $\mathscr{D}^b_{lf}(\mathscr{A})$ and
$\mathscr{D}^b_{lf}(\mathscr{A}\!^{op})$ admit Auslander-Reiten triangles.  Shklyarov \cite{Sh} developed
a Riemann-Roch Theorem for homologically smooth DG algebras.
Besides these, some important classes of DG algebras are homologically smooth. For example, Calabi-Yau DG algebras introduced by Ginzburg in \cite{Gin} are homologically smooth by definition. Especially, non-trivial Noetherian DG down-up algebras and DG free algebras generated by two degree $1$ elements are Calabi-Yau by \cite{MHLX} and \cite{MXYA}, respectively. Moreover, there is a construction called `Calabi-Yau completion' \cite{Kel2} which produces a canonical Calabi-Yau DG algebra from a homologically smooth DG algebra. Recently, the author \cite{Mao} gives some criteria for a connected cochain DG algebra $\mathscr{A}$ to be homologically smooth in terms of the singularity category, the cone length of the canonical module $\mathbbm{k}$  and the global dimension of $\mathscr{A}$, when the graded algebra $H(\mathscr{A})$ is Noetherian.

 In \cite{HW1},  He-Wu introduced the concept of
Koszul DG algebras, and obtained a DG version of the Koszul duality for Koszul, homologically smooth and Gorenstein DG algebras. Let $\mathscr{A}$ be a connected cochain DG algebra.
Recall that $\mathscr{A}$ is Koszul if the trivial left DG $\mathscr{A}$-module $\mathbbm{k}$ admits a minimal semi-free resolution $F$ which has a semi-free basis concentrated in degree $0$. If $\dim_kH(R\Hom_{\mathscr{A}}(\mathbbm{k},\mathscr{A})) =1$, (resp. $\dim_{k}H(R\Hom_{\mathscr{A}^{op}}(\mathbbm{k},\mathscr{A}))=1$), then $\mathscr{A}$ is called left (resp. right) Gorenstein.   If $\mathscr{A}$ is both left Gorenstein and right Gorenstein, then we say that $\mathscr{A}$ is Gorenstein. By the way, homologically smooth DG algebras are called `regular' DG algebras in \cite{HW1} and \cite{MW2}.  For any homologically smooth DG algebra $\mathscr{A}$, it is left Gorenstein if and only if it is right Gorenstein by \cite[Remark 7.6]{MW2}.
Let $\mathscr{A}$ be a Koszul cochain DG algebra. He-Wu proved in \cite{HW1} that $\mathscr{A}$ is homologically smooth and Gorenstein if and only if its $\mathrm{Ext}$-algebra $H(R\Hom_{\mathscr{A}}(\mathbbm{k},\mathbbm{k}))$ is a Frobenius algebra.
This can be seen as a generalization of the corresponding result for graded case in \cite{Sm}.
A natural question is whether we can drop the Koszul condition as in the graded context. As far as the author is aware, it is still an open question whether the Gorensteinness of a homologically smooth connected cochain DG algebra
can be characterized by the Frobenius property of its Ext-algebra.
 In this paper, we give a positive answer to it. We have the following theorem (cf. Theorem \ref{Gorenstein}).
\\
\begin{bfseries}
Theorem\, A
\end{bfseries}
 Let $\mathcal{E}$  be the  Koszul dual DG algebra of a homologically smooth connected cochain DG algebra $\mathscr{A}$. Then the following conditions are equivalent.
\begin{enumerate}
\item The Ext-algebra $H(\mathcal{E})$ is a Frobenius graded algebra.
\item The DG algebra $\mathscr{A}$ is left Gorenstein.
\item The DG algebra $\mathscr{A}$ is right Gorenstein.
\item $(\mathcal{E}^*)_{\mathcal{E}}\in \mathscr{D}^c(\mathcal{E}^{op})$ and ${}_{\mathcal{E}}(\mathcal{E}^*)\in \mathscr{D}^c(\mathcal{E})$.
\item $\dim_{\mathbbm{k}}H(R\Hom_{\mathcal{E}}(\mathbbm{k},\mathcal{E}))<\infty$ and $\dim_{\mathbbm{k}}H(R\Hom_{\mathcal{E}^{op}}(\mathbbm{k},\mathcal{E}))<\infty$.
\item $\dim_{\mathbbm{k}}H(R\Hom_{\mathcal{E}}(\mathbbm{k},\mathcal{E}))=1$ and $\dim_{\mathbbm{k}}H(R\Hom_{\mathcal{E}^{op}}(\mathbbm{k},\mathcal{E}))=1$.
\item $\mathscr{D}^c(\mathcal{E})$ and $\mathscr{D}^c(\mathcal{E}^{op})$ admit Auslander-Reiten triangles.
\item $\mathscr{D}^b_{lf}(\mathscr{A})$ and $\mathscr{D}^b_{lf}(\mathscr{A}^{op})$ admit Auslander-Reiten triangles.
\end{enumerate}

By Theorem $A$, one sees that a connected cochain DG algebra $\mathscr{A}$ is homologically smooth and Gorenstein if and only if its Ext-algebra $H(R\Hom_{\mathscr{A}}(\mathbbm{k},\mathbbm{k}))$ is a graded Frobenius algebra.

A special family of Gorenstein homologically smooth DG algebras are called Calabi-Yau.
 Since Ginzburg introduced Calabi-Yau (DG) algebras in \cite{Gin}, they have been extensively studied due to their links to
mathematical physics, representation theory and non-commutative algebraic geometry. One well-known family of Calabi-Yau DG algebras are the Ginzburg DG algebras associated with finite quivers \cite{Gin,VdB}. Generally, it is difficult to judge whether a given DG algebra is Calabi-Yau.
If one considers a DG algebra $\mathcal{A}$ as a living thing, then the underlying graded algebra $\mathcal{A}^{\#}$  and the differential $\partial_{\mathcal{A}}$ are its body and soul, respectively.  The homological
properties of a DG algebra is determined by the joint effects of its graded algebra structure and differential structure.
So it is meaningful to find some ways to detect the Calabi-Yauness of a given DG algebra.

In \cite{MHLX} and \cite{MXYA}, non-trivial Noetherian DG
down-up algebras and DG free algebras generated by two degree one elements are proved Calabi-Yau, respectively.
Recently,  a connected cochain DG algebra $\mathscr{A}$ is proved to be Calabi-Yau if the trivial DG algebra $(H(\mathscr{A}),0)$ is Calabi-Yau \cite{MYY}.
The author and J.-W. He \cite{HM} give a criterion for a connected cochain DG algebra to be $0$-Calabi-Yau, and prove that a locally finite connected cochain DG algebra is $0$-Calabi-Yau if and only if it is defined by a potential. It is natural for one to ask whether one can get similar results for non-Koszul cases.
Another motivation of this paper is to seek some criteria for a Gorenstein and homologically smooth DG algebra to be Calabi-Yau.
 By \cite[Proposition 6.4]{MXYA}, any Calabi-Yau DG algebra is Gorenstein. While the converse is generally not true. For a $n$-Calabi-Yau connected cochain DG algebra $\mathscr{A}$,  the full triangulated subcategory $\mathscr{D}^b_{lf}(\mathscr{A})$ of $\mathscr{D}(\mathscr{A})$ containing DG $\mathscr{A}$-modules with finite dimensional total cohomology is a $n$-Calabi-Yau triangulated category (cf. \cite{CV}).
 The notion of Calabi-Yau triangulated category was introduced by Kontsevich \cite{Kon1} in the late 1990s. They appear in string theory,  conformal field theory, Mirror symmetry, integrable system and
 representation theory of quivers and finite-dimensional algebras. Due to the applications of triangulated Calabi-Yau categories in the categorification of Fomin-Zelevinsky's cluster
algebras \cite{FZ}, they have become popular in representation theory.
 In this paper,
 we obtain the following theorem (cf. Theorem \ref{cycond}).
\\
\begin{bfseries}
Theorem\, B
\end{bfseries}
   Let $\mathcal{E}$ be the Koszul dual DG algebra of a homologically smooth and Gorenstein DG algebra
$\mathscr{A}$. Then the following conditions are equivalent.
\begin{enumerate}
\item The DG algebra $\mathscr{A}$ is Calabi-Yau.
\item The Ext-algebra $H(\mathcal{E})$ is a symmetric Frobenius graded algebra.
\item The triangulated categories $\mathscr{D}^c(\mathcal{E})$ and $\mathscr{D}^c(\mathcal{E}^{op})$ are Calabi-Yau.
\item The triangulated categories $\mathscr{D}^b_{lf}(\mathscr{A})$ and $\mathscr{D}^b_{lf}(\mathscr{A}^{op})$ are Calabi-Yau.
\end{enumerate}

By Theorem A and Theorem B, one sees that a connected cochain DG algebra 
$\mathscr{A}$ is Calabi-Yau if and only if its Ext-algebra $H(\mathcal{E})$ is a symmetric Frobenius graded algebra.
This generalize the corresponding result for Koszul case in \cite{HM}.

As applications, we apply Theorem A and Theorem B to detect Gorensteinness and Calabi-Yauness of some examples of homologically smooth DG algebras. In Section \ref{appl}, we show that a connected cochain DG algebra is not Calabi-Yau but Gorenstein if its cohomology graded algebra is a cubic Artin-Schelter regular algebra of type A. In \cite{MHLX}, DG down-up algebras are introduced and systematically studied. It is proved there that all non-trivial Noetherian DG down-up algebras are Calabi-Yau. In this paper,  we apply Theorem B to reprove the Calabi-Yauness of one meaningful DG down-up algebra.
In the last section of this paper,  we give applications on DG free algebras introduced in \cite{MXYA}.  Recall that a connected cochain DG algebra $\mathscr{A}$ is called DG free if its underlying graded algebra $$\mathcal{A}^{\#}=\k\langle x_1,x_2,\cdots, x_n\rangle,\,\, \text{with}\,\, |x_i|=1,\,\, \forall i\in \{1,2,\cdots, n\}.$$
For the case $n=2$, it is proved in \cite{MXYA} that all those non-trivial DG free algebras are Koszul and Calabi-Yau. It is worth noting that as $n$ grows large, the classifications, cohomology and homological properties of the DG free algebras become increasingly difficult to compute and study. This increased complexity is, in large part, due to the irregular increase of the number of cases one needs to study separately.
Unlike the case of $n=2$, we have examples of non-Koszul and non-Gorenstein DG free algebras generated by three degree one elements.
In spite of this, there are Koszul and Calabi-Yau examples. For convenience, we summarize them by the following tabular.
\\
\begin{tabular}{l|llll}\hline
\backslashbox{cases}{properties}& Koszul & Gorenstein & homologically smooth & Calabi-Yau \\
 \hline
Example \ref{example1}  & \XSolid & \Checkmark &\Checkmark & \Checkmark \\
 Example \ref{ex3}  & \Checkmark & \XSolid &\Checkmark & \XSolid  \\
 Example \ref{ex2}, \ref{ex5} & \Checkmark &\Checkmark &\Checkmark & \Checkmark
 \\
 \hline
\end{tabular}

%%% ----------------------------------------------------------------------
\section{frobenius algebras}
In this section, we review some basics on Frobenius algebras, which are a family of algebras with some symmetric properties.
Throughout the paper, $\mathbbm{k}$ is a fixed algebraic closed field of characteristic zero.

\begin{defn}\cite{Nak1,Nak2}{\rm
A finite dimensional $\mathbbm{k}$-algebra $E$ is called Frobenius  if there is
a nondegenerate associative bilinear form
$\langle-,-\rangle: E\times E\to \mathbbm{k}$ satisfying
$\langle xy,z\rangle=\langle x,yz\rangle$ for any $x,y,z\in E$. }
\end{defn}

As a convention, the bilinear form in the definition of a Frobenius algebra is called the Frobenius form of the algebra.
For any $x,y,z, x_i,y_i\in E$ and $k\in \mathbbm{k}$, the following conditions hold.
\begin{enumerate}
\item $\langle xy,z\rangle=\langle x,yz\rangle$;
\item $\langle x_1+x_2,y\rangle =\langle x_1,y\rangle+\langle x_2,y\rangle$;
\item $\langle x,y_1+y_2\rangle=\langle x,y_1\rangle + \langle x,y_2\rangle$;
\item $\langle kx,y\rangle =\langle x,ky\rangle =k\langle x,y\rangle $;
\item  $\langle x,y\rangle=0$ for all $x\in E$ only if $y=0$.
\end{enumerate}

For any $\mathbbm{k}$-vector space $V$, we write $V^*=\Hom_{\mathbbm{k}}(V,\mathbbm{k})$. Let $\{e_i|i\in I\}$ be a basis of a finite dimensional $\mathbbm{k}$-vector space $V$.  We denote the dual basis of $V$ by $\{e_i^*|i\in I\}$, i.e., $\{e_i^*|i\in I\}$ is a basis of $V^*$ such that $e_i^*(e_j)=\delta_{i,j}$. Let $E$ be a Frobenius algebra. Then $E^*$ has a natural left and right $E$-module structure. It is well-known that the definition of Frobenius algebras is equivalent to the condition that $E\cong E^*$ as left $E$-modules or as right $E$-modules.
For each Frobenius algebra $E$, there is a $\mathbbm{k}$-algebra automorphism $\mu:E\to E$ such that $\langle\mu(a),b\rangle =\langle b,a\rangle$
for any $a,b\in E$. The automorphism $\mu$ is called Nakayama automorphism (cf. \cite{Mur}).

\begin{defn}
{\rm
Let $E$ be a Frobenius algebra.  If the Frobenius form $$\langle-, -\rangle: E\times E\to \mathbbm{k}$$ satisfies the condition that $\langle a,b\rangle=\langle b,a\rangle$ for any $a, b\in E$, then we say that $E$ is a symmetric Frobenius algebra.
}
\end{defn}

Now, let us transfer to graded context. For any graded vector space $W$ and $j\in\Bbb{Z}$,  the $j$-th suspension $\Sigma^j W$ of $W$ is a graded vector space defined by $(\Sigma^j W)^i=W^{i+j}$. Recall that a $\Bbb{Z}$-graded algebra is an associated algebra $A$ which is an associative algebra $A$ which is a direct sum of subspaces $$A=\bigoplus\limits_{i\in \Bbb{Z}}A^i$$ such that $A^i\cdot A^j\subseteq A^{i+j}$ for all $i,j\in \Bbb{Z}$. The elements in $A^i$ are said to be homogeneous of degree $i$.
The concept of (symmetric) Frobenius algebras have their graded versions in graded context.

\begin{defn}\label{gradedfrob}{\rm  Let $E$ be a finite dimensional graded $\mathbbm{k}$-algebra. Then it is called graded Frobenius if one of the following equivalent conditions hold.
\begin{enumerate}
\item There is an isomorphism of graded $E^e$ modules $E^*\cong \Sigma^{l}E_{\nu}$, for some graded algebra automorphism $\nu\in \mathrm{Aut}_{\mathbbm{k}}E$ and $l\in \Bbb{Z}$.
\item  There is an isomorphism of left $E$-modules: $\Sigma^{l}E\to E^*$, for some $l\in \Bbb{Z}$.

\item There is an isomorphism of right $E$-modules: $\Sigma^{l}E\to E^*$, for some $l\in \Bbb{Z}$.

\item There is an integer $d$ and a graded non-degenerate bilinear form $$\langle-, -\rangle: E\times E\to \Sigma^d \mathbbm{k}$$ such that $\langle ab,c\rangle =\langle a, bc\rangle$ for any $a,b,c\in E$.
\end{enumerate}
Furthermore, if the Frobenius form of a Frobenius graded algebra $E$  satisfies the condition that $\langle a,b\rangle=(-1)^{ij}\langle b,a\rangle$ for all graded elements $a\in E^i, b\in E^j$, then we say that $E$ is a symmetric Frobenius graded algebra.}

\end{defn}

Let $A$ be a connected $\Bbb{N}$-graded finitely $\mathbbm{k}$-algebra which is generated in degree $1$. Consider the trivial left $A$-module ${}_A\mathbbm{k}=A/A^{\ge 1}$, we can construct the associated
$\mathrm{Ext}$-algebra $E=E(A)=\oplus_{i=0}^{\infty}\mathrm{Ext}_{A}^i(\mathbbm{k},\mathbbm{k})$. It has a well-known multiplication given by the Yoneda product, and as such it is again a connected $\Bbb{N}$-graded algebra. The study of $\mathrm{Ext}$-algebras and their properties is an important general topic in homological algebra.
 In noncommutative algebraic geometry, it is well known that there are close relations between graded Frobenius algebras and Artin-Schelter regular algebras. Let
$$\cdots\to P_n\to \cdots\to P_2\to P_1\to P_0\to \mathbbm{k}\to 0$$
be the minimal free resolution of the trivial module ${}_Ak$. We say that $A$ is Koszul if each $P_i$ is generated in degree $i$, for each $i\ge 0$. For any Noetherian Koszul algebra $A$, it is Artin-Schelter regular if and only if its $\mathrm{Ext}$-algebra $E(A)$ is a graded Frobenius algebra by \cite{Sm}. In \cite[Corollary D]{LPWZ}, this result is generalized to non-Koszul case. Let $A$ be a connected $\Bbb{N}$-graded algebra generated in degree $1$ with $\mathrm{gl.dim}(A)<\infty$. It is proved that $A$ is AS-Gorenetein if and only if $E(A)$ is a graded Frobenius algebra. One sees that any connected graded algebra can be seen as a connected cochain DG algebra with zero differential. The motivation of this paper is to seek an analogous result in DG context.

%%% ----------------------------------------------------------------------
\section{preliminaries in dg homological algebra}
In this section, we review some basics on differential graded (DG
for short) homological algebra. There is some
overlap here with the papers \cite{MW1, MW2,FHT2}. It is assumed that
the reader is familiar with basics on DG modules, triangulated
categories and derived categories. If this is not the case, we refer
to  \cite{Nee, Wei} for more details on them.

  Let $\mathscr{A}$ be a $\Bbb{Z}$-graded
$\mathbbm{k}$-algebra.  If there is a $\mathbbm{k}$-linear map $\partial_{\mathscr{A}}: \mathscr{A}\to \mathscr{A}$  of
degree $1$ such that $\partial_{\mathscr{A}}^2 = 0$ and
\begin{align*}
\partial_{\mathscr{A}}(ab) = \partial_{\mathscr{A}}(a)b + (-1)^{n|a|}a\partial_{\mathscr{A}}(b)
\end{align*}
for all graded elements $a, b\in \mathscr{A}$, then  $\mathscr{A}$ is called a cochain
differential graded $\mathbbm{k}$-algebra. We write DG for differential graded.
For any cochain DG $\mathbbm{k}$-algebra $\mathscr{A}$, its
underlying graded algebra obtained by forgetting the differential of
$\mathscr{A}$ is denoted by $\mathscr{A}^{\#}$. If $\mathscr{A}^{\#}$ is a connected graded
algebra, then $\mathscr{A}$ is called a connected cochain DG algebra.
For the rest of this paper, we denote $\mathscr{A}$ as a
connected DG algebra over a field $\mathbbm{k}$ if no special assumption is
emphasized. The cohomology graded algebra of $\mathscr{A}$ is the graded algebra $$H(\mathscr{A})=\bigoplus_{i\in \Bbb{Z}}\frac{\mathrm{ker}(\partial_{\mathscr{A}}^i)}{\mathrm{im}(\partial_{\mathscr{A}}^{i-1})}.$$
 For any cocycle element $z\in \mathrm{ker}(\partial_{\mathscr{A}}^i)$, we write $\lceil z \rceil$ as the cohomology class in $H(\mathscr{A})$ represented by $z$. It is easy to
check that $H(\mathscr{A})$ is a connected graded algebra if $\mathscr{A}$ is a
connected DG algebra. We denote $\mathscr{A}\!^{op}$ as the opposite DG
algebra of $\mathscr{A}$, whose multiplication is defined as
 $a \cdot b = (-1)^{|a|\cdot|b|}ba$ for all
homogeneous elements $a$ and $b$ in $\mathscr{A}$. For any connected cochain DG algebra $\mathscr{A}$, it has the following maximal DG
ideal $$\frak{m}: \cdots\to 0\to \mathscr{A}^1\stackrel{\partial_1}{\to}
\mathscr{A}^2\stackrel{\partial_2}{\to} \cdots \stackrel{\partial_{n-1}}{\to}
\mathscr{A}^n\stackrel{\partial_n}{\to}
 \cdots .$$ Obviously, the enveloping DG algebra $\mathscr{A}^e = \mathscr{A}\otimes \mathscr{A}\!^{op}$ of $\mathscr{A}$
is also a connected DG algebra with $H(\mathscr{A}^e)\cong H(\mathscr{A})^e$, and its
maximal DG ideal is $\frak{m}\otimes \mathscr{A}^{op} + \mathscr{A}\otimes
\frak{m}^{op}$.

A DG $\mathscr{A}$-module $F$ is called DG free, if it is isomorphic to a
direct sum of suspensions of $\mathscr{A}$. We must emphasize that DG free $\mathscr{A}$-modules are not free objects in
the category of DG $\mathscr{A}$-modules.  Let $Y$ be a graded set, we denote
$\mathscr{A}^{(Y)}$ as the DG free DG module $\oplus_{y\in Y}\mathscr{A} e_y$, where
$|e_y|=|y|$ and $\partial(e_y)=0$. Let $M$ be a DG $\mathscr{A}$-module. A
subset $E$ of $M$ is called a \emph{semi-basis} if it is a basis of
$M^{\#}$ over $\mathscr{A}^{\#}$ and has a decomposition $E =
\bigsqcup_{i\ge0}E^i$ as a union of disjoint graded subsets $E^i$
such that
\begin{align*}
\partial(E^0)=0 \,\,\, \textrm {and} \,\,\,\partial(E^u)\subseteq \mathcal{A}
(\bigsqcup_{i<u}E^i)\, \,\,\textrm{for all}\,\, \,u >0.
%\in \Bbb{Z}.
\end{align*}
A DG $\mathscr{A}$-module $M$ is called semi-free if there is a
sequence of DG submodules
\begin{align*}
0=M_{-1}\subset M_{0}\subset\cdots\subset M_{n}\subset\cdots
\end{align*}
such that $M = \cup_{n \ge 0} \,M_{n}$ and that each $M_{n}/M_{n-1}=\mathscr{A}^{(Y)}$
is free on a basis $\{e_y|y\in Y\}$ of cocycles. We usually say that $M_n$ is an extension of $M_{n-1}$ since $M_{n-1}$ is a DG submodule of $M_{n-1}$, $M^{\#}_n=M^{\#}_{n-1}\oplus \mathcal{A}^{(Y)}$ and $\partial_{M_{n}}(e_y)\subseteq M_{n-1}$ for any $y\in Y$. Note that we have the following linearly split short exact sequence of DG $\mathscr{A}$-modules $$0\to M_{n-1}\to M_n\to M_n/M_{n-1}\to 0.$$
It is easy to see that a DG
$\mathscr{A}$-module is semi-free if and only if it has a semi-basis. A semi-free resolution of a DG $\mathscr{A}$-module $M$ is a
quasi-isomorphism $\varepsilon:F \to M$, where $F$ is a semi-free DG
$\mathscr{A}$-module. Sometimes we call $F$ itself a semi-free resolution of
$M$.  A semi-free resolution $F$ is called minimal if $\partial_F(F)\subseteq \frak{m}_{\mathscr{A}}F$.
By \cite[Proposition $2.4$]{MW1}, any DG $\mathscr{A}$-module with bounded below cohomology admits a minimal semi-free resolution. Another well-known semi-free resolution is called Eilenberg-Moore resolution, which is constructed from the free resolution of the cohomology graded module of the given DG module. It has a semi-basis which is in one to one correspondence with the free basis of the terms in that free resolution.
The readers can see a detailed proof of it in \cite[P. 279-280]{FHT2}.

We write $\mathscr{D}(\mathscr{A})$ for
the derived category of left DG modules over $\mathscr{A}$. We write
$\mathscr{D}^{b}(\mathscr{A})$ for the full subcategories of $\mathscr{D}(\mathscr{A})$,
whose objects are cohomologically
bounded.  We say a graded vector
space $M = \oplus_{i\in \Bbb{Z}}M^i$ is locally finite, if
each $M^i$ is finite dimensional. The full subcategory of
$\mathscr{D}(\mathscr{A})$ consisting of DG modules with locally finite
cohomology is denoted by $\mathscr{D}_{lf}(\mathscr{A})$.
A DG $\mathscr{A}$-module $M$ is compact if the functor $\Hom_{\mathscr{D}(\mathscr{A})}(M,-)$ preserves
all coproducts in $\mathscr{D}(\mathscr{A})$.
 Compact DG modules play the same role as finitely presented
modules of finite projective dimension do in ring theory (cf. \cite{Jor2}). It follows from \cite[Proposition 3.3]{MW1} that a DG $\mathscr{A}$-module $M$ is compact if and only if $M$ has a minimal semi-free resolution $F_M$, which has a finite semi-basis.
The full subcategory of $\mathscr{D}(\mathscr{A})$ consisting of compact DG $\mathscr{A}$-modules is denoted by $\mathscr{D}^c(\mathscr{A})$.

\begin{defn}\label{smooth}
{\rm
Let $\mathscr{A}$ be a connected cochain DG algebra. If ${}_{\mathscr{A}^e}\mathscr{A}\in \mathscr{D}^c(\mathscr{A}^e)$, then $\mathscr{A}$
is called homologically smooth. }
\end{defn}
By \cite[Corollary 2.7]{MW3}, $\mathscr{A}$ is homologically smooth if and only if
 $\mathbbm{k} \in \mathscr{D}^c(\mathscr{A})$.
Besides the homologically smoothness of a DG algebra, Koszul, Gorenstein and Calabi-Yau properties have been objects of much interest recently.
We list their definitions as follows.
\begin{defn}\cite{FHT1}\label{gordef}{\rm
Let $\mathscr{A}$ be a connected cochain DG algebra.
If $$\dim_{\mathbbm{k}}H(R\Hom_{\mathscr{A}}(\mathbbm{k},\mathscr{A})) =1,\,\,(\text{resp.} \dim_{\mathbbm{k}}H(R\Hom_{\mathscr{A}^{op}}(\mathbbm{k},\mathscr{A}))=1),$$ then $\mathscr{A}$ is called left (resp. right) Gorenstein.   If $\mathscr{A}$ is both left Gorenstein and right Gorenstein, then we say that $\mathscr{A}$ is Gorenstein. }
\end{defn}
 By \cite[Remark 7.6]{MW2}, any homologically smooth DG algebra is left Gorenstein if and only if it is right Gorenstein. And we have the following lemma.
\begin{lem}\cite[Proposition $5.1$]{HW1}\label{Goren}
Let $\mathscr{A}$ be a connected cochain DG algebra such that $H(\mathscr{A})$ is a left AS-Gorenstein graded algebra. Then $\mathscr{A}$ is a left Gorenstein DG algebra.
\end{lem}

\begin{defn}\label{CY}\cite{Gin,VdB}
{\rm If $\mathcal{A}$ is homologically smooth and $$R\Hom_{\mathcal{A}^e}(\mathcal{A}, \mathcal{A}^e)\cong
\Sigma^{-n}\mathcal{A}$$ in  the derived category $\mathscr{D}((\mathcal{A}^e)^{op})$ of right DG $\mathcal{A}^e$-modules, then $\mathcal{A}$ is called an $n$-Calabi-Yau DG algebra. }
\end{defn}
 By \cite[Proposition 6.4]{MXYA}, Calabi-Yau DG algebras are Gorenstein. However, the converse is generally not true even in Koszul case. The reader can check this by Proposition \ref{noncykosz}. In reference, there are two different versions of Koszul DG algebras.
 The earlier one is given by R. Bezrukavnikov in \cite{Bez}, where a DG algebra is called Koszul if its underlying graded algebra is Koszul.
  In \cite{HW1}, He-Wu gave a different definition of Koszul DG algebra. We must emphasize that these two definitions are not equivalent to each other. In this paper, we adopt the latter definition.
\begin{defn}\label{koszul}
{\rm Let $\mathscr{A}$ be a connected cochain DG algebra.
If ${}_{\mathscr{A}}\mathbbm{k}$, or equivalently ${}_{\mathscr{A}^e}\mathscr{A}$, has a minimal semi-free resolution with a semi-basis concentrated in degree $0$, then $\mathscr{A}$ is called Koszul.
}\end{defn}

The following lemma indicates that the concept of Koszul DG algebras in cochain cases is a natural generalization of Koszulness in graded context.
\begin{lem}\cite[Proposition $2.3$]{HW1}\label{koszul}
Let $\mathscr{A}$ be a connected cochain DG algebra such that $H(\mathscr{A})$ is a Koszul graded algebra.  Then $\mathscr{A}$ is a Koszul DG algebra.
\end{lem}
Due to the Koszul duality for Koszul, Gorenstein and homologically smooth DG algebras,  we have the following lemma which indicates that a homologically smooth and Koszul DG algebra is Gorenstein if and only if its Ext-algebra is Frobenius.
\begin{lem}\cite[Proposition $5.4$]{HW1}
Let $\mathscr{A}$ be a Koszul connected cochain DG algebra with Ext-algebra $E=H(R\Hom_{\mathscr{A}}(\mathbbm{k},\mathbbm{k}))$. Then $\mathscr{A}$ is homologically smooth and Gorenstein if and only if $E$ is a Frobenius algebra.
\end{lem}
In \cite{HM}, the author and J.-W. He give a criterion for a connected cochain DG algebra to be $0$-Calabi-Yau, and prove that a locally finite connected cochain DG algebra is $0$-Calabi-Yau if and only if it is defined by a potential. From these result, one sees a Koszul homologically smooth DG algebra is Calabi-Yau if and only if its Ext-algebra is a symmetric Frobenius algebra. One motivation of this paper is to obtain a similar result for non-Koszul cases.

\section{auslander-reiten theory and calabi-yau triangulated categories}
In this section, we will review basic notations and some important
results on Auslander-Reiten triangles and Calabi-Yau triangulated categories.

Let $\mathcal{C}$ be a $\Hom$-finite triangulated $\mathbbm{k}$-category. It is called
Krull-Schmidt if all the idempotent morphisms in $\mathcal{C}$ are
split, or equivalently the endomorphism ring of any indecomposable
object is local (see \cite[p.52]{Ri}). In this case, any object is uniquely decomposed into a direct sum of indecomposables, up to isomorphisms and up to the order of indecomposable direct summands. In the rest of this section, we assume that $\mathcal{C}$ is a $\Hom$-finite Krull-Schmidt triangulated $\mathbbm{k}$-category.
\begin{defn}{\rm
For any morphism $f:M\to N$ in $\mathcal{C}$, we say that $f$ is a section
if there is a morphism $g: N\to M$ such that $g\circ f = \mathrm{id}_M$.  And $f$ is called a retraction, if
there is a morphism $h: N\to M$ such that $f\circ h = \mathrm{id}_N$. }
\end{defn}
\begin{defn}\cite[Definition 2.1]{Kr2}{\rm
 An exact triangle
$M\stackrel{f}{\to}N\stackrel{g}{\to} P\stackrel{h}{\to}\Sigma M$
in $\mathcal{C}$ is called an Auslander-Reiten triangle if the
following conditions
are satisfied: \\
(AR1) Each morphism $M\to N'$ which is not a section factors
through
$f$.\\
(AR2) Each morphism $N'\to P$ which is not a retraction
factors
through $g$.\\
(AR3) $h\neq 0$.
}
\end{defn}

Given an object $P$ of $\mathcal{C}$ in the definition above, there
may or may not exist an Auslander-Reiten triangle. But if there
does, then it is determined up to isomorphism by \cite[Proposition
3.5(i)]{Hap1}. We denote $M$ as $\tau_{\mathcal{C}}P$, and the operation $\tau_{\mathcal{C}}$ is
called the Auslander-Reiten translation of $\mathcal{C}$. Note
that $\tau_{\mathcal{C}}P$ is only defined up to isomorphism. By \cite[Lem
2.3]{Kr1}, the endomorphism rings of the end terms $M$ and $P$ in an
Auslander-Reiten triangle are local. In particular, $M$ and $P$ are
indecomposable. By definition, $\mathcal{C}$ has Auslander-Reiten triangles if there is an Auslander-Reiten triangle
$M\to N\to P\to \Sigma M,$ for any indecomposable $P$. When it comes to the existence of Auslander-Reiten triangles, we have the following well-known lemma.

\begin{lem}\cite[Theorem 2.2]{Kr1}\label{brown}
Let $Z$ be a compact object in $\mathcal{C}$
such that the endomorphism ring $\Gamma=\Hom_{\mathcal{C}}(Z,Z)$ is a
local ring. Denote by $\mu:\Gamma/\mathrm{rad}\Gamma\to I$ the
injective envelope of $\Gamma\!^{op}$-module
$\Gamma/\mathrm{rad}\Gamma$. If there is an object $T_ZI$ in
$\mathcal{C}$ such that
$$\Hom_{\Gamma\!^{op}}(\Hom_{\mathcal{C}}(Z, -), I)\cong
\Hom_{\mathcal{C}}(-, T_ZI), \quad\quad\quad (*)$$ then there exists
an Auslander-Reiten triangle
$\Sigma^{-1}(T_ZI)\stackrel{\alpha}{\to} Y\stackrel{\beta}{\to}
Z\stackrel{\gamma}{\to} T_ZI,$ where $\gamma$ is the natural
morphism of
$$\Hom_{\mathcal{C}}(Z,Z)\stackrel{\pi}{\to}
\Gamma/\mathrm{rad}\Gamma \stackrel{\mu}{\to} I$$ under the
correspondence of $(*)$.
\end{lem}

Recall that a morphism $X\to Y$ in $\mathcal{C}$ is
called a pure monomorphism if the map
$\Hom_{\mathcal{C}}(C,X)\to \Hom_{\mathcal{C}}(C, Y)$ is a
monomorphism for any compact object $C$. An object $X$ in
$\mathcal{C}$ is called pure-injective, if each pure
monomorphism $X\to Y$ is a split monomorphism. Finally, a morphism
$\alpha: X\to Y$ in $\mathcal{C}$ is a pure-injective
envelope, if $Y$ is pure injective and if the composite
$\beta\circ\alpha$ is a pure monomorphism if and only if $\beta$ is
a pure monomorphism.

\begin{lem}\cite[Lemma 4.3]{Kr1}\label{pureinj}
Let $k$ be a commutative Noetherian, complete local ring and let
$\mathcal{C}$ be a compactly generated $k$-linear category such that
$\Hom_{\mathcal{C}}(X, Y)$ is a finitely generated $k$-module for any
compact object $X$ and $Y$. Then any compact object in $\mathcal{C}$
is pure-injective.
\end{lem}

\begin{prop}\cite[Proposition 3.2]{Kr1}\label{large}
Suppose that we have the following map between triangles in a
compactly generated triangulated category.
\begin{align*}
\xymatrix{\delta': \quad\quad  &X'\ar[r]\ar[d]^{\alpha} & Y'
\ar[d]\ar[r]& Z\ar[d]^{id_Z}\ar[r] &\Sigma X'\ar[d]
\\
\delta: \quad\quad  &X \ar[r] &Y \ar[r] &Z \ar[r] & \Sigma X
\\}
\end{align*}
Let the objects in the first triangle be compact, and let the
morphism $X'\to X$ be a pure-injective envelope. Then the following
statements are equivalent.

\emph{(1)}$\delta'$ is an Auslander-Reiten triangle in the category
of compact objects.

\emph{(2)}$\delta$ is an Auslander-Reiten triangle in the category
of all objects.
\end{prop}

Recall from Bondal and Kapranov\cite{BK} that a $\mathbbm{k}$-linear functor $F:\mathcal{C}\to \mathcal{C}$ is a right Serre functor if there exist isomorphisms
$$\eta_{A,B}: \Hom_{\mathcal{C}}(A,B)\to \Hom_{\mathcal{C}}(B,F(A))^*, \forall A,B\in \mathcal{C},$$
which are natural both in $A$ and $B$. Such a right $F$ is unique up to a natural isomorphism, and fully-faithful; if it is an equivalence, then a quasi-inverse $F^{-1}$ is a left Serre functor;  in this case we call $F$ a Serre functor. Recall that a left Serre functor is a functor $G: \mathcal{C}\to \mathcal{C}$ together with isomorphisms
$$\xi_{A,B}: \Hom_{\mathcal{C}}(A,B)\to \Hom_{\mathcal{C}}(G(B),A)^*, \forall A,B\in \mathcal{C}.$$
Note that $\mathcal{C}$  has a Serre functor if and only if it has both right and left Serre functor \cite{RV}.
By \cite[Theorem 1.2.4]{RV}, $\mathcal{C}$ has a Serre functor $F$ if and only if $\mathcal{C}$ has Auslander-Reiten triangles. In this case, $F$ coincides with $\Sigma\circ \tau_{\mathcal{C}}$ on objects, up to isomorphisms.

\begin{defn}\cite{Kon1}{\rm
Let $\Sigma$ be the shift functor of a Hom-finite triangulated category $\mathcal{C}$ with Serre functor $F$.
We call $\mathcal{C}$ a Calabi-Yau category if there is a natural isomorphism $F\cong \Sigma^d$ for some $d\in \Bbb{Z}$.  The minimal non-negative integer $d$ such that $F\cong \Sigma^d$ is called the Calabi-Yau dimension of $\mathcal{C}$.}
\end{defn}

 The notion of Calabi-Yau triangulated category was introduced by Kontsevich \cite{Kon1} in the late 1990s. Calabi-Yau triangulated  categories occur naturally in string theory, conformal field theory, integrable system and
 representation theory of quivers and finite-dimensional algebras. Due to the applications of triangulated Calabi-Yau categories in the categorification of Fomin-Zelevinsky's cluster
algebras \cite{FZ}, they are popular in representation theory. The following proposition indicates that we can use the notion of Calabi-Yau category to characterize the symmetric properties of a Frobenius graded algebra.
\begin{prop}
Let $E$ be a graded Frobenius algebra. Then $\mathscr{D}^c(E)$ is a Calabi-Yau category if and only if $E$ is a graded symmetric algebra.
\end{prop}
\begin{proof}
If $E$ is a graded Frobenius algebra, then there is an isomorphism of graded $E^e$ modules $E^*\cong \Sigma^{l}E_{\nu}$  by Definition \ref{gradedfrob}, for some graded algebra automorphism $\nu\in \mathrm{Aut}_{\mathbbm{k}}E$ and $l\in \Bbb{Z}$. For any indecomposable object of
$\mathscr{D}^c(E)$, there exists a natural equivalence
\begin{align}\label{serefun}
\Hom_{\mathscr{D}(E)}(Q,-)^*\simeq \Hom_{\mathscr{D}(E)}(-,E^*\otimes_E^LQ).
\end{align}

If $\mathscr{D}^c(E)$ is a Calabi-Yau category, then (\ref{serefun}) indicates that $E^*\otimes_E^L-$ is a Serre functor of $\mathscr{D}^c(E)$.
Since $\mathscr{D}^c(E)$ is a Calabi-Yau triangulated category, we have
$$\Sigma^l=E^*\otimes_E^L-=\Sigma^l {}_EE_{\nu}\otimes_E-\simeq \Sigma^l\circ \Phi,$$
where $\Phi:\mathscr{D}^c(E)\to \mathscr{D}^c(E)$ is the functor given by the operation $M\mapsto {}_\nu^{-1}M$. It implies that $\nu$ is an inner automorphism. Then $\Sigma^{-l}E^*\cong E$ in $\mathscr{D}(E^e)$. So $E$ is a symmetric Frobenius algebra.

Conversely, if the Frobenius algebra $E$ is symmetric, then $\nu$ is an inner automorphism and $E^*\cong \Sigma^{l}E$ in $\mathscr{D}(E^e)$. The natural equivalence in (\ref{serefun}) indicates that the functor $E^*\otimes_E^L-\cong \Sigma^l$ is a Serre functor of $\mathscr{D}^c(E)$.
\end{proof}

For a Calabi-Yau connected DG algebra $\mathscr{A}$, one sees that the full triangulated subcategory $\mathscr{D}^b_{lf}(\mathscr{A})$ of $\mathscr{D}(\mathscr{A})$ containing DG $\mathscr{A}$-modules with finite dimensional total cohomology is a Calabi-Yau triangulated category (cf. \cite{CV}).  It is natural to ask whether the converse is true. As far as the author is aware, it is still an open problem even if the DG algebra is homologically smooth and Gorenstein.

%%% ----------------------------------------------------------------------
\section{koszul dual dg algebra of a homologically smooth dg algebra}
In this section,  we assume that $\mathscr{A}$ is a homologically smooth connected cochain DG algebra.
One sees that both ${}_{\mathscr{A}}\mathbbm{k}$ and $\mathbbm{k}_{\mathscr{A}}$ are compact \cite[Lemma 2.16]{Mao}.
 Let $K$ and $L$ be the minimal semi-free resolutions of ${}_{\mathscr{A}}\mathbbm{k}$ and $\mathbbm{k}_{\mathscr{A}}$, respectively.  We have $\langle K\rangle = \langle {}_{\mathscr{A}}\mathbbm{k} \rangle$ and $\langle L\rangle =\langle \mathbbm{k}_{\mathscr{A}}\rangle$  in $\mathscr{D}(\mathscr{A})$.
Set $$\mathcal{N}=\langle {}_{\mathscr{A}}\mathbbm{k}\rangle^{\bot} =\langle {}_{\mathscr{A}}K\rangle^{\bot}, \mathscr{D}^{\mathrm{tors}}(\mathscr{A})={}^{\bot}\mathcal{N}\,\, \text{and}\,\, \mathscr{D}^{\mathrm{comp}}(\mathscr{A})=\mathcal{N}^{\bot}$$ in $\mathscr{D}(\mathscr{A})$. The DG modules in $\mathscr{D}^{\mathrm{tors}}(\mathscr{A})$ and $\mathscr{D}^{\mathrm{comp}}(\mathscr{A})$ are called torsion DG modules and complete DG modules, respectively.
Then $\mathscr{D}^{\mathrm{tors}}(\mathscr{A})=\langle{}_{\mathscr{A}}\mathbbm{k}\rangle=\langle{}_{\mathscr{A}}K\rangle$.
We call the endomorphism DG algebra $\mathcal{E}=\Hom_{\mathscr{A}}(K,K)$ by the Koszul dual DG algebra of $\mathscr{A}$.
By \cite[Lemma $5.4$]{Mao}, $\mathcal{E}$ satisfies the following conditions.
 \begin{enumerate}
\item $\dim_kH(\mathcal{E})<\infty$;
\item  $0=\sup\{i\in \Bbb{Z}|H^i(\mathcal{E})\neq 0\}$;
\item  $H^0(\mathcal{E})$ is a local finite dimensional algebra.
 \end{enumerate}
 Let $J$ be its maximal ideal.
Set $b=\inf\{i|H^i(\mathcal{E})\neq 0\}$, $Z^i=\mathrm{ker}(d_{\mathcal{E}}^i)$, $C^i=\mathcal{E}^i/Z^i$, $H^i=H^i(\mathcal{E})$ and $B^i=\mathrm{im}(d_{\mathcal{E}}^{i-1})$. Then $\mathcal{E}$ admits two DG subalgebras
$$\mathcal{E}':\quad\quad \cdots \stackrel{d_{\mathcal{E}}^{i-1}}{\to} \mathcal{E}^i\stackrel{d_{\mathcal{E}}^{i}}{\to} \mathcal{E}^{i+1}
\stackrel{d_{\mathcal{E}}^{i+1}}{\to}\cdots \stackrel{d_{\mathcal{E}}^{-2}}{\to}\mathcal{E}^{-1}\stackrel{d_{\mathcal{E}}^{-1}}{\to} Z^0\to 0$$
and
$$\mathcal{E}'':\quad\quad  \cdots \stackrel{d_{\mathcal{E}}^{j-1}}{\to} \mathcal{E}^j\stackrel{d_{\mathcal{E}}^{j}}{\to} \mathcal{E}^{j+1}
\stackrel{d_{\mathcal{E}}^{j+1}}{\to}\cdots  \stackrel{d_{\mathcal{E}}^{b-1}}{\to} B^b\to 0.$$ Clearly, $\mathcal{E}''$ is a DG ideal of $\mathcal{E}'$. Note that the DG algebra $\mathcal{E}'/\mathcal{E}''$ is
$$0\to C^b\oplus H^b \stackrel{d_{\mathcal{E}}^{b}}{\to} \mathcal{E}^{b+1}\stackrel{d_{\mathcal{E}}^{b+1}}{\to} \cdots \stackrel{d_{\mathcal{E}}^{-2}}{\to} \mathcal{E}^{-1} \stackrel{d_{\mathcal{E}}^{-1}}{\to} Z^0\to 0.$$
One sees that both the inclusion morphism $\iota: \mathcal{E}'\to \mathcal{E}$ and the canonical surjection $\varepsilon: \mathcal{E}'\to \mathcal{E}'/\mathcal{E}''$ are quasi-isomorphisms. We can apply the following lemma.
\begin{lem}\label{quasi-iso}\cite[Theorem III. 4.2]{KM}
Let $\phi:\mathscr{A}\to \mathscr{A}'$ be a quasi-isomorphism of DG algebras. Then the pull back functor
$\phi^*:\mathscr{D}(\mathscr{A}')\to \mathscr{D}(\mathscr{A})$ is an equivalence of categories with inverse given by the extension of scalars functor $\mathscr{A}'\otimes^L_{\mathscr{A}}-$.
\end{lem}

By Lemma \ref{quasi-iso},  we have two pairs of quasi-equivalences of categories as follows.
$$\xymatrix{\mathscr{D}(\mathcal{E})   \ar@<1ex>[r]^{\Phi}
& \mathscr{D}(\mathcal{E}') \ar@<1ex>[l]^{\Psi}  \ar@<1ex>[r]^{\Theta} & \mathscr{D}(\mathcal{E}'/\mathcal{E}'') \ar@<1ex>[l]^{\mathrm{\Omega}} } .$$
For convenience, we call $\mathcal{E}'$ by the truncated Koszul dual DG algebra of $\mathscr{A}$.
Write $Q=K^*{}^L\otimes_{\mathcal{E}}K$. Let $R_t=(\mathcal{E}')^{-t}$ and $d_t^R=d_{\mathcal{E}}^{-t}$ for any $t\ge 0$.  In this way, $\mathcal{E}'$ can be considered as a chain DG algebra $R$:
$$\cdots\stackrel{d^R_{i+1}}{\to} R_{i}\stackrel{d_{i}^R}{\to} R_{i-1}\stackrel{d_{i-1}^R}{\to}\cdots \stackrel{d_{2}^R}{\to} R_1\stackrel{d_{1}^R}{\to} R_0\to 0.$$
Moreover,
$H_0(R)=R_0/\mathrm{im}(d_1^R)\cong H^0$ is a finite dimensional local algebra and $\dim_{\mathbbm{k}}H(R)=\dim_{\mathbbm{k}}H(\mathcal{E}')=\dim_{\mathbbm{k}} H(\mathcal{E})<\infty$. Each $H_i(R)$ is a finitely generated $H_0(R)$-module and $-b=\sup\{i\in \Bbb{Z}|H_i(R)\neq 0\}$.  So $R$ is  a local chain DG algebra introduced in \cite{FJ}. Its maximal DG ideal is
$$\mathfrak{m}_R: \quad \cdots \stackrel{d^R_{i+1}}{\to} R_i\stackrel{d_i^R}{\to}\cdots \stackrel{d_{2}^R}{\to}R_1 \stackrel{d_1^R}{\to} R_0=B^0\oplus J\to 0.$$
\begin{rem}
The DG algebra $\mathcal{E}'$ and $\mathcal{E}$ are both augmented DG algebras with augmented DG ideals
\begin{align*}
\mathfrak{m}_{\mathcal{E}'}: \quad\quad \cdots \stackrel{d_{\mathcal{E}}^{i-1}}{\to} \mathcal{E}^i\stackrel{d_{\mathcal{E}}^{i}}{\to}\cdots \stackrel{d_{\mathcal{E}}^{-2}}{\to}\mathcal{E}^{-1} \stackrel{d_{\mathcal{E}}^{-1}}{\to} B^0\oplus J\to 0
\end{align*}
and
\begin{align*}
\mathfrak{m}_{\mathcal{E}}: \quad\quad \cdots \stackrel{d_{\mathcal{E}}^{i-1}}{\to} \mathcal{E}^i\stackrel{d_{\mathcal{E}}^{i}}{\to}\cdots \stackrel{d_{\mathcal{E}}^{-2}}{\to}\mathcal{E}^{-1} \stackrel{d_{\mathcal{E}}^{-1}}{\to} B^0\oplus J\oplus C^0 \stackrel{d_{\mathcal{E}}^{0}}{\to} \mathcal{E}^1 \stackrel{d_{\mathcal{E}}^{1}}{\to} \cdots \stackrel{d_{\mathcal{E}}^{j}}{\to} \mathcal{E}^j\stackrel{d_{\mathcal{E}}^{j+1}}{\to}\cdots.
\end{align*}
\end{rem}

 \begin{prop}\label{minres}
 Let $X$ be a left DG $\mathcal{E}'$-module such that each $H^i(X)$ is a finitely generated $H^0(\mathcal{E}')$-module and $u=\sup\{i|H^i(X)\neq 0\}<\infty$. Then $X$ admits a minimal semi-free resolution $F$ with $F^{\#}_X=\coprod\limits_{j\le u}\Sigma^j(\mathcal{E}'^{\#})^{(\beta_j)}$, where each $\beta_j$ is finite.
 \end{prop}

\begin{proof}
Let $M=\bigoplus\limits_{j\in \Bbb{Z}} M_j$ with $M_{-i}=X^i$ for any $i\in \Bbb{Z}$. Then $M$ is DG $R$-module such that each $H_i(M)$ is a finitely generate $H_0(R)$-module. And $H(M)$ is bounded below with $-u=\inf\{i|H_i(M)\neq 0\}$. It follows from \cite[0.5]{FJ} that $M$ admits a minimal semi-free resolution $G$ such that
$$G^{\#}=\coprod\limits_{i\ge -u}\Sigma^{-i}(R^{\#})^{(\beta_i)},$$ where each $\beta_i$ is finite. Let $F^i=G_{-i}$. Then $F$ is a minimal semi-free $\mathcal{E}'$-module with
$$F^{\#}=\coprod\limits_{j\le u}\Sigma^{-j}(\mathcal{E}'^{\#})^{\beta_j}.$$ Moreover, it is a minimal semi-free resolution of $X$.
\end{proof}

\begin{prop}\label{semifree}
Let $N$ be a DG $\mathcal{E}$-module such that $u=\sup\{i|H^i(X)\neq 0\}<\infty$ and each $H^i(N)$ is a finitely generated $H^0(\mathcal{E})$-module.
Then $N$ admits a minimal semi-free resolution $F$ such that $F^{\#}=\coprod\limits_{j\le u}\Sigma^{-j}(\mathcal{E}^{\#})^{\beta_j}$,  where each $\beta_j$ is finite.
\end{prop}

\begin{proof}
Via the inclusion morphism $\iota: \mathcal{E}'\to \mathcal{E}$, $N$ can be considered as a DG $\mathcal{E}'$-module. By Proposition \ref{minres}, ${}_{\mathcal{E}'}N$ admits a minimal semi-free resolution $G$ such that $$G^{\#}=\coprod\limits_{j\le u}\Sigma^{-j}(\mathcal{E}'^{\#})^{\beta_j},$$  where each $\beta_j$ is finite. One sees easily that $F=\mathcal{E}\otimes_{\mathcal{E}'}G$ is a minimal semi-free resolution of ${}_{\mathcal{E}}N$ and
$$F^{\#}=\coprod\limits_{j\le u}\Sigma^{-j}(\mathcal{E}^{\#})^{\beta_j},$$ where each $\beta_j$ is finite.
\end{proof}
By Proposition \ref{semifree} and \cite[Lemma A. 13]{Jor4}, we have the following corollary.
\begin{cor}\label{compact}
Let $N$ be a DG module in $\mathscr{D}_{fg}(\mathcal{E})$. Then $N$ is a compact DG $\mathcal{E}$-module if and only if it has a minimal semi-free resolution $F$  which admits a finite semi-basis.
\end{cor}

\begin{prop}\label{tensor}
Let $\mathcal{E'}$ be the truncated Koszul dual DG algebra of a homologically smooth DG algebra $\mathscr{A}$. If $M$ and $N$ are object in $\mathscr{D}_{fg}(\mathcal{E}'^{op})$ and $\mathscr{D}_{fg}(\mathcal{E}')$ respectively,  then $\sup\{i|H^i( M\otimes_{\mathcal{E}'}^L N )\neq 0\}=\sup\{i| H^i(M)\neq 0\}+\sup\{i|H^i(N)\neq 0\}$.
\end{prop}
\begin{proof}
Let $u=\sup\{i|H^i(M)\neq 0\}$ and $v=\sup\{i|H^i(N)\neq 0\}$.
By Proposition \ref{minres}, $M$ and $N$ admit minimal semi-free resolution $F_M$ and $F_N$ such that
$$F_M^{\#}=\coprod\limits_{j\le u}\Sigma^{-j}(\mathcal{E}'^{op\#})^{\beta_j}\quad \text{and}\quad F_N^{\#}=\coprod\limits_{j\le v}\Sigma^{-j}(\mathcal{E}'^{\#})^{\alpha_j},$$
where $\beta_j$ and $\alpha_j$ are finite.
Note that $\mathcal{E}'$ is concentrated in degrees $\le 0$. So we have the truncation
$$\Gamma^{\le u}M: \cdots \to M^{u-2}\stackrel{\partial_M^{u-2}}{\to} M^{u-1}\stackrel{\partial_M^{u-1}}{\to} \mathrm{ker}(\partial_M^u)\to 0. $$
It is a DG $\mathcal{E}'$-submodule of $M$ which is quasi-isomorphic to $M$.
Since $$(\Gamma^{\ge u}M\otimes_{\mathcal{E}'}F_N)^{\#}=\coprod\limits_{j\le v}\Sigma^{-j}(\Gamma^{\ge u}M)^{\alpha_j}$$ is concentrated in degrees $\le u+v$, we have $$
\sup\{i|H^i(M\otimes_{\mathcal{E}'}^LN)\neq 0\}=\sup\{i|H^i(\Gamma^{\ge u}M\otimes_{\mathcal{E}'}F_N)\neq 0\} \le u+v.$$
On the other hand, $F_M$ and $F_N$ admit semi-basis elements $\lambda$ and $\omega$ such that $|\lambda|=u$ and $|\omega|=v$.
Since $F_M\otimes_{\mathcal{E}'}F_N$ is concentrated in degrees $\le u+v$, the element $\lambda\otimes \omega$ is a cocycle element in $F_M\otimes_{\mathcal{E}'}F_N$. By the minimality of $F_M$ and $F_N$, one sees that $\lambda\otimes \omega$ is not a coboundary element since
$\lambda\otimes \omega\not\in F_M\frak{m}_{\mathcal{E}'}\otimes F_N+F_M\otimes \frak{m}_{\mathcal{E}'}F_N$.
Therefore, $u+v=\sup\{H^i(M\otimes_{\mathcal{E}'}^L N)\neq 0$.

\end{proof}

\begin{prop}\label{goren}
Let $\mathcal{E'}$ be the truncated Koszul dual DG algebra of a homologically smooth DG algebra $\mathscr{A}$. Then the following statements are equivalent to each other.
\begin{enumerate}
\item There are isomorphism $\mathbbm{k}$-vector spaces
$$H^i(R\Hom_{\mathcal{E}'}(\mathbbm{k},\mathcal{E}'))\cong \left\{
 \begin{aligned}
\mathbbm{k}, i=b\\
0, i\neq b
\end{aligned}\right\}\cong H^i(R\Hom_{\mathcal{E}'^{op}}(\mathbbm{k},\mathcal{E}')).$$

\item There are isomorphisms
${}_{\mathcal{E}'}(\mathcal{E}^*)\cong \Sigma^b{}_{\mathcal{E}'}\mathcal{E}'$ in $\mathscr{D}(\mathcal{E}')$ and $(\mathcal{E}'^*)_{\mathcal{E}'}\cong \Sigma^b \mathcal{E}'_{\mathcal{E}'}$ in $\mathscr{D}(\mathcal{E}'^{op})$.
\item $\dim_{\mathbbm{k}}H(R\Hom_{\mathcal{E}'}(\mathbbm{k},\mathcal{E}'))<\infty$ and $\dim_{\mathbbm{k}}H(R\Hom_{\mathcal{E}'^{op}}(\mathbbm{k},\mathcal{E}'))<\infty$.
\item The graded algebra $H(\mathcal{E}')$ is a Frobenius graded algebra. Equivalently,
there are isomorphisms of graded $H(\mathcal{E}')$-modules
$${}_{H(\mathcal{E}')}[H(\mathcal{E}')]^*\cong \Sigma^b{}_{H(\mathcal{E}')}H(\mathcal{E}')\quad\text{and}\quad [H(\mathcal{E}')]^*_{H(\mathcal{E}')}\cong \Sigma^bH(\mathcal{E}')_{H(\mathcal{E}')}.$$
\item ${}_{\mathcal{E}'}(\mathcal{E}'^*)\in \mathscr{D}^c(\mathcal{E}')$ and $(\mathcal{E}'^*)_{\mathcal{E}}\in \mathscr{D}^c(\mathcal{E}'^{op})$.

\end{enumerate}
\end{prop}

\begin{proof}
By Proposition \ref{minres}, $({}_{\mathcal{E}'}\mathcal{E}')^*$ admits a minimal semi-free resolution $F$.

(1)$\Leftrightarrow$ (2). By the minimality of $F$,
we have \begin{align*}
H(R\Hom_{\mathcal{E}'}(\mathbbm{k},\mathcal{E}'))&\cong H(R\Hom_{\mathcal{E}'}(\mathbbm{k},\mathcal{E}'^{**}))\\
&\cong H(\Hom_{\mathbbm{k}}(\mathcal{E}'^*\otimes_{\mathcal{E}'}^L\mathbbm{k},\mathbbm{k}))\cong \Hom_{\mathbbm{k}}(F\otimes_{\mathcal{E}'}\mathbbm{k},\mathbbm{k}).
\end{align*}
Hence
 $$H^i(R\Hom_{\mathcal{E}'}(\mathbbm{k},\mathcal{E}'))=\begin{cases}
\mathbbm{k}, i=b\\
0, i\neq b
\end{cases}$$ holds if and only if $F=\Sigma^b \mathcal{E}'_{\mathcal{E}'}$. We can symmetrically show that
$$H^i(R\Hom_{\mathcal{E}'^{op}}(\mathbbm{k},\mathcal{E}'))=\begin{cases}
\mathbbm{k}, i=b\\
0, i\neq b
\end{cases}$$ if and only if $(\mathcal{E}'_{\mathcal{E}'})^*\cong \Sigma^b {}_{\mathcal{E}'}\mathcal{E}'$.

(1)$\Rightarrow$ (3) and (2)$\Rightarrow$ (4) are trivial.

(3)$\Leftrightarrow$ (5) Note that $
H(R\Hom_{\mathcal{E}'}(\mathbbm{k},\mathcal{E}'))\cong \Hom_{\mathbbm{k}}(F\otimes_{\mathcal{E}'}\mathbbm{k},\mathbbm{k})$. This implies that $\dim_{\mathbbm{k}}H(R\Hom_{\mathcal{E}'}(\mathbbm{k},\mathcal{E}'))<\infty$ if and only if $F$ has a finite semi-basis, which is also equivalent
to $({}_{\mathcal{E}'}\mathcal{E}')^*\in \mathscr{D}^c(\mathcal{E}'^{op})$.  Similarly, $\dim_{\mathbbm{k}}H(R\Hom_{\mathcal{E}'^{op}}(\mathbbm{k},\mathcal{E}'))<\infty$ if and only if $(\mathcal{E}'_{\mathcal{E}'})^*\in \mathscr{D}^c(\mathcal{E}')$.

(4)$\Rightarrow$(5). If $H(\mathcal{E})$ is a Frobenius graded algebra,
then there is an isomorphism of graded $H(\mathcal{E})$-modules: $\Sigma^b H(\mathcal{E})\to [H(\mathcal{E})_{H(\mathcal{E})}]^*$. This implies
that $H[(\mathcal{E}_{\mathcal{E}})^*]$ has a graded free resolution $$0\to \Sigma^{b}H(\mathcal{E})\to
H[(\mathcal{E}_{\mathcal{E}})^*]\to 0.
$$
By the construction of the Eilenberg-Moore resolution, $(\mathcal{E}_{\mathcal{E}})^*$
admits a semifree resolution: $\Sigma^{b}\mathcal{E}\stackrel{\simeq}{\to}
(\mathcal{E}_{\mathcal{E}})^*$. Obviously $[\mathcal{E}_{\mathcal{E}}]^*$ is a compact DG $\mathcal{E}$-module. Similarly,
we can prove that $({}_{\mathcal{E}}\mathcal{E})^*$ is a compact DG $\mathcal{E}\!^{op}$-module.

(5)$\Rightarrow$ (1).  Since ${}_{\mathcal{E}'}(\mathcal{E}'^*)\in \mathscr{D}^c(\mathcal{E}')$,
we have \begin{align*}
H(R\Hom_{\mathcal{E}'}(\mathbbm{k},\mathcal{E}')\otimes_{\mathcal{E}'}^L(\mathcal{E}'^*))&\cong H(R\Hom_{\mathcal{E}'}(\mathbbm{k},\mathcal{E}'\otimes_{\mathcal{E}'}^L(\mathcal{E}'^*)))\\
&\cong H(R\Hom_{\mathcal{E}'}(\mathbbm{k},\mathcal{E}'^*))\cong \mathbbm{k}.
\end{align*}
By Corollary \ref{compact}, $(\mathcal{E}'^*)_{\mathcal{E}'}$ admits a minimal semi-free resolution $F$ which has a finite semi-basis.
We have \begin{align*}
\dim_{\mathbbm{k}}H(R\Hom_{\mathcal{E}'}(\mathbbm{k},\mathcal{E}'))&=\dim_{\mathbbm{k}}H(R\Hom_{\mathcal{E}'}(\mathbbm{k},\mathcal{E}'^{**}))\\
&=\dim_{\mathbbm{k}}H(\Hom_{\mathbbm{k}}((\mathcal{E}'^*)\otimes_{\mathcal{E}'}^L\mathbbm{k},\mathbbm{k}))\\
&=\dim_{\mathbbm{k}}(\Hom_{\mathbbm{k}}(F\otimes_{\mathcal{E}'}\mathbbm{k},\mathbbm{k}))<\infty.
\end{align*}
Hence $R\Hom_{\mathcal{E}'}(\mathbbm{k},\mathcal{E}')\in \mathscr{D}_{fg}(\mathcal{E}'^{op})$. As an object in $\mathscr{D}^c(\mathcal{E}')$, ${}_{\mathcal{E}'}(\mathcal{E}'^*)$ is also in $\mathscr{D}_{fg}(\mathcal{E}')$. Since $\sup\{i|H^i(\mathcal{E}'^*)\neq 0\}=-b$,
 we have
\begin{align*}
\sup\{i|H^i(R\Hom_{\mathcal{E}'}(\mathbbm{k},\mathcal{E}'))\neq 0\}-b=\sup\{i|H^i(\mathbbm{k})\neq 0\}=0
\end{align*}
by Proposition \ref{tensor}.
Then $\sup\{i|H^i(R\Hom_{\mathcal{E}'}(\mathbbm{k},\mathcal{E}'))=b$.
On the other hand,
\begin{align*}
\inf\{i|H^i(R\Hom_{\mathcal{E}'}(\mathbbm{k},\mathcal{E}'))\neq 0\}&=\inf\{i|H^i(R\Hom_{\mathcal{E}'}(\mathbbm{k},(\mathcal{E}')^{**}))\\
&=\inf\{i|H^i(\Hom_{\mathbbm{k}}((\mathcal{E}')^{*}\otimes_{\mathcal{E}'}^L\mathbbm{k},\mathbbm{k}))\neq 0\}\\
&=-\sup\{i|H^i[(\mathcal{E}')^{*}\otimes_{\mathcal{E}'}^L\mathbbm{k}]\neq 0\}\\
&\stackrel{(a)}{=}-[\sup\{i|H^i[(\mathcal{E}')^{*}]\neq 0\}+\sup\{i|H^i(\mathbbm{k})\neq 0\}]\\
&=-[(-b)+0]=b,
\end{align*}
where $(a)$ is by Proposition \ref{tensor} again. So $H(R\Hom_{\mathcal{E}'}(\mathbbm{k},\mathcal{E}'))$ is concentrated in degree $b$.
Applying the truncation technique to the DG $\mathcal{E}'^{op}$-module $R\Hom_{\mathcal{E}'}(\mathbbm{k},\mathcal{E}')$, we can easily show that
$R\Hom_{\mathcal{E}'}(\mathbbm{k},\mathcal{E}')\cong \Sigma^{-b} (\mathbbm{k}^{(\Lambda)})_{\mathcal{E}'}$
in $\mathscr{D}(\mathcal{E}'^{op})$, where $\Lambda$ is an indexed set. Inserting this into
$$H(R\Hom_{\mathcal{E}'}(\mathbbm{k},\mathcal{E}')\otimes_{\mathcal{E}'}^L(\mathcal{E}'^*))\cong \mathbbm{k}$$ shows that $\Lambda$ contains a single element.
Hence $R\Hom_{\mathcal{E}'}(\mathbbm{k},\mathcal{E}')\cong \Sigma^{-b}\mathbbm{k}_{\mathcal{E}'}$. It is equivalent to the first isomorphism in $(1)$. We can obtain the second one by a symmetric argument.
\end{proof}

\begin{rem}\label{replace}
Since the inclusion morphism $\iota: \mathcal{E}'\to \mathcal{E}$ is a quasi-isomorphism, we can obtain a similar result if we replace $\mathcal{E}'$ by $\mathcal{E}$ in Proposition \ref{goren}.

\end{rem}

 The DG module $K$ acquires the structure ${}_{\mathscr{A},{}_{\mathcal{E}}}K$ while $K^{\dagger}=\Hom_{\mathscr{A}}(K,\mathscr{A})$ has the structure $K^{\dagger}_{\mathscr{A},\mathcal{E}}$. Define functors $T(-)=-{\otimes}_{\mathcal{E}}^LK,$
$$W(-)=\Hom_{\mathscr{A}}(K,-)\simeq K^{\dagger}{\otimes}_{\mathscr{A}}^L- \,\, \text{and}\,\, C(-)=R\Hom_{\mathcal{E}^{op}}(K^{\dagger},-),$$
which form adjoint pairs $(T,W)$ and $(W,C)$ between $\mathscr{D}(\mathcal{E}^{op})$ and $\mathscr{D}(\mathscr{A})$.
There are pairs of quasi-inverse equivalences of categories as follows
$$\xymatrix{\mathscr{D}^{\mathrm{comp}}(\mathscr{A})   \ar@<1ex>[r]^{\quad W}
& \mathscr{D}(\mathcal{E}^{op}) \ar@<1ex>[l]^{\quad C}  \ar@<1ex>[r]^{T} & \mathscr{D}^{\mathrm{tors}}(\mathscr{A}) \ar@<1ex>[l]^{W} } .$$
In particular, $WC$ and $WT$ are equivalent to the identity functor on $\mathscr{D}(\mathcal{E}^{op})$ , so if we set $$\Gamma =TW, \Lambda=CW,$$
then we get endofunctors of $\mathscr{D}(\mathscr{A})$ which form an adjoint pair $(\Gamma, \Lambda)$ and satisfy
$$\Gamma^2\simeq \Gamma, \Lambda^2\simeq \Lambda, \Gamma\Lambda\simeq \Gamma, \Lambda\Gamma\simeq \Lambda.$$
These functors are adjoints of inclusions as follows, where left-adjoint are displayed above right-adjoints
$$\xymatrix{\mathscr{D}^{\mathrm{comp}}(\mathscr{A})   \ar@<1ex>[r]^{\quad \mathrm{inc}}
& \mathscr{D}(\mathscr{A}) \ar@<1ex>[l]^{\quad \Lambda}  \ar@<1ex>[r]^{\Gamma} & \mathscr{D}^{\mathrm{tors}}(\mathscr{A}) \ar@<1ex>[l]^{ \mathrm{inc }} } .$$
Write $Q=K^{\dagger}{}^L\otimes_{\mathcal{E}}K$ and $D=Q^{*}=\Hom_{\mathbbm{k}}(Q,\mathbbm{k})$. One sees that $Q$ and $D$ have the structures ${}_{\mathscr{A}}Q_{\mathscr{A}}$ and ${}_{\mathscr{A}}D_{\mathscr{A}}$,  respectively.
\begin{prop}\label{mainthm}
Let $\mathscr{A}$ be a connected cochain DG algebra. Then $\mathscr{A}$ is Gorenstein  and homologically smooth if and only if its $\mathrm{Ext}$-algebra $H(R\Hom_{\mathscr{A}}(\mathbbm{k},\mathbbm{k}))=H(\mathcal{E})$ is a graded Frobenius algebra.
\end{prop}

\begin{proof}
If the DG algebra $\mathscr{A}$ is homologically smooth and Gorenstein, then we have
$\dim_{\mathbbm{k}}H(R\Hom_{\mathscr{A}}(\mathbbm{k},\mathbbm{k}))<\infty$ and $\dim_{\mathbbm{k}}H(W(\mathscr{A}))=\dim_{\mathbbm{k}}H(R\Hom_{\mathscr{A}}(\mathbbm{k},\mathscr{A}))=1$.
By \cite[Lemma 2.7]{MGYC}, there exists $m\in \Bbb{Z}$ such that $$W(A)=R\Hom_{\mathscr{A}}(\mathbbm{k},\mathscr{A})=\Hom_{\mathscr{A}}(K,\mathscr{A})\simeq \Sigma^{m}\mathbbm{k}_{\mathscr{A},\mathcal{E}}$$ in $\mathscr{D}(\mathscr{A}^{op})$.
Then
\begin{align*}
\mathbbm{k}=H(R\Hom_{\mathscr{A}}(\mathscr{A},\mathbbm{k}))&=H(R\Hom_{\mathscr{A}}(CW(\mathscr{A}),TW(\mathbbm{k})))\\
&=H(R\Hom_{\mathcal{E}^{op}}(WCW(\mathscr{A}),WTW(\mathbbm{k})))\\
&=H(R\Hom_{\mathcal{E}^{op}}(W(\mathscr{A}),W(\mathbbm{k})))\\
&=H(R\Hom_{\mathcal{E}^{op}}(\Sigma^m\mathbbm{k}_{\mathscr{A},\mathcal{E}},\mathcal{E}))\\
&=\Sigma^{-m}H(R\Hom_{\mathcal{E}^{op}}(\mathbbm{k}_{\mathscr{A},\mathcal{E}},\mathcal{E})).
\end{align*}
So $H(R\Hom_{\mathcal{E}^{op}}(\mathbbm{k},\mathcal{E}))=\Sigma^{m}k$. Since $\dim_kH(\mathcal{E})<\infty$, we have $\mathcal{E}\simeq \mathcal{E}^{**}$ as a DG $\mathcal{E}^{op}$-module.
The DG $\mathcal{E}$-module $\mathcal{E}^*$ admits a minimal semi-free resolution $F_{\mathcal{E}^{*}}$. Then $F_{\mathcal{E}^{*}}^*$ is a minimal semi-injective resolution of
$\mathcal{E}^{**}$.
Then
\begin{align*}
\Sigma^m\mathbbm{k}=H(R\Hom_{\mathcal{E}^{op}}(\mathbbm{k},\mathcal{E}))&=H(\Hom_{\mathcal{E}^{op}}(\mathbbm{k},F_{\mathcal{E}^{*}}^*)\\
&=H(\Hom_{\mathbbm{k}}(\mathbbm{k}\otimes_{\mathcal{E}}F_{\mathcal{E}^*},k)),\\
\end{align*}
which implies that $F_{\mathcal{E}^*}=\Sigma^{-m}\mathcal{E}$. Therefore the graded $H(\mathcal{E})$-module $H(\mathcal{E})^*$ is isomorphic to $\Sigma^{-m}H(\mathcal{E})$.
By definition, the Ext-algebra $H(\mathcal{E})$ is a graded Frobenius algebra.

Conversely, if the Ext-algebra $H(\mathcal{E})$ is graded Frobenius, then $\dim_kH(\mathcal{E})<\infty$ and the graded right $H(\mathcal{E})$-module $H(\mathcal{E})^*$ is isomorphic to $\Sigma^{-m}H(\mathcal{E})$.
Since $$\dim_{\mathbbm{k}}\Hom_{\mathscr{A}}(K,\mathbbm{k})= \dim_{\mathbbm{k}}H(R\Hom_{\mathscr{A}}(\mathbbm{k},\mathbbm{k})=\dim_{\mathbbm{k}}H(\mathcal{E})<\infty,$$ the minimal semi-free resolution $K$ of ${}_{\mathscr{A}}\mathbbm{k}$ admits a finite semi-free basis.
Thus $\mathscr{A}$ is homologically smooth. And the DG right $\mathcal{E}$-module $\mathcal{E}^*$ admits a semi-free resolution $f:\Sigma^{-m}\mathcal{E}\to \mathcal{E}^*$ by the construction of Eilenberg-Moore resolution.
Since the inclusion $\iota: \mathcal{E}'\to \mathcal{E}$ is a quasi-isomorphism, the DG right $\mathcal{E}'$-module $\mathcal{E}'^*$ admits a semi-free resolution $f:\Sigma^{-m}\mathcal{E}'\to \mathcal{E}'^*$. The morphism of DG $\mathcal{E}'$-modules
\begin{align*}
 &\theta : \mathcal{E}' \mapsto \mathcal{E}'^{**}\\
&x \mapsto (f\mapsto (-1)^{|f|\cdot|x|}f(x))
\end{align*} is a quasi-isomorphism since $H(\mathcal{E}')$ is finite dimensional. Then the composition
$$\mathcal{E}'\stackrel{\theta}{\to} \mathcal{E}'^{**}\stackrel{f^*}{\to} [\Sigma^{-m}\mathcal{E}']^*$$ is a quasi-isomorphism and hence a semi-injective resolution of ${}_{\mathcal{E}'}\mathcal{E}'$.
Thus $$H(R\Hom_{\mathcal{E}'}(\mathbbm{k},\mathcal{E}'))=H(\Hom_{\mathcal{E}'}(\mathbbm{k}, [\Sigma^{-m}\mathcal{E}']^*))=\Sigma^m\mathbbm{k}.$$
We want to show that $$\dim_kW(\mathscr{A})=\dim_{\mathbbm{k}}H(R\Hom_{\mathscr{A}}(\mathbbm{k},\mathscr{A}))=\dim_{\mathbbm{k}}H(\Hom_{\mathscr{A}}(K,\mathscr{A}))=1.$$
By \cite[Proposition 2.4]{MW1}, we have $$K^{\#}\cong \coprod\limits_{0\le i\le
u}\Sigma^{-i}(\mathscr{A}^{\#})^{(\Lambda^i)},$$ where each $\Lambda^i$ is a finite
index set. Then $$\Hom_{\mathscr{A}}(K,\mathscr{A})^{\#}\cong \coprod\limits_{0\le i\le
u}\Sigma^{i}(\mathscr{A}^{\#})^{(\Lambda^i)} $$ is locally finite and bounded below since $\mathscr{A}$ is locally finite and positively graded. Let $\inf\{i|H^i(\Hom_{\mathscr{A}}(K,\mathscr{A}))\neq 0\}=b.$
We claim that $\Hom_{\mathscr{A}}(K,\mathscr{A}))\cong \Sigma^b \mathbbm{k}$ as a DG $\mathcal{E}^{op}$-module. Note that any
DG $\mathcal{E}^{op}$-module can be considered as a DG $\mathcal{E'}^{op}$-module via the quasi-isomorphism $\iota: \mathcal{E}'\to \mathcal{E}$.
It suffices to show that
$\Hom_{\mathscr{A}}(K,\mathscr{A})\cong \Sigma^b \mathbbm{k}$ as a DG $\mathcal{E'}^{op}$-module. We can apply truncation technique to the DG $\mathcal{E'}^{op}$-module
$\Hom_{\mathscr{A}}(K,\mathscr{A})$. There exists a bounded below $\mathcal{E'}^{op}$-module $S$ such that $S$ is quasi-isomorphic to $\Hom_{\mathscr{A}}(K,\mathscr{A})$ and
$b=\inf\{i|S^i\neq 0\}$. Since $S^b$ is the left-most non-zero component of $S$, there is a surjection of right $\mathcal{E}'^0$-modules $S^b\to H^b(S)$.
 Moreover, $H^b(S)\cong H^b(\Hom_{\mathscr{A}}(K,\mathscr{A}))$ is  finitely generated as a left $H^0(\mathcal{E}')$-module since
 $\dim_kH^b(\Hom_{\mathscr{A}}(K,\mathscr{A}))<\infty$. Nakayama's Lemma indicates that there exists a surjection of right $H^0(\mathcal{E}')$-modules
$H^b(S)\to \mathbbm{k}$. Altogether, there is a surjection of right $\mathcal{E}'^0$-modules, $S^b\to \mathbbm{k}$. This gives rise to a surjection of right $\mathcal{E}'$-modules
$\beta: S\to \Sigma^{-b}\mathbbm{k}$. Let $S'$ be its kernel. Then we have a short exact sequence of right DG $\mathcal{E}'$-modules
\begin{align}\label{shortseq}
0\to S'\stackrel{i}{\to}S\stackrel{\beta}{\to}\Sigma^{-b}\mathbbm{k}\to 0.
\end{align}
Acting the functor $R\Hom_{\mathcal{E}'^{op}}(-,\mathcal{E}')=\Hom_{\mathcal{E}'^{op}}(-,(\Sigma^{-m}\mathcal{E}')^*)$ on the short exact sequence (\ref{shortseq}) gives a new short exact sequence
$$0\to \Sigma^{m+b}\mathbbm{k}\stackrel{\beta^*}{\to} \Sigma^mS^*\stackrel{i^*}{\to} \Sigma^{m}S'^*\to 0,$$
whose cohomology long exact sequence contains
\begin{align*}
\stackrel{H^{m+b-1}(i^*)}{\to}H^{m-1}(S'^*)\stackrel{\delta^{m+b-1}}{\to}\mathbbm{k}\stackrel{H^{m+b}(\beta^*)}{\to}H^m(S^*)\stackrel{H^{m+b}(i^*)}{\to}H^m(S'^*)\stackrel{\delta^{m+b}}{\to}0.
\end{align*}
Obviously, $H^{m+b}(\beta^*)$ is an injection since $\beta^*$ is injective. On the other hand, we have
\begin{align*}
\mathbbm{k}=H(R\Hom_{\mathscr{A}}(\mathscr{A},\mathbbm{k}))&=H(R\Hom_{\mathscr{A}}(CW(\mathscr{A}),TW(\mathbbm{k})))\\
&=H(R\Hom_{\mathcal{E}^{op}}(WCW(\mathscr{A}),WTW(\mathbbm{k})))\\
&=H(R\Hom_{\mathcal{E}^{op}}(W(\mathscr{A}),W(\mathbbm{k})))\\
&=H(R\Hom_{\mathcal{E}^{op}}(W(\mathscr{A}),\mathcal{E}))\\
&=H(R\Hom_{\mathcal{E}'^{op}}(W(\mathscr{A}),\mathcal{E}))\\
&=H(R\Hom_{\mathcal{E}'^{op}}(W(\mathscr{A}),\mathcal{E}'))\\
&=H(R\Hom_{\mathcal{E}'^{op}}(S,\mathcal{E}'))\\
&=H(\Hom_{\mathcal{E}'^{op}}(S,(\Sigma^{-m}\mathcal{E}')^*))\\
&=\Sigma^mH(S^*).
\end{align*}
So $H^{m+b}(\beta^*)$ is an isomorphism and $H(S'^*)$ is quasi-trivial. Since $H(S'^*)\cong H(S')^*$, we conclude that $H(S')$ is quasi-trivial and
$S\cong \Sigma^{-b}\mathbbm{k}$ as a right DG $\mathcal{E}'$-module. Then $\dim_{\mathbbm{k}}H(\Hom_{\mathscr{A}}(K,\mathscr{A}))=\dim_{\mathbbm{k}}H(S)=1$ and the homologically smooth DG algebra $\mathscr{A}$ is hence Gorenstein.

\end{proof}

\section{gorensteinness of homologically smooth dg algebras}
Let $\mathscr{A}$ be a homologically smooth connected cochain DG algebra. Suppose that
 $\mathcal{E}$ is its Koszul dual DG algebra. In this section, we will give some criteria for $\mathscr{A}$ to be Gorenstein.
\begin{prop}\label{krull-sch}
Let $\mathcal{E}$ be the Koszul dual DG algebra of  a homologically smooth DG algebra $\mathscr{A}$.
 Then $\mathscr{D}^c(\mathcal{E})$ and $\mathscr{D}^c(\mathcal{E}^{op})$ are both
$\Hom$-finite, $k$-linear Krull-Schmidt triangulated categories.
\end{prop}

\begin{proof}
Let $M$ and $N$ be objects in $\mathscr{D}^c(\mathcal{E})$. By Corollary
\ref{compact}, $M$ and $N$ admit minimal semifree resolutions
$F_M$ and $F_N$ respectively, and both $F_M$ and $F_N$ have a finite
semibasis.  Hence $F_{M}$ and $F_N$ have the following semifree
filtrations of finite length $$ F_{M}(0)\subset
F_{M}(1)\subset\cdots\subset F_{M}(n) = F_{M}\quad \text{and}\quad
F_N(0)\subset F_N(1)\subset\cdots\subset F_N(l) = F_N,$$ where
$F_{M}(0), F_N(0)$ and each $F_{M}(i)/F_{M}(i-1), F_N(j)/F_N(j-1)$
are DG free $\mathcal{E}$-modules on a cocycle basis.
We have two sequences of semi-split short exact sequences of
DG $\mathcal{E}$-modules:
\begin{equation}\label{qian}
0\to F_{M}(i)\to F_{M}(i+1) \to F_{M}(i+1)/F_{M}(i)\to 0 \quad (i=0,
1, \cdots, n-1),
\end{equation}
and
\begin{equation}\label{hou}
0\to F_N(j)\to F_N(j+1) \to F_N(j+1)/F_N(j)\to 0 \quad (i=0, 1,
\cdots, l-1).
\end{equation}
By those cohomology long exact sequences induced from (\ref{hou})
and the fact that $\dim_k H(\mathcal{E})<\infty$, we can prove
inductively that $H(N)$ is locally finite. Acting the functor
$\Hom_{\mathcal{E}}(-,N)$ on (\ref{qian}), we can get a new sequence of short
exact sequences ($i=0, 1, \cdots, n-1$):
\begin{small}
\begin{align*}
0\to \Hom_{\mathcal{E}}(F_{M}(i+1)/F_{M}(i),N) \to \Hom_{\mathcal{E}}(F_M(i+1),N) \to
\Hom_{\mathcal{E}}(F_M(i),N) \to 0.
\end{align*}
\end{small}\\
By those homology long exact sequences induced from this sequence of
short exact sequences, we can prove inductively that
 $H^0(\Hom_{\mathcal{E}}(F_M, N))
 \cong \Hom_{\mathscr{D}(\mathcal{E})}(M,N)$ is finite dimensional.

Since $\mathscr{D}(\mathcal{E})$ has arbitrary coproducts,  any idempotent
morphism in $\mathscr{D}(\mathcal{E})$ is split by \cite[Proposition 3.2]{BN}.
Certainly, $\mathscr{D}^c(\mathcal{E})$ is a full triangulated subcategory of
$\mathscr{D}(\mathcal{E})$ closed under direct summands. Hence any idempotent
morphism in $\mathscr{D}^c(\mathcal{E})$ is split. So $\mathscr{D}^c(\mathcal{E})$
is a $\Hom$-finite, $\mathbbm{k}$-linear Krull-Schmidt triangulated category.
\end{proof}

\begin{prop}\label{arcompact}
Let $\mathcal{E}$ be the Koszul dual DG algebra of  a homologically smooth DG algebra $\mathscr{A}$. Then
 $\mathscr{D}^c(\mathcal{E})$ admits Auslander-Reiten triangles if and only if $(\mathcal{E}_{\mathcal{E}})^*$
is a compact DG $\mathcal{E}$-module.
\end{prop}

\begin{proof}
Let $P$ be an indecomposable object in $\mathscr{D}^c(\mathcal{E})$. By
Proposition \ref{krull-sch}, $\Gamma=\Hom_{\mathscr{D}(\mathcal{E})}(P,P)$ is
local. It is easy to check that $\Hom_{\mathbbm{k}}(\Gamma,\mathbbm{k})$ is an injective
envelope of $\Gamma/\mathrm{rad}\Gamma$ as a $\Gamma$-module.
Applying the same proof as \cite[Lemma 4.1]{Jor3}, we may prove the
natural equivalence $\Hom_{\mathscr{D}(\mathcal{E})}(P,-)^*\simeq
\Hom_{\mathscr{D}(\mathcal{E})}(-,\mathcal{E}^*\otimes_{\mathcal{E}}^LP).$ Hence we have
$\Hom_{\Gamma}(\Hom_{\mathscr{D}(\mathcal{E})}(P,-), \Hom_{\mathbbm{k}}(\Gamma, \mathbbm{k}))\simeq
\Hom_{\mathscr{D}(\mathcal{E})}(-,\mathcal{E}^*\otimes_{\mathcal{E}}^LP).$
Using Proposition \ref{brown}, we can get the following
Auslander-Reiten triangle: $$(*): \Sigma^{-1}(\mathcal{E}^*\otimes_{\mathcal{E}}^LP)\to
N\to P\to \mathcal{E}^*\otimes_{\mathcal{E}}^LP$$ in $\mathscr{D}(\mathcal{E})$. If $(\mathcal{E}_{\mathcal{E}})^*$ is a compact DG $\mathcal{E}$-module, then $\mathcal{E}^*\otimes_{\mathcal{E}}^LP$ is
also a compact DG $\mathcal{E}$-module since $\mathcal{E}$ is a compact DG $\mathcal{E}$-module.
Hence $(*)$ is an Auslander-Reiten triangle in $\mathscr{D}^c(\mathcal{E})$. So
$\mathscr{D}^c(\mathcal{E})$ admits Auslander-Reiten triangles.

Conversely, if $\mathscr{D}^c(\mathcal{E})$ admits Auslander-Reiten triangles,
we can get the following Auslander-Reiten triangle:
$$(**): \tau \mathcal{E}\to Y\to \mathcal{E}\to \Sigma(\tau \mathcal{E})$$ in
$\mathscr{D}^c(\mathcal{E})$, since
$\Hom_{\mathscr{D}^c(\mathcal{E})}(\mathcal{E},\mathcal{E})=\Hom_{\mathscr{D}(\mathcal{E})}(\mathcal{E},\mathcal{E})=H^0(\mathcal{E})$ is
a local ring.  By Lemma \ref{pureinj} and Proposition \ref{large},
it is easy to prove that $(**)$ is also an Auslander-Reiten triangle
in $\mathscr{D}(\mathcal{E})$.

On the other hand, if the DG module $P$ in $(*)$ is replaced by $\mathcal{E}$,
we can get the following Auslander-Reiten triangle:
$$(***): \Sigma^{-1}(\mathcal{E}^*\otimes_{\mathcal{E}}^L\mathcal{E})\to N\to \mathcal{E}\to
\mathcal{E}^*\otimes_{\mathcal{E}}^L\mathcal{E} $$ in $\mathscr{D}(\mathcal{E})$. By \cite[Proposition
3.5(i)]{Hap1}, $(**)$ and $(***)$ are isomorphic to each other in
$\mathscr{D}(\mathcal{E})$. So $\mathcal{E}^*$ is a compact DG $\mathcal{E}$-module.
\end{proof}

\begin{thm}\label{Gorenstein}
 Let $\mathcal{E}$  be the  Koszul dual DG algebra of a homologically smooth connected cochain DG algebra $\mathscr{A}$. Then the following conditions are equivalent.
\begin{enumerate}
\item the Ext-algebra $H(\mathcal{E})$ is a Frobenius graded algebra;
\item $\mathscr{A}$ is left Gorenstein;
\item $\mathscr{A}$ is right Gorenstein;
\item $(\mathcal{E}^*)_{\mathcal{E}}\in \mathscr{D}^c(\mathcal{E}^{op})$ and ${}_{\mathcal{E}}(\mathcal{E}^*)\in \mathscr{D}^c(\mathcal{E})$;
\item $\dim_{\mathbbm{k}}H(R\Hom_{\mathcal{E}}(\mathbbm{k},\mathcal{E}))<\infty$ and $\dim_{\mathbbm{k}}H(R\Hom_{\mathcal{E}^{op}}(\mathbbm{k},\mathcal{E}))<\infty$;
\item $\dim_{\mathbbm{k}}H(R\Hom_{\mathcal{E}}(\mathbbm{k},\mathcal{E}))=1$ and $\dim_{\mathbbm{k}}H(R\Hom_{\mathcal{E}^{op}}(\mathbbm{k},\mathcal{E}))=1$;
\item $\mathscr{D}^c(\mathcal{E})$ and $\mathscr{D}^c(\mathcal{E}^{op})$ admit Auslander-Reiten triangles;
\item $\mathscr{D}^b_{lf}(\mathscr{A})$ and $\mathscr{D}^b_{lf}(\mathscr{A}^{op})$ admit Auslander-Reiten triangles;
\end{enumerate}

\end{thm}
\begin{proof}
(7)$\Rightarrow $(6).  If $\mathrm{D}^c(\mathcal{E})$ and $\mathscr{D}^c(\mathcal{E}\!^{op})$ admit
Auslander-Reiten triangles, then $(\mathcal{E}_{\mathcal{E}})^*$ and $({}_{\mathcal{E}}\mathcal{E})^*$ are compact by Proposition \ref{arcompact}.
Then there are isomorphisms of $\mathbbm{k}$-vector spaces
$$H^i(R\Hom_{\mathcal{E}}(\mathbbm{k},\mathcal{E}))\cong \left\{
 \begin{aligned}
\mathbbm{k}, i=b\\
0, i\neq b
\end{aligned}\right\}\cong H^i(R\Hom_{\mathcal{E}^{op}}(\mathbbm{k},\mathcal{E}))$$
by Remark \ref{replace} and Proposition \ref{goren}.

(6)$\Rightarrow$ (7). If $\dim_{\mathbbm{k}}H(R\Hom_{\mathcal{E}}(\mathbbm{k},\mathcal{E}))=1$ and $\dim_{\mathbbm{k}}H(R\Hom_{\mathcal{E}^{op}}(\mathbbm{k},\mathcal{E}))=1$,
then $$\dim_{\mathbbm{k}}H(R\Hom_{\mathcal{E}}(\mathbbm{k},\mathcal{E})) =
\dim_{\mathbbm{k}}H(\Hom_{\mathbbm{k}}(F_{({}_{\mathcal{E}}\mathcal{E})^*}\otimes_{\mathcal{E}}\mathbbm{k}, \mathbbm{k})) =1$$ and
$$\dim_{\mathbbm{k}}H(R\Hom_{\mathcal{E}\!^{op}}(\mathbbm{k},\mathcal{E}))= \dim_kH(\Hom_{\mathbbm{k}}(\mathbbm{k}\otimes_{\mathcal{E}} F_{(\mathcal{E}_{\mathcal{E}})^*},\mathbbm{k})) =1,$$
where $F_{({}_{\mathcal{E}}\mathcal{E})^*}$ and $F_{(\mathcal{E}_{\mathcal{E}})^*}$ are the minimal semifree
resolutions of $({}_{\mathcal{E}}\mathcal{E})^*$ and $(\mathcal{E}_{\mathcal{E}})^*$ respectively. Hence
$F_{(\mathcal{E}_{\mathcal{E}})^*}\cong \Sigma^{u}{}_{\mathcal{E}}\mathcal{E}$ and $F_{({}_{\mathcal{E}}\mathcal{E})^*} \cong
\Sigma^{u}\mathcal{E}_{\mathcal{E}}$. Clearly, we have $(\mathcal{E}_{\mathcal{E}})^*\in \mathscr{D}^c(\mathcal{E})$ and
$({}_{\mathcal{E}}\mathcal{E})^*\in \mathscr{D}^c(\mathcal{E}\!^{op})$. By Proposition
\ref{arcompact}, both $\mathscr{D}^c(\mathcal{E})$ and $\mathscr{D}^c(\mathcal{E}\!^{op})$
admit Auslander-Reiten triangles.

(7)$\Leftrightarrow$ (8). By \cite[Lemma 8.2]{Jor3},
there are quasi-inverse equivalences of categories,
\begin{align*}
\xymatrix{&\mathscr{D}^b_{lf}(\mathscr{A})\quad\quad\ar@<1ex>[r]^{\Hom_{\mathscr{A}}(K,
-)}&\quad\quad
\mathscr{D}^{c}(\mathcal{E}^{\!op})\ar@<1ex>[l]^{-\otimes_{\mathcal{E}}^L
K}}.
\end{align*}
So $\mathscr{D}^b_{lf}(\mathscr{A})$ and $\mathscr{D}^{c}(\mathcal{E}^{\!op})$ admit Auslander-Reiten triangles simultaneously. By a symmetric argument, we can show that $\mathscr{D}^b_{lf}(\mathscr{A}^{op})$ has Auslander-Reiten triangles if and only if $\mathscr{D}^{c}(\mathcal{E})$ has Auslander-Reiten triangles.

(8)$\Leftrightarrow$ (1). By \cite[Theroem 10.8]{MW2}, $\mathscr{D}^b_{lf}(\mathscr{A})$  and $\mathscr{D}^b_{lf}(\mathscr{A}^{op})$ admit  Auslander-Reiten triangles if and only if the Ext-algebra $H(\mathcal{E})$ is a Frobenius graded algebra.

(1)$\Leftrightarrow$ (2)$\Leftrightarrow$ (3) Since $\mathscr{A}$ is homologically smooth, it is left Gorenstein if and only if it is right Gorenstein by \cite[Remark 7.6]{MW2}. By Proposition \ref{mainthm}, $\mathscr{A}$ is Gorenstein if and only if the Ext-algebra $H(\mathcal{E})$ is a Frobenius graded algebra.

(1)$\Rightarrow $ (4) If $H(\mathcal{E})$ is a Frobenius graded algebra,
then there is an isomorphism of graded $H(\mathcal{E})$-modules: $\Sigma^b H(\mathcal{E})\to [H(\mathcal{E})_{H(\mathcal{E})}]^*$. This implies
that $H[(\mathcal{E}_{\mathcal{E}})^*]$ has a graded free resolution $$0\to \Sigma^{b}H(\mathcal{E})\to
H[(\mathcal{E}_{\mathcal{E}})^*]\to 0.
$$
By the construction of the Eilenberg-Moore resolution, $(\mathcal{E}_{\mathcal{E}})^*$
admits a semifree resolution: $\Sigma^{b}\mathcal{E}\stackrel{\simeq}{\to}
(\mathcal{E}_{\mathcal{E}})^*$. Obviously $[\mathcal{E}_{\mathcal{E}}]^*$ is a compact DG $\mathcal{E}$-module. Similarly,
we can prove that $({}_{\mathcal{E}}\mathcal{E})^*$ is a compact DG $\mathcal{E}\!^{op}$-module.

(4)$\Rightarrow $ (5). By Corollary \ref{compact}, $({}_{\mathcal{E}}\mathcal{E})^*$  admits a minimal semifree resolution $F$ which has a finite semi-basis. Hence \begin{align*}
\dim_{\mathbbm{k}}H(R\Hom_{\mathcal{E}}(\mathbbm{k},\mathcal{E}))&=\dim_{\mathbbm{k}} H(R\Hom_{\mathcal{E}}(\mathbbm{k},\mathcal{E}^{**}))\\
&=\dim_{\mathbbm{k}} H(\Hom_{\mathbbm{k}}(({}_{\mathcal{E}}\mathcal{E})^*\otimes_{\mathcal{E}}^L\mathbbm{k},\mathbbm{k}))\\
&=\dim_{\mathbbm{k}}\Hom_{\mathbbm{k}}(F\otimes_{\mathcal{E}}\mathbbm{k},\mathbbm{k})<\infty.
\end{align*}
One can symmetrically prove $\dim_{\mathbbm{k}}H(R\Hom_{\mathcal{E}^{op}}(\mathbbm{k},\mathcal{E}))<\infty$.

(5)$\Leftrightarrow $ (6). It can be obtained by Proposition \ref{goren} and Remark \ref{replace}.
\end{proof}

\section{Calabi-Yau DG aglebras}
In this section, we assume that $\mathscr{A}$ is a homologically smooth DG algebra. By \cite[Proposition 7.3]{MW2}, $\mathscr{D}^b_{lf}(\mathscr{A})$ is a full triangulated subcategory of $\mathscr{D}^c(\mathscr{A})$.  Set $\Omega=R\Hom_{\mathscr{A}^{e}}(\mathscr{A},\mathscr{A}^e)$. For any $M\in \mathscr{D}^b_{lf}(\mathscr{A})$ and $N\in \mathscr{D}(\mathscr{A})$, we have the following canonical isomorphisms
\begin{align*}
\xymatrix{(N\otimes M^*)\otimes_{\mathscr{A}^e}^L\Omega \ar@<1ex>[r]\ar@{=}[d] &  \Hom_k(M,N)\otimes_{\mathscr{A}^e}^L\Omega \ar@<1ex>[r] &R\Hom_{\mathscr{A}^e}(A,\Hom_k(M,N))\ar@{=}[d]\\
N\otimes_{\mathscr{A}}^L\Omega\otimes_{\mathscr{A}}^LM^* & & R\Hom_{\mathscr{A}}(M,N)\\}
\end{align*}
and
\begin{align*}
\xymatrix{R\Hom_{\mathscr{A}}(M,N)^*\ar[r] &[N\otimes_{\mathscr{A}}\Omega\otimes_{\mathscr{A}}M^*]^*\ar@{=}[d]&\\
&R\Hom_{\mathscr{A}}(N\otimes_{\mathscr{A}}^L\Omega, M^{**})\ar[r]&R\Hom_{\mathscr{A}}(N\otimes_{\mathscr{A}}^L\Omega, M)\\}
\end{align*}
in $\mathscr{D}(k)$. Then we obtain a canonical isomorphism:
\begin{align}\label{isom}
\Hom_{\mathscr{D}(\mathscr{A})}(M,N)^*\stackrel{\sim}{\to}\Hom_{\mathscr{D}(\mathscr{A})}(N\otimes_{\mathscr{A}}^L\Omega, M)
\end{align}
by taking zeroth cohomology.
 By \cite[Lemma $10.1$, Lemma $10.2$]{MW2},
$\mathscr{D}^b_{lf}(\mathscr{A})$ is a $\mathbbm{k}$-linear, Hom-finite triangulated category. By the isomorphism (\ref{isom}), the category $\mathscr{D}^b_{lf}(\mathscr{A})$ admits a left Serre functor given by $-\otimes_{\mathscr{A}}^L\Omega$.

\begin{prop}\label{suff}
Let $\mathscr{A}$ be a Calabi-Yau connected cochain DG algebra then $\mathscr{D}^b_{lf}(\mathscr{A})$ is a Calabi-Yau triangulated category and the  Ext-algebra $H(\mathcal{E})$ is a symmetric Frobenius graded algebra.
\end{prop}

\begin{proof}
By \cite[Proposition 6.4]{MXYA}, the homologically smooth DG algebra $\mathscr{A}$ is Gorenstein.
  So $H(R\Hom_{\mathscr{A}}(\mathbbm{k},\mathbbm{k}))$ is a Frobenius graded algebra by Theorem \ref{mainthm}. By the definition of Calabi-Yau DG algebra, there exists some $n\in \Bbb{Z}$ such that
$\Omega\cong \Sigma^n \mathscr{A}$ in $\mathscr{D}((\mathscr{A}^{e})^{op})$. Then $\Omega \otimes_{\mathscr{A}}^L-$ is  $\Sigma^n \mathscr{A}\otimes_{\mathscr{A}}-=\Sigma^n$ when restricted to $\mathscr{D}^b_{lf}(\mathscr{A})$. Hence $\mathscr{D}^b_{lf}(\mathscr{A})$ is a Calabi-Yau triangulated category. The Serre duality condition $(\ref{isom})$ endows the Ext-algebra
$H(R\Hom_{\mathscr{A}}(\mathbbm{k},\mathbbm{k}))$ with the structure of a symmetric Frobenius graded algebra.
\end{proof}

\begin{lem}\label{Aetwist}
Let $\Omega$ be a DG $(\mathscr{A}^e)^{op}$-module such that ${}_{\mathscr{A}}M\cong \Sigma^b {}_{\mathscr{A}}\mathscr{A}$. Then there exists an automorphism $\phi\in \mathrm{Aut}_{dg}\mathscr{A}$ such that $\Sigma^{-b}M\cong \mathscr{A}(\phi)$ in $\mathscr{D}((\mathscr{A}^e)^{op})$, where $\mathscr{A}(\phi)$ is the invertible DG bimodule with generator $e$ such that $ea=\phi(a)e$ for any graded element $a\in \mathscr{A}$.
\end{lem}

\begin{proof}
By the assumption, we have $\Sigma^{-b}{}_{\mathscr{A}}M\cong {}_{\mathscr{A}}\mathcal{A}$. Let $e$ be the generator of
$\Sigma^{-b}{}_{\mathcal{A}}M$. Then for any graded element $a\in \mathcal{A}$, there exists unique $l_a\in \mathscr{A}$  such that
$ea=l_ae$.
 This induces a map $\phi:\mathscr{A}\to \mathscr{A}$ such that $\phi(a)=l_a$.
  It is straightforward to check that $\phi$ is an automorphisms of the graded algebra $\mathscr{A}^{\#}$.
Furthermore, $\phi$ is a chain map since
\begin{align*}
\partial_{\mathscr{A}}(\phi(a))e&=\partial_{\mathscr{A}}(l_a)e\\
&=\partial_{\Sigma^{-b}M}(l_ae)\\
&=\partial_{\Sigma^{-b}M}(ea)\\
&=e\cdot \partial_{\mathscr{A}}(a)\\
&=l_{\partial_{\mathscr{A}}(a)}e =\phi(\partial_{\mathscr{A}}(a))e,
\end{align*}
for any $a\in \mathscr{A}$. So $\phi\in \mathrm{Aut}_{dg}\mathscr{A}$ and $\Sigma^{-b}M\cong \mathscr{A}_{\phi}$ as a DG $(\mathscr{A}^e)^{op}$-module.

\end{proof}

\begin{prop}\label{tocy}
Let $\mathscr{A}$ be a homologically smooth and Gorenstein DG algebra. If the Ext-algebra $H(\mathcal{E})$ is a symmetric Frobenius graded algebra, then $\mathscr{A}$ is Calabi-Yau.
\end{prop}
\begin{proof}
The graded $H(\mathcal{E})^e$-module $H(\mathcal{E})^*\cong \Sigma^b H(\mathcal{E})$ since
$H(\mathcal{E})$ is a symmetric Frobenius graded algebra. On the other hand, the graded $H(\mathcal{E})^e$-module $H(\mathcal{E}^*)\cong H(\mathcal{E})^*$ since the functor $\Hom_{\mathbbm{k}}(-,\mathbbm{k})=(\quad )^*$ is exact. So we have an isomorphism of the graded $H(\mathcal{E})^e$-module
$\sigma: H(\mathcal{E}^*)\stackrel{\cong}{\to} H(\Sigma^b \mathcal{E})$.
Let \begin{align}\label{freeres1}\cdots \stackrel{d_{n+1}}{\to} F_n\stackrel{d_n}{\to} F_{n-1}\stackrel{d_{n-1}}{\to} \cdots \stackrel{d_1}{\to} F_0\stackrel{\varepsilon}{\to} H(\mathcal{E}^*)\to 0
\end{align}
be a graded free resolution of the $H(\mathcal{E})^e$-module $H(\mathcal{E}^*)$.
We can construct an Eilenberg-Moore resolution $\tau: F \stackrel{\simeq}{\to}\mathcal{E}^*$ from (\ref{freeres1}). One sees that
\begin{align}\label{freeres2}\cdots \stackrel{d_{n+1}}{\to} F_n\stackrel{d_n}{\to} F_{n-1}\stackrel{d_{n-1}}{\to} \cdots \stackrel{d_1}{\to} F_0\stackrel{\sigma\circ \varepsilon}{\to} H(\Sigma^b \mathcal{E})\to 0
\end{align}
is a graded free resolution of the $H(\mathcal{E})^e$-module $H(\Sigma^b\mathcal{E})$. From (\ref{freeres2}), one can also get a quasi-isomorphism $\eta: F\stackrel{\simeq}{\to} \Sigma^b\mathcal{E}$. Then $ \mathcal{E}^*\cong \Sigma^b\mathcal{E}$ in $\mathscr{D}(\mathcal{E})$, since we have
\begin{align*}
\xymatrix{
\mathcal{E}^* &F\ar[l]_{\simeq}\ar[r]^{\simeq} & \Sigma^b\mathcal{E}}.
\end{align*}
Note that there is a natural equivalence
\begin{align}\label{serre}
\Hom_{\mathscr{D}(\mathcal{E})}(M,-)^*\simeq \Hom_{\mathscr{D}(\mathcal{E})}(-,\mathcal{E}^*\otimes_{\mathcal{E}}^LM)
\end{align}
for any $M\in \mathscr{D}^c(\mathcal{E})$. By Theorem \ref{Gorenstein}, $\mathscr{D}^c(\mathcal{E})$ has AR Triangles. Then (\ref{serre}) implies that $\mathcal{E}^*\otimes_{\mathcal{E}}^L-\cong \Sigma^b\mathcal{E}\otimes_{\mathcal{E}}^L-\cong \Sigma^b$ is a Serre functor of $\mathscr{D}^c(\mathcal{E})$.  Therefore $\mathscr{D}^c(\mathcal{E})$ is a Calabi-Yau triangulated category. By a symmetric argument, we can show that $\mathscr{D}^c(\mathcal{E}^{op})$ is Calabi-Yau. By \cite[Lemma 8.2]{Jor3},
there are quasi-inverse equivalences of categories,
\begin{align*}
\xymatrix{&\mathscr{D}^b_{lf}(\mathscr{A})\quad\quad\ar@<1ex>[r]^{\Hom_{\mathscr{A}}(K,
-)}&\quad\quad
\mathscr{D}^{c}(\mathcal{E}^{\!op})\ar@<1ex>[l]^{-\otimes_{\mathcal{E}}^L
K}}.
\end{align*}
Hence $\mathscr{D}^b_{lf}(\mathscr{A})$ is Calabi-Yau with a Serre functor $\Sigma^b$.  Since $\mathscr{A}$ is homologically smooth and Gorenstein, the Ext-algebra $H(\mathcal{E})$ is a Frobenius graded algebra by Theorem \ref{Gorenstein}. It follows from \cite[Theorem $10.8$]{MW2} that $\mathscr{D}^b_{lf}(\mathscr{A})$ and $\mathscr{D}^b_{lf}(\mathscr{A}^{op})$ admit Auslander-Reiten triangles.
By (\ref{isom}),
 we have
a canonical isomorphism:
\begin{align*}
\Hom_{\mathscr{D}(\mathscr{A})}(M,X)^*\stackrel{\sim}{\to}\Hom_{\mathscr{D}(\mathscr{A})}(X\otimes_{\mathscr{A}}^L\Omega, M),
\end{align*}
for any objects $M, X\in \mathscr{D}^b_{lf}(\mathscr{A})$. These imply that $\mathscr{D}^b_{lf}(\mathscr{A})$ admits a Serre functor given by $-\otimes_{\mathscr{A}}^L\Omega$. Hence $-\otimes_{\mathscr{A}}^L\Omega = \Sigma^b$. By a symmetric argument, we have $\Omega\otimes_{\mathscr{A}}^L- = \Sigma^b$. We get ${}_{\mathscr{A}}\Omega\cong \Sigma^b{}_{\mathscr{A}}\mathscr{A}$ since $-\otimes_{\mathscr{A}}^L\Omega = \Sigma^b$. By Lemma \ref{Aetwist}, there exists $\phi\in \mathrm{Aut}_{dg}\mathscr{A}$ such that
$\Sigma^{-b}\Omega \cong \mathscr{A}(\phi)$ in $\mathscr{D}((\mathscr{A}^e)^{op})$. Let $\Phi: \mathscr{D}^b_{lf}(\mathscr{A})\to \mathscr{D}^b_{lf}(\mathscr{A})$ be the functor given by the operation $M\mapsto {}^{\phi}M$, which twists the action on a DG $\mathscr{A}$-module by the automorphism $\phi$. Then we have
$$\Omega\otimes_{\mathscr{A}}^L-= \Sigma^b \mathscr{A}(\phi)\otimes_{\mathscr{A}}-=\Sigma^b\circ \Phi.$$
On the other hand, $\Omega\otimes_{\mathscr{A}}^L- =\Sigma^b$.  So $\Phi\cong \mathrm{id}_{\mathscr{D}^b_{lf}(\mathscr{A})}$ and $\phi=\mathrm{id}_{\mathscr{A}}$. Hence $\Omega\cong \Sigma^b \mathscr{A}$ in $\mathscr{D}((\mathscr{A}^e)^{op})$.

\end{proof}

\begin{thm}\label{cycond}
Let $\mathcal{E}$ be the Koszul dual DG algebra of a homologically smooth and Gorenstein DG algebra
$\mathscr{A}$. Then the following conditions are equivalent.
\begin{enumerate}
\item $\mathscr{A}$ is Calabi-Yau;
\item The Ext-algebra $H(\mathcal{E})$ is a symmetric Frobenius graded algebra;
\item The triangulated categories $\mathscr{D}^c(\mathcal{E})$ and $\mathscr{D}^c(\mathcal{E}^{op})$ are Calabi-Yau;
\item The triangulated categories $\mathscr{D}^b_{lf}(\mathscr{A})$ and $\mathscr{D}^b_{lf}(\mathscr{A}^{op})$ are Calabi-Yau.
\end{enumerate}
\end{thm}
\begin{proof}

(1)$\Rightarrow$  (2) and (1)$\Rightarrow$ (4) can be obtained by Proposition \ref{suff}.

(2)$\Rightarrow$ (1). It is obtained by Proposition \ref{tocy}.

(3)$\Leftrightarrow$ (4)
By \cite[Lemma 8.2]{Jor3},
there are quasi-inverse equivalences of categories,
\begin{align*}
\xymatrix{&\mathscr{D}^b_{lf}(\mathscr{A})\quad\quad\ar@<1ex>[r]^{\Hom_{\mathscr{A}}(K,
-)}&\quad\quad
\mathscr{D}^{c}(\mathcal{E}^{\!op})\ar@<1ex>[l]^{-\otimes_{\mathcal{E}}^L
K}}.
\end{align*}
Symmetrically, $\mathscr{D}^c(\mathcal{E})$ and $\mathscr{D}^b_{lf}(\mathscr{A}^{op})$ are also equivalent.

(3)$\Rightarrow $ (2). Since $\mathscr{A}$ is homologically smooth and Gorenstein, its Ext-algebra $H(\mathcal{E})$ is a Frobenius graded algebra by Theorem \ref{Gorenstein}. By definition, there is an automorphism $\sigma$ of the graded algebra $H(\mathcal{E})$ such that $[H(\mathcal{E})]^*\cong \Sigma^b H(\mathcal{E})_{\sigma}$ as graded $H(\mathcal{E})^e$-modules.
Since $\mathscr{D}^c(\mathcal{E})$ and $\mathscr{D}^c(\mathcal{E}^{op})$ are Calabi-Yau, they admit Serre functors and hence Auslander-Reiten triangles. By Proposition \ref{Gorenstein}, $({}_{\mathcal{E}}\mathcal{E})^*\cong \Sigma^b \mathcal{E}_{\mathcal{E}}$ in $\mathscr{D}(\mathcal{E}^{op})$ and $(\mathcal{E}_{\mathcal{E}})^*\cong \Sigma^b {}_{\mathcal{E}}\mathcal{E}$ in $\mathscr{D}(\mathcal{E})$. By a similar proof as that of Lemma \ref{Aetwist}, we get an automorphism $\phi$ of  $\mathcal{E}$ such that $\Sigma^{-b}\mathcal{E}^*\cong {}_{\mathcal{E}}\mathcal{E}_{\phi}$ in $\mathscr{D}(\mathcal{E}^e)$. One sees easily that $H(\phi)=\sigma$. Let $Q$ be an indecomposable object of $\mathscr{D}^c(\mathcal{E})$. Then there is a natural equivalence
\begin{align}\label{serrefunctor}
\Hom_{\mathscr{D}(\mathcal{E})}(Q,-)^*\simeq \Hom_{\mathscr{D}(\mathcal{E})}(-,\mathcal{E}^*\otimes_{\mathcal{E}}^LQ).
\end{align}
It indicates that the functor $\mathcal{E}^*\otimes_{\mathcal{E}}^L-$ is a Serre functor of $\mathscr{D}^c(\mathcal{E})$. Since $\mathscr{D}^c(\mathcal{E})$ is a Calabi-Yau triangulated category, we have $$\Sigma^b=\mathcal{E}^*\otimes_{\mathcal{E}}^L-=\Sigma^{b}{}_{\mathcal{E}}\mathcal{E}_{\phi}\otimes_{\mathcal{E}}-\simeq \Sigma^b\circ \Phi,$$
where
$\Phi: \mathscr{D}^c(\mathcal{E})\to \mathscr{D}^c(\mathcal{E})$ is the functor given by the operation $M\mapsto {}_{\phi^{-1}}M$.
It implies that $\phi$ is an inner automorphism. Then $\Sigma^{-b}\mathcal{E}^*\cong \mathcal{E}$ in $\mathscr{D}(\mathcal{E}^e)$. So
$\Sigma^{-b}H(\mathcal{E})^*\cong H(\mathcal{E})$ as a graded $H(\mathcal{E})^e$-module.
\end{proof}
By Theorem \ref{cycond} and Proposition \ref{mainthm}, we can easily get the following corollary.
\begin{cor}
A connected cochain DG algebra $\mathscr{A}$ is Calabi-Yau if and only if its $\mathrm{Ext}$-algebra $H(\mathcal{E})$ is a symmetric Frobenius graded algebra.
\end{cor}

\section{applications on some examples}\label{appl}
In this section, we give some applications of Theorem \ref{cycond} and Theorem \ref{Gorenstein} to detect Calabi-Yauness and Gorensteinness of
a homologically smooth DG algebra.

In general, the cohomology graded algebra $H(\mathscr{A})$ of a cochain DG algebra $\mathscr{A}$ usually contains some homological information.
One sees that $\mathscr{A}$ is a Calabi-Yau DG algebra if the trivial DG algebra $(H(\mathscr{A}),0)$ is Calabi-Yau by \cite{MYY}, and it is proved in \cite{MH} that
 a connected cochain DG algebra $\mathscr{A}$ is a Koszul Calabi-Yau DG algebra if $H(\mathscr{A})$ belongs to one of the following cases:
\begin{align*}
& (a) H(\mathscr{A})\cong \mathbbm{k};  \quad \quad (b) H(\mathscr{A})= \mathbbm{k}[\lceil z\rceil], z\in \mathrm{ker}(\partial_{\mathscr{A}}^1); \\
& (c) H(\mathscr{A})= \frac{\mathbbm{k}\langle \lceil z_1\rceil, \lceil z_2\rceil\rangle}{(\lceil z_1\rceil\lceil z_2\rceil +\lceil z_2\rceil \lceil z_1\rceil)}, z_1,z_2\in \mathrm{ker}(\partial_{\mathscr{A}}^1).
\end{align*}
By \cite[Therorem B]{MH}, a connected cochain DG algebra $\mathscr{A}$ is not Calabi-Yau  but Koszul, homologically smooth and Gorenstein, if $$H(\mathscr{A})=\mathbbm{k}[\lceil z_1\rceil, \lceil z_2\rceil], \,\, \text{for some }\,\, z_1, z_2\in Z^1(\mathscr{A}).$$
It is natural for one to study various homological properties of a connected cochain DG algebra if its cohomology graded algebra belongs to other special types. In this paper, we choose cubic Artin-Schelter algebras of type A to study.
\begin{prop}\label{noncykosz}
Let $\mathscr{A}$ be a connected cochain DG algebra such that
$$H(\mathscr{A})=\mathbbm{k}\frac{k\langle \lceil x\rceil,\lceil y\rceil\rangle}{(f_1,f_2)}, x, y\in Z^1(\mathscr{A}),$$
 \begin{align*}
f_1=a\lceil x\rceil \lceil y\rceil^2+b\lceil y\rceil \lceil x\rceil \lceil y\rceil+a\lceil y\rceil^2\lceil x\rceil+c\lceil x\rceil^3,\\
f_2=a\lceil y\rceil \lceil x\rceil^2+b\lceil x\rceil \lceil y\rceil \lceil x\rceil+a\lceil x\rceil^2\lceil y\rceil+c\lceil y\rceil^3,
\end{align*}
where $(a:b:c)\in \Bbb{P}_k^2-\mathfrak{D}$ and $\mathfrak{D}:=\{(0:0:1), (0:1:0)\}\sqcup\{(a:b:c)|a^2=b^2=c^2\}.$ Then $\mathscr{A}$ is a homologically smooth and Gorenstein DG algebra. However, it is neither Koszul nor Calabi-Yau.
\end{prop}

\begin{proof}
Note that $H(\mathscr{A})$ is a cubic Artin-Schelter algebras of type A. The graded simple left $H(\mathscr{A})$-module $\mathbbm{k}$ admits a minimal graded free resolution:
$$0\to H(\mathscr{A})e_{\omega}\stackrel{d_3}{\to}  H(\mathscr{A})e_{r_1}\oplus H(\mathscr{A})e_{r_2}\stackrel{d_2}{\to} H(\mathscr{A})e_{x}\oplus H(\mathscr{A})e_{y}\stackrel{d_1}{\to}H(\mathscr{A})\stackrel{\varepsilon}{\to} \mathbbm{k}\to 0,$$
where $d_1, d_2$ and $d_3$ are defined by $d_1(e_{x})=\lceil x\rceil, d_1(e_{y})=\lceil y\rceil$,
\begin{align*}
d_2(e_{r_1})&=(a\lceil x\rceil \lceil y\rceil + b \lceil y\rceil \lceil x\rceil )e_{y} + (a\lceil y\rceil^2+c\lceil x\rceil^2)e_x\\
d_2(e_{r_2})&=(a\lceil y\rceil \lceil x\rceil + b\lceil x\rceil \lceil y\rceil)e_x+(a\lceil x\rceil^2+c\lceil y\rceil^2)e_y \\
d_3(e_{\omega})&=\lceil x\rceil e_{r_1}+\lceil y\rceil e_{r_2}.
\end{align*}
Applying the construction procedure of Eilenberg-Moore resolution in \cite{FHT2}, we can construct a semi-free resolution
$F$ of ${}_{\mathscr{A}}\mathbbm{k}$ such that
$$F^{\#}=\mathscr{A}^{\#}\oplus \mathscr{A}^{\#}\otimes (\mathbbm{k}\Sigma e_x\oplus \mathbbm{k}\Sigma e_y)\oplus  \mathscr{A}^{\#}\otimes (\mathbbm{k}\Sigma^2 e_{r_1}\oplus \mathbbm{k}\Sigma^2 e_{r_2})\oplus \mathscr{A}^{\#}\Sigma^3 e_{\omega}$$
where $\partial_F$ is defined by
\begin{align*}
&\partial_F(\Sigma e_x)=x, \partial_F(\Sigma e_y)=y\\
&\partial_F(\Sigma^2 e_{r_1})=(axy+byx)\Sigma e_y+(ay^2+cx^2)\Sigma e_x\\
&\partial_F(\Sigma^2 e_{r_2})=(ayx+bxy)\Sigma e_x+(ax^2+cy^2)\Sigma e_y\\
&\partial_F(\Sigma^3 e_{\omega})=x\Sigma^2 e_{r_1}+y\Sigma^2 e_{r_2}.
\end{align*}
Note that $F$ is minimal and admits a semi-basis $\{1,\Sigma e_{x},\Sigma e_{y},\Sigma^2 e_{r_1},\Sigma^2 e_{r_2},\Sigma^3 e_{\omega}\}$ concentrated in degrees $0$ and $1$ . Hence $\mathscr{A}$ is a non-Koszul and homologically smooth DG algebra.
Since $H(\mathscr{A})$ is an AS-regular algebra of dimension $3$, the DG algebra $\mathscr{A}$ is Gorenstein by
Lemma \ref{Goren}.
By the minimality of $F$, we have \begin{align*}
 H(R\Hom_{\mathscr{A}}(\mathbbm{k},\mathbbm{k}))&\cong H(\Hom_{\mathscr{A}}(F,\mathbbm{k}))=\Hom_{\mathscr{A}}(F,\mathbbm{k})\\
 &=\mathbbm{k}1^*\oplus \mathbbm{k}(\Sigma e_x)^*\oplus \mathbbm{k}(\Sigma e_y)^*\oplus \mathbbm{k}(\Sigma^2 e_{r_1})^*\oplus \mathbbm{k}(\Sigma^2e_{r_2})^*\oplus  \mathbbm{k}(\Sigma^3 e_{\omega})^* .
 \end{align*}
 It is concentrated in degrees $-1$ and $0$.
Since $$\Hom_{\mathscr{A}}(F,F)^{\#}\cong [\mathbbm{k}1^*\oplus \mathbbm{k}(\Sigma e_x)^*\oplus \mathbbm{k}(\Sigma e_y)^*\oplus \mathbbm{k}(\Sigma^2 e_{r_1})^*\oplus \mathbbm{k}(\Sigma^2 e_{r_2})^*\oplus \mathbbm{k}(\Sigma^3 e_{\omega})^*]\otimes F^{\#}$$ is concentrated in degree $\ge -1$, we have $H^{-1}(\Hom_{\mathscr{A}}(F,F))=Z^{-1}(\Hom_{\mathscr{A}}(F,F))$.
 For any $f\in \Hom_{\mathscr{A}}(F,F)^{-1}$, it is uniquely determined by a matrix
 $$A_f=\left(
     \begin{array}{cccccc}
       0 & 0 & 0 &  0& 0 & 0 \\
       0 & 0 & 0 & 0 & 0 & 0 \\
       0 & 0 & 0 & 0 & 0 & 0 \\
       a_{11} & a_{12} & a_{13} & 0 & 0 & 0 \\
       a_{21} & a_{22} & a_{23} & 0 & 0 & 0 \\
       a_{31} & a_{32} & a_{33} & 0 & 0 & 0 \\
     \end{array}
   \right)\in M_{6}(\mathbbm{k})$$ such that
   $$\left(
       \begin{array}{c}
         f(1) \\
         f(\Sigma e_x) \\
         f(\Sigma e_y) \\
         f(\Sigma^2 e_{r_1}) \\
         f(\Sigma^2 e_{r_2}) \\
         f(\Sigma^3 e_{\omega}) \\
       \end{array}
     \right)=A_f\left(
                  \begin{array}{c}
                    1 \\
                    \Sigma e_x \\
                    \Sigma e_y \\
                    \Sigma^2 e_{r_1} \\
                    \Sigma^2 e_{r_2} \\
                    \Sigma^3 e_{\omega} \\
                  \end{array}
                \right).$$

 Since
\begin{align*}
&\quad \quad \begin{cases}
\partial_F\circ f(1)=0\\
\partial_F\circ f(\Sigma e_x)=0\\
\partial_F\circ f(\Sigma e_y)=0\\
\partial_F\circ f(\Sigma^2 e_{r_1})=a_{12}x+a_{13}y\\
\partial_F\circ f(\Sigma^2 e_{r_2})=a_{22}x+a_{23}y\\
\partial_F\circ f(\Sigma^3 e_{\omega})=a_{32}x+a_{33}y
\end{cases} \\
& \text{and}\quad \begin{cases}
                                                 f\circ \partial_F(1)=0  \\
                                                   f\circ \partial_F(\Sigma e_x)=0          \\
                                                    f\circ \partial_F(\Sigma e_y)=0         \\
                                                    f\circ \partial_F(\Sigma^2 e_{r_1})=0    \\
                                                     f\circ \partial_F(\Sigma^2 e_{r_2})=0\\
                                                    f\circ \partial_F(\Sigma^3 e_{\omega})=-xa_{11}-ya_{21}-(xa_{12}+ya_{22})\Sigma e_x-(xa_{13}+ya_{23})\Sigma e_y
                                                   \end{cases},
\end{align*}
 the $\mathscr{A}$-linear map $\partial_{\Hom}(f)=\partial_F\circ f+f\circ \partial_F$ corresponds to the matrix
\begin{align}\label{coboundary}
\left(
     \begin{array}{cccccc}
       0 & 0 & 0 &  0& 0 & 0 \\
       0 & 0 & 0 & 0 & 0 & 0 \\
       0 & 0 & 0 & 0 & 0 & 0 \\
       a_{12}x+a_{13}y & 0 & 0 & 0 & 0 & 0 \\
       a_{22}x+a_{23}y & 0 & 0 & 0 & 0 & 0 \\
       (a_{32}-a_{11})x+(a_{33}-a_{21})y & -xa_{12}-ya_{22} & -xa_{13}-ya_{23} & 0 & 0 & 0 \\
     \end{array}
   \right).
\end{align}
If $f\in Z^{-1}(\Hom_{\mathscr{A}}(F,F))$, then
\begin{align*}
\begin{cases}
a_{12}x+a_{13}y=0\\
a_{22}x+a_{23}y=0\\
a_{32}x+a_{33}y-xa_{11}-ya_{21}=0\\
-xa_{12}-ya_{22}=0\\
-xa_{13}-ya_{23}=0,
\end{cases}
\end{align*}
which indicate that $a_{12}=a_{13}=a_{22}=a_{23}=0$ and $a_{11}=a_{32}, a_{21}=a_{33}$.
Hence $$H^{-1}(\Hom_{\mathscr{A}}(F,F))\cong \left\{\left(
     \begin{array}{cccccc}
       0 & 0 & 0 &  0& 0 & 0 \\
       0 & 0 & 0 & 0 & 0 & 0 \\
       0 & 0 & 0 & 0 & 0 & 0 \\
       n & 0 & 0 & 0 & 0 & 0 \\
       l & 0 & 0 & 0 & 0 & 0 \\
       m & n & l & 0 & 0 & 0 \\
     \end{array}
   \right)|m,n,l\in \mathbbm{k} \right\}.$$
     Hence we have $\dim_{\mathbbm{k}} \Hom_{\mathscr{A}}(F,F)^{-1}=\dim_{\mathbbm{k}}Z^{-1}(\Hom_{\mathscr{A}}(F,F))=3$. It implies that $Z^0(\Hom_{\mathscr{A}}(F,F))=H^0(\Hom_{\mathscr{A}}(F,F))$. For any $g\in Z^0(\Hom_{\mathscr{A}}(F,F))$, it is uniquely determined by a
     matrix $$X_g= \left(
     \begin{array}{cccccc}
       b_{11} & b_{12} & b_{13} &  0& 0 & 0 \\
       b_{21} & b_{22} & b_{23} & 0 & 0 & 0 \\
       b_{31} & b_{32} & b_{33} & 0 & 0 & 0 \\
       d_{11} & d_{12} & d_{13} & c_{11} & c_{12} & c_{13} \\
       d_{21} & d_{22} & d_{23} & c_{21} & c_{22} & c_{23} \\
       d_{31} & d_{32} & d_{33} & c_{31} & c_{32} & c_{33} \\
     \end{array}
   \right)=\left(\begin{array}{cc}
   B&0\\
   D&C\\
   \end{array}
   \right)$$
   such that$$\left(
       \begin{array}{c}
         g(1) \\
         g(\Sigma e_x) \\
         g(\Sigma e_y) \\
         g(\Sigma^2 e_{r_1}) \\
         g(\Sigma^2 e_{r_2}) \\
         g(\Sigma^3 e_{\omega}) \\
       \end{array}
     \right)=X_g\left(
                  \begin{array}{c}
                    1 \\
                    \Sigma e_x \\
                    \Sigma e_y \\
                    \Sigma^2 e_{r_1} \\
                    \Sigma^2 e_{r_2} \\
                    \Sigma^3 e_{\omega} \\
                  \end{array}
                \right),$$
where $b_{ij},c_{ij}\in \mathbbm{k}$ and $d_{ij}\in \mathscr{A}^1$. For simplicity, we let
$$\left(
      \begin{array}{cccccc}
        0 & 0 & 0 & 0 & 0 & 0 \\
        x & 0 & 0 & 0 & 0 & 0 \\
        y & 0 & 0 & 0 & 0 & 0 \\
        0 & ay^2+cx^2 & axy+byx & 0 & 0 & 0 \\
        0 & ayx+bxy & ax^2+cy^2 & 0 & 0 & 0 \\
        0 & 0 & 0 & x & y & 0 \\
      \end{array}
    \right)=\left(
      \begin{array}{cc}
      Q_1& 0\\
      Q_3& Q_2\\
      \end{array}\right),$$
      where
      $$Q_1=\left(
      \begin{array}{ccc}
        0 & 0 & 0  \\
        x & 0 & 0 \\
        y & 0 & 0 \\
\end{array}
\right), Q_2=\left(
      \begin{array}{ccc}
        0 & 0 & 0  \\
        0 & 0 & 0 \\
        x & y & 0 \\
\end{array}
\right), Q_3=\left(
      \begin{array}{ccc}
        0 & ay^2+cx^2 & axy+byx  \\
        0 & ayx+bxy & ax^2+cy^2 \\
        0 & 0 & 0 \\
\end{array}
\right).$$
Since $0=\partial_{\Hom}(g)=\partial_F\circ g-g\circ \partial_F$, we have
\begin{align*}
& \left(
     \begin{array}{cc}
      0 & 0  \\
      \partial_{\mathscr{A}}(D) & 0 \\
     \end{array}
   \right)-\left(
     \begin{array}{cc}
      0 & 0  \\
      DQ_1 & 0 \\
     \end{array}
   \right)+ \left(\begin{array}{cc}
   B&0\\
   0&C\\
   \end{array}
   \right) \left(
      \begin{array}{cc}
        Q_1 & 0 \\
        Q_3 &  Q_2 \\
      \end{array}
    \right)\\
    &=\left(
      \begin{array}{cc}
        Q_1 & 0 \\
        Q_3 &  Q_2 \\
      \end{array}
    \right)\left(\begin{array}{cc}
   B&0\\
   D&C\\
   \end{array}
   \right),
    \end{align*}
    where $\partial_{\mathscr{A}}(D)=\left(
      \begin{array}{ccc}
        \partial_{\mathscr{A}}(d_{11}) &  \partial_{\mathscr{A}}(d_{12})& \partial_{\mathscr{A}}(d_{13})  \\
        \partial_{\mathscr{A}}(d_{21}) & \partial_{\mathscr{A}}(d_{22}) & \partial_{\mathscr{A}}(d_{23}) \\
        \partial_{\mathscr{A}}(d_{31}) & \partial_{\mathscr{A}}(d_{32}) & \partial_{\mathscr{A}}(d_{33}) \\
\end{array}
\right).$
Then
$$\left(
      \begin{array}{cc}
        BQ_1 & 0 \\
      \partial_{\mathscr{A}}(D)-DQ_1+ CQ_3 &  CQ_2 \\
      \end{array}
    \right)=\left(
      \begin{array}{cc}
        Q_1B & 0 \\
        Q_3B+Q_2D &  Q_2C \\
      \end{array}
    \right).$$
Therefore,
\begin{align*}
\left(
      \begin{array}{ccc}
        b_{11} & b_{12} & b_{13}  \\
        b_{21} & b_{22} & b_{23} \\
        b_{31} & b_{32} & b_{33} \\
\end{array}
\right)\left(
      \begin{array}{ccc}
        0 & 0 & 0  \\
        x & 0 & 0 \\
        y & 0 & 0 \\
\end{array}
\right)=\left(
      \begin{array}{ccc}
        0 & 0 & 0  \\
        x & 0 & 0 \\
        y & 0 & 0 \\
\end{array}
\right)\left(
      \begin{array}{ccc}
        b_{11} & b_{12} & b_{13}  \\
        b_{21} & b_{22} & b_{23} \\
        b_{31} & b_{32} & b_{33} \\
\end{array}
\right)\\
\left(
      \begin{array}{ccc}
        c_{11} & c_{12} & c_{13}  \\
        c_{21} & c_{22} & c_{23} \\
        c_{31} & c_{32} & c_{33} \\
\end{array}
\right)\left(
      \begin{array}{ccc}
        0 & 0 & 0  \\
        0 & 0 & 0 \\
        x & y & 0 \\
\end{array}
\right)=\left(
      \begin{array}{ccc}
        0 & 0 & 0  \\
        0 & 0 & 0 \\
        x & y & 0 \\
\end{array}
\right)\left(
      \begin{array}{ccc}
        c_{11} & c_{12} & c_{13}  \\
        c_{21} & c_{22} & c_{23} \\
        c_{31} & c_{32} & c_{33} \\
\end{array}
\right)
\end{align*}
and
\begin{align*}
& \partial_{\mathscr{A}}(D)-D\left(
      \begin{array}{ccc}
        0 & 0 & 0  \\
        x & 0 & 0 \\
        y & 0 & 0 \\
\end{array}
\right)+\left(
      \begin{array}{ccc}
        c_{11} & c_{12} & c_{13}  \\
        c_{21} & c_{22} & c_{23} \\
        c_{31} & c_{32} & c_{33} \\
\end{array}
\right)\left(
      \begin{array}{ccc}
        0 & ay^2+cx^2 & axy+byx  \\
        0 & ayx+bxy & ax^2+cy^2 \\
        0 & 0 & x^2 \\
\end{array}
\right)\\
&=\left(
      \begin{array}{ccc}
        0 & ay^2+cx^2 & axy+byx  \\
        0 & ayx+bxy & ax^2+cy^2 \\
        0 & 0 & 0 \\
\end{array}
\right)\left(
      \begin{array}{ccc}
        b_{11} & b_{12} & b_{13}  \\
        b_{21} & b_{22} & b_{23} \\
        b_{31} & b_{32} & b_{33} \\
\end{array}
\right)+\left(
      \begin{array}{ccc}
        0 & 0 & 0  \\
        0 & 0 & 0 \\
        x & y & 0 \\
\end{array}
\right)D.
\end{align*}
These equations imply  that
\begin{align*}
\begin{cases}
b_{12}=b_{13}=b_{23}=0\\
b_{11}=b_{22}=b_{33}=c_{11}=c_{22}=c_{33}\\
c_{12}=c_{13}=c_{23}=0\\
b_{31}+c_{32}=0\\
c_{31}+b_{21}=0\\
d_{12}=-cb_{21}x-bb_{31}y\\
d_{13}=-ab_{21}y-ab_{31}x\\
d_{22}=-ab_{21}y-ab_{31}x\\
d_{23}=-bb_{21}x-cb_{31}y\\
d_{11}=\theta x+sy \\
d_{21}=\lambda x+\xi y\\
d_{31}=ux + vy\\
d_{32}=-\theta x -\lambda y\\
d_{33}=-s x-\xi y
\end{cases}
\end{align*}
where $\theta,s,\lambda, \xi, u,v,\lambda \in \mathbbm{k}$.
So $Z^0(\Hom_{\mathscr{A}}(F,F))$ is isomorphic to
$$\left\{ \left(
     \begin{array}{cccccc}
       r & 0 & 0 &  0& 0 & 0 \\
       p & r & 0 & 0 & 0 & 0 \\
       q & 0 & r & 0 & 0 & 0 \\
       \theta x+sy & -cpx-bqy & -apy-aqx & r &  0 & 0 \\
       \lambda x+\xi y & -apy-aqx & -bpx-cqy & 0 & r &  0 \\
       ux+vy & -\theta x-\lambda y & -sx-\xi y & -p & -q & r \\
     \end{array}
   \right) | r,p,q,\theta,s,\lambda,\xi, u,v\in \mathbbm{k}\right\}.$$
For any $g\in Z^0(\Hom_{\mathscr{A}}(F,F))$ such that
$$\left(
                  \begin{array}{c}
                   g(1)\\
                    g(\Sigma e_x) \\
                    g(\Sigma e_y) \\
                    g(\Sigma^2 e_{r_1}) \\
                    g(\Sigma^2 e_{r_2}) \\
                    g(\Sigma^3 e_{\omega}) \\
                  \end{array}
                \right)= \left(
     \begin{array}{cccccc}
       0 & 0 & 0 &  0& 0 & 0 \\
       0 & 0 & 0 & 0 & 0 & 0 \\
       0 & 0 & 0 & 0 & 0 & 0 \\
       \theta x+sy & 0 & 0 & 0 &  0 & 0 \\
       \lambda x+\xi y & 0 & 0 & 0 & 0 &  0 \\
       ux+vy & -\theta x-\lambda y & -sx-\xi y & 0 & 0 & 0\\
     \end{array}
   \right) \left(
                  \begin{array}{c}
                    1 \\
                    \Sigma e_x \\
                    \Sigma e_y \\
                    \Sigma^2 e_{r_1} \\
                    \Sigma^2 e_{r_2} \\
                    \Sigma^3 e_{\omega} \\
                  \end{array}
                \right)$$
with $\theta, s, \lambda, \xi, u, v\in \mathbbm{k}$, we have  $\partial_{\Hom}(h)=g$, where $h\in \Hom_{\mathscr{A}}(F,F)^{-1}$ defined by
$$
\left(
                  \begin{array}{c}
                    h(1)\\
                    h(\Sigma e_x)\\
                    h(\Sigma e_y) \\
                    h(\Sigma^2 e_{r_1}) \\
                    h(\Sigma^2 e_{r_2}) \\
                    h(\Sigma^3 e_{\omega}) \\
                  \end{array}
                \right)=\left(
     \begin{array}{cccccc}
       0 & 0 & 0 &  0& 0 & 0 \\
       0 & 0 & 0 & 0 & 0 & 0 \\
       0 & 0 & 0 & 0 & 0 & 0 \\
       0 & \theta & s & 0 &  0 & 0 \\
       0 & \lambda & \xi & 0 & 0 &  0 \\
        0& u & v & 0 & 0 & 0\\
     \end{array}
   \right)\left(
                  \begin{array}{c}
                    1 \\
                    \Sigma e_x \\
                    \Sigma e_y \\
                    \Sigma^2 e_{r_1} \\
                    \Sigma^2 e_{r_2} \\
                    \Sigma^3 e_{\omega} \\
                  \end{array}
                \right).$$
So $H^0(\Hom_{\mathscr{A}}(F,F))$ is isomorphic to  $$ \left\{ \left(
     \begin{array}{cccccc}
       r & 0 & 0 &  0& 0 & 0 \\
       p & r & 0 & 0 & 0 & 0 \\
       q & 0 & r & 0 & 0 & 0 \\
       0 & -cpx-bqy & -apy-aqx & r &  0 & 0 \\
       0 &  -apy-aqx &-bpx-cqy  & 0 & r &  0 \\
       0 & 0 & 0 & -p & -q & r \\
     \end{array}
   \right)\quad |\quad r,p,q\in \mathbbm{k}\right\}.$$
Set $e_0=\sum\limits_{i=1}^6E_{ii}$,
\begin{align*}
 &e_1=\left(
     \begin{array}{cccccc}
       0 & 0 & 0 &  0& 0 & 0 \\
       1 & 0 & 0 & 0 & 0 & 0 \\
       0 & 0 & 0 & 0 & 0 & 0 \\
       0 & -cx & -ay & 0 &  0 & 0 \\
       0 & -ay & -bx  & 0 & 0 &  0 \\
       0 & 0 & 0 & -1 & 0 & 0 \\
     \end{array}
   \right),\,\, e_2=\left(
     \begin{array}{cccccc}
       0 & 0 & 0 &  0& 0 & 0 \\
       0 & 0 & 0 & 0 & 0 & 0 \\
       1 & 0 & 0 & 0 & 0 & 0 \\
       0 & -by & -ax & 0 &  0 & 0 \\
       0 & -ax & -cy  & 0 & 0 &  0 \\
       0 & 0 & 0 & 0 & -1 & 0 \\
     \end{array}
   \right),\\
 & e_3=\left(
     \begin{array}{cccccc}
       0 & 0 & 0 &  0& 0 & 0 \\
       0 & 0 & 0 & 0 & 0 & 0 \\
       0 & 0 & 0 & 0 & 0 & 0 \\
       0 & 0 & 0 & 0 & 0 & 0 \\
       1 & 0 & 0 & 0 & 0 &  0 \\
       0 & 0 & 1 & 0 & 0 & 0 \\
     \end{array}
   \right), \,\, e_4=\left(
     \begin{array}{cccccc}
       0 & 0 & 0 &  0& 0 & 0 \\
       0 & 0 & 0 & 0 & 0 & 0 \\
       0 & 0 & 0 & 0 & 0 & 0 \\
       1 & 0 & 0 & 0 &  0 & 0 \\
       0 & 0 & 0 & 0 & 0 &  0 \\
       0 & 1 & 0 &  0 & 0 & 0 \\
     \end{array}
   \right), \,\, e_5=E_{61}.
\end{align*}
From the computations above, the Ext-algebra
$$H(\Hom_{\mathscr{A}}(F,F))\cong \bigoplus\limits_{i=0}^5\mathbbm{k}e_i, \,\, |e_0|=|e_1|=|e_2|=0, \,\, |e_3|=|e_4|=|e_5|=-1.$$
as a graded $\mathbbm{k}$-vector space.
 Note that
\begin{align*}
&e_1^2=\left(
     \begin{array}{cccccc}
       0 & 0 & 0 &  0& 0 & 0 \\
       0 & 0 & 0 & 0 & 0 & 0 \\
       0 & 0 & 0 & 0 & 0 & 0 \\
       -cx & 0 & 0 & 0 &  0 & 0 \\
       -ay & 0 & 0  & 0 & 0 &  0 \\
       0 & cx & ay & 0 & 0 & 0 \\
     \end{array}
   \right),\,\, e_2^2=\left(
     \begin{array}{cccccc}
       0 & 0 & 0 &  0& 0 & 0 \\
       0 & 0 & 0 & 0 & 0 & 0 \\
       0 & 0 & 0 & 0 & 0 & 0 \\
       -ax & 0 & 0 & 0 &  0 & 0 \\
       -cy & 0 & 0  & 0 & 0 &  0 \\
       0 & ax & cy & 0 & 0 & 0 \\
     \end{array}
   \right),\\
&e_1e_2= \left(
     \begin{array}{cccccc}
       0 & 0 & 0 &  0& 0 & 0 \\
       0 & 0 & 0 & 0 & 0 & 0 \\
       0 & 0 & 0 & 0 & 0 & 0 \\
       -ay & 0 & 0 & 0 &  0 & 0 \\
       -bx & 0 & 0  & 0 & 0 &  0 \\
       0 & by & ax & 0 & 0 & 0 \\
     \end{array}
   \right),\,\, e_2e_1= \left(
     \begin{array}{cccccc}
       0 & 0 & 0 &  0& 0 & 0 \\
       0 & 0 & 0 & 0 & 0 & 0 \\
       0 & 0 & 0 & 0 & 0 & 0 \\
       -by& 0 & 0 & 0 &  0 & 0 \\
       -ax & 0 & 0  & 0 & 0 &  0 \\
       0 & ay & bx & 0 & 0 & 0 \\
     \end{array}
   \right).
\end{align*}
By (\ref{coboundary}),  the morphisms of DG $\mathscr{A}$-modules corresponding to them all belong to $B^0(\Hom_{\mathscr{A}}(F,F))$.
Thus $e_1^2=0,e_2^2=0,e_1e_2=0,e_2e_1=0$ in $H(\Hom_{\mathscr{A}}(F,F))$.
The multiplication structure of $H(\Hom_{\mathscr{A}}(F,F))$ is determined by the following tabular:\\
\begin{center} \begin{tabular}{l|llllll}
$\cdot$ & $e_0$  & $e_1$ &  $e_2$ & $e_3$ & $e_4$ & $e_5$\\
 \hline
$e_0$    & $e_0$   & $e_1$ & $e_2$  & $e_3$ & $e_4$ & $e_5$ \\
$e_1$  & $e_1$ & $0$ & $0$    &$0$  & $-e_5$ & $0$ \\
$e_2$  & $e_2$ & $0$   & $0$    &$-e_5$  & $0$   & $0$\\
$e_3$  &$e_3$  & $0$ & $e_5$  &$0$    & $0$   & $0$\\
$e_4$  &$e_4$  & $e_5$ & $0$    &$0$    & $0$   & $0$\\
$e_5$  & $e_5$ & $0$   & $0$    &$0$    & $0$   & $0$
 \\
\end{tabular}.
\end{center}
For simplicity, we set $\mathcal{E}=H(\Hom_{\mathscr{A}}(F,F))$. We want to show that $\mathcal{E}$ is not a symmetric Frobenius algebra.
If $\mathcal{E}$ is a symmetric Frobenius algebra. Then there exists an isomorphism $\sigma: \mathcal{E}\to \Sigma \Hom_k(\mathcal{E},k)$ of left $\mathcal{E}^e$-modules.
There is a matrix $\Omega=\left(
                       \begin{array}{cc}
                         0_{3\times 3} & X \\
                         Y &  0_{3\times 3}  \\
                       \end{array}
                     \right)
\in \mathrm{GL}_6(k)$ such that
\begin{align*}
\left(  \begin{array}{c}
                                                    \sigma (e_0)   \\
                                                   \sigma(e_1)               \\
                                                   \sigma(e_2)                  \\
                                                   \sigma(e_3)    \\
                                                   \sigma(e_4)    \\
                                                   \sigma(e_5)
                                                   \end{array}\right)=\Omega \left(   \begin{array}{c}
                                                    \Sigma e_0^*   \\
                                                    \Sigma e_1^*               \\
                                                   \Sigma e_2^*                  \\
                                                   \Sigma e_3^*    \\
                                                   \Sigma e_4^*    \\
                                                   \Sigma e_5^*
                                                   \end{array}\right)
\end{align*}
where $$X=\left(
            \begin{array}{ccc}
              X_{11} & X_{12} & X_{13} \\
              X_{21} & X_{22} & X_{23} \\
              X_{31} & X_{32} & X_{33} \\
            \end{array}
          \right), Y=\left(
            \begin{array}{ccc}
              Y_{11} & Y_{12} & Y_{13} \\
              Y_{21} & Y_{22} & Y_{23} \\
              Y_{31} & Y_{32} & Y_{33} \\
            \end{array}
          \right).$$
Since $\sigma$ is a morphism of $\mathcal{E}^e$-modules, we have
$\sigma(e_5)=\sigma(e_4e_1)=e_4\sigma(e_1)=\sigma(e_4)e_1$ and $\sigma(e_5)=\sigma(e_3e_2)=\sigma(e_3)e_2=e_3\sigma(e_2)$.
On the other hand,
\begin{align*}
&\sigma(e_4)e_1=[Y_{21}\Sigma e_0^*+Y_{22}\Sigma e_1^* + Y_{23} \Sigma e_2^*]e_1:\quad  \left\{   \begin{array}{c}
                                                   \Sigma e_0   \\
                                                   \Sigma e_1               \\
                                                   \Sigma e_2                 \\
                                                   \Sigma e_3    \\
                                                   \Sigma e_4    \\
                                                   \Sigma e_5
                                                   \end{array}\right \}  \begin{array}{c}
                                                   \longrightarrow\\
                                                   \longrightarrow\\
                                                   \longrightarrow\\
                                                   \longrightarrow\\
                                                   \longrightarrow\\
                                                   \longrightarrow\\
                                                   \end{array}
                                                    \left\{   \begin{array}{c}
                                                     Y_{22}\Sigma 1_{\mathbbm{k}} \\
                                                     0 \,\,      \\
                                                     0 \,\,  \\
                                                     0 \,\, \\
                                                     0 \,\, \\
                                                     0 \,\, \\
                                                   \end{array}\right \},\\
 &e_4\sigma(e_1)=e_4[X_{21}\Sigma e_3^*+X_{22}\Sigma e_4^*+X_{23}\Sigma e_5^*]:\quad
 \left\{   \begin{array}{c}
                                                   \Sigma e_0   \\
                                                   \Sigma e_1               \\
                                                   \Sigma e_2                 \\
                                                   \Sigma e_3    \\
                                                   \Sigma e_4    \\
                                                   \Sigma e_5
                                                   \end{array}\right \}  \begin{array}{c}
                                                    \longrightarrow\\
                                                    \longrightarrow\\
                                                   \longrightarrow\\
                                                   \longrightarrow\\
                                                   \longrightarrow\\
                                                   \longrightarrow\\
                                                   \end{array}
                                                    \left\{   \begin{array}{c}
                                                     -X_{22}\Sigma 1_{\mathbbm{k}} \,\, \\
                                                     X_{23}\Sigma 1_{\mathbbm{k}}\,\,      \\
                                                     0 \,\,  \\
                                                     0 \,\, \\
                                                      0 \,\, \\
                                                       0 \,\, \\
                                                   \end{array}\right \},\\
&\sigma(e_3)e_2=[Y_{11}\Sigma e_0^*+Y_{12}\Sigma e_1^* + Y_{13} \Sigma e_2^*]e_2: \quad \left\{   \begin{array}{c}
                                                   \Sigma e_0   \\
                                                    \Sigma e_1               \\
                                                   \Sigma e_2                 \\
                                                    \Sigma e_3    \\
                                                    \Sigma e_4    \\
                                                    \Sigma e_5
                                                   \end{array}\right \}  \begin{array}{c}
                                                          \longrightarrow\\
                                                                 \longrightarrow\\
                                                   \longrightarrow\\
                                                   \longrightarrow\\
                                                   \longrightarrow\\
                                                   \longrightarrow\\
                                                   \end{array}
                                                    \left\{   \begin{array}{c}
                                                     Y_{13}\Sigma 1_{\mathbbm{k}} \\
                                                     0 \,\,      \\
                                                     0 \,\,  \\
                                                     0 \,\, \\
                                                     0 \,\, \\
                                                     0\,\, \\
                                                   \end{array}\right \},\\
&e_3\sigma(e_2)=e_3[X_{31}\Sigma e_3^*+X_{32} \Sigma e_4^*+X_{33} \Sigma e_5^*]:\quad \left\{   \begin{array}{c}
                                                    \Sigma e_0   \\
                                                    \Sigma e_1               \\
                                                    \Sigma e_2                 \\
                                                    \Sigma e_3    \\
                                                    \Sigma e_4    \\
                                                    \Sigma e_5
                                                   \end{array}\right \}  \begin{array}{c}
                                                   \longrightarrow\\
                                                    \longrightarrow\\
                                                   \longrightarrow\\
                                                   \longrightarrow\\
                                                   \longrightarrow\\
                                                   \longrightarrow\\
                                                   \end{array}
                                                    \left\{   \begin{array}{c}
                                                     -X_{31}\Sigma 1_{\mathbbm{k}} \\
                                                     0 \,\,      \\
                                                     X_{33}\Sigma 1_{\mathbbm{k}} \\
                                                     0 \,\, \\
                                                     0 \,\,  \\
                                                     0 \,\, \\
                                                   \end{array}\right \},\\
 \end{align*}
 and
 \begin{align*}
  &\sigma(e_5)=Y_{31}\Sigma e_{0}^*+Y_{32}\Sigma e_1^* + Y_{33}\Sigma e_2^*:\quad
  \left\{   \begin{array}{c}
                                                    \Sigma e_0   \\
                                                    \Sigma e_1                \\
                                                    \Sigma e_2           \\
                                                    \Sigma e_3   \\
                                                    \Sigma e_4  \\
                                                    \Sigma e_5
                                                   \end{array}\right \}  \begin{array}{c}
                                                   \longrightarrow\\
                                                   \longrightarrow\\
                                                   \longrightarrow\\
                                                   \longrightarrow\\
                                                    \longrightarrow\\
                                                   \longrightarrow\\
                                                   \end{array}
                                                    \left\{   \begin{array}{c}
                                                     Y_{31}\Sigma 1_{\mathbbm{k}} \,\, \\
                                                     Y_{32}\Sigma 1_{\mathbbm{k}} \,\,      \\
                                                     Y_{33}\Sigma 1_{\mathbbm{k}} \,\,  \\
                                                      0 \,\, \\
                                                      0 \,\,  \\
                                                      0 \,\, \\
                                                   \end{array}\right \}.
 \end{align*}
 Since $\mathrm{char}\mathbbm{k}=0$, we have
 \begin{align}\label{eqsone}
 \begin{cases}
 X_{23}=X_{33}=Y_{32}=Y_{33}=0\\
 Y_{22}=-X_{22}=Y_{13}=-X_{31}=Y_{31}.
 \end{cases}
 \end{align}
 On the other hand,
  \begin{align*}
& \sigma(e_1)e_4=[X_{21}\Sigma e_3^*+ X_{22}\Sigma e_4^*+X_{23}\Sigma e_5^*]e_4: \quad \left\{   \begin{array}{c}
                                                   \Sigma e_0   \\
                                                   \Sigma e_1               \\
                                                   \Sigma e_2                 \\
                                                   \Sigma e_3    \\
                                                   \Sigma e_4    \\
                                                   \Sigma e_5
                                                   \end{array}\right \}  \begin{array}{c}
                                                   \longrightarrow\\
                                                   \longrightarrow\\
                                                   \longrightarrow\\
                                                   \longrightarrow\\
                                                   \longrightarrow\\
                                                   \longrightarrow\\
                                                   \end{array}
                                                    \left\{   \begin{array}{c}
                                                     -X_{22}\Sigma 1_{\mathbbm{k}} \\
                                                     -X_{23}\Sigma 1_{\mathbbm{k}}    \\
                                                     0 \,\,  \\
                                                     0 \,\, \\
                                                     0 \,\, \\
                                                     0 \,\, \\
                                                   \end{array}\right \},\\
& e_1\sigma(e_4)=e_1[Y_{21}\Sigma e_0^*+Y_{22}\Sigma e_1^*  + Y_{23} \Sigma e_2^*]:\quad
 \left\{   \begin{array}{c}
                                                   \Sigma e_0   \\
                                                   \Sigma e_1               \\
                                                   \Sigma e_2                 \\
                                                   \Sigma e_3    \\
                                                   \Sigma e_4    \\
                                                   \Sigma e_5
                                                   \end{array}\right \}  \begin{array}{c}
                                                    \longrightarrow\\
                                                    \longrightarrow\\
                                                   \longrightarrow\\
                                                   \longrightarrow\\
                                                   \longrightarrow\\
                                                   \longrightarrow\\
                                                   \end{array}
                                                    \left\{   \begin{array}{c}
                                                     Y_{22}\Sigma 1_{\mathbbm{k}}\\
                                                     0 \,\,      \\
                                                     0 \,\,  \\
                                                     0 \,\, \\
                                                     0 \,\, \\
                                                     0 \,\, \\
                                                   \end{array}\right \},\\
&\sigma(e_2)e_3=[X_{31}\Sigma e_3^*+X_{32}\Sigma e_4^* + X_{33} \Sigma e_5^*]e_3:\quad \left\{   \begin{array}{c}
                                                   \Sigma e_0   \\
                                                    \Sigma e_1               \\
                                                   \Sigma e_2                 \\
                                                    \Sigma e_3    \\
                                                    \Sigma e_4    \\
                                                    \Sigma e_5
                                                   \end{array}\right \}  \begin{array}{c}
                                                          \longrightarrow\\
                                                                 \longrightarrow\\
                                                   \longrightarrow\\
                                                   \longrightarrow\\
                                                   \longrightarrow\\
                                                   \longrightarrow\\
                                                   \end{array}
                                                    \left\{   \begin{array}{c}
                                                     -X_{31}\Sigma 1_{\mathbbm{k}}  \\
                                                     0 \,\,      \\
                                                     -X_{33}\Sigma 1_{\mathbbm{k}} \\
                                                     0 \,\, \\
                                                     0 \,\, \\
                                                     0\,\, \\
                                                   \end{array}\right \},\\
&e_2\sigma(e_3)=e_2[Y_{11}\Sigma e_0^*+Y_{12}\Sigma e_1^* + Y_{13} \Sigma e_2^*]:\quad \left\{   \begin{array}{c}
                                                    \Sigma e_0   \\
                                                    \Sigma e_1               \\
                                                    \Sigma e_2                 \\
                                                    \Sigma e_3    \\
                                                    \Sigma e_4    \\
                                                    \Sigma e_5
                                                   \end{array}\right \}  \begin{array}{c}
                                                   \longrightarrow\\
                                                    \longrightarrow\\
                                                   \longrightarrow\\
                                                   \longrightarrow\\
                                                   \longrightarrow\\
                                                   \longrightarrow\\
                                                   \end{array}
                                                    \left\{   \begin{array}{c}
                                                     Y_{13}\Sigma 1_{\mathbbm{k}} \\
                                                     0 \,\,      \\
                                                     0 \,\,  \\
                                                     0 \,\, \\
                                                     0 \,\,  \\
                                                     0 \,\, \\
                                                   \end{array}\right \},\\
 \end{align*}
and
\begin{align*}
&\sigma(e_2)e_3=e_2\sigma(e_3)=-\sigma (e_5)=e_1\sigma(e_4)=\sigma(e_1)e_4: \left\{   \begin{array}{c}
                                                    \Sigma e_0   \\
                                                    \Sigma e_1               \\
                                                    \Sigma e_2                 \\
                                                    \Sigma e_3    \\
                                                    \Sigma e_4    \\
                                                    \Sigma e_5
                                                   \end{array}\right \}  \begin{array}{c}
                                                   \longrightarrow\\
                                                    \longrightarrow\\
                                                   \longrightarrow\\
                                                   \longrightarrow\\
                                                   \longrightarrow\\
                                                   \longrightarrow\\
                                                   \end{array}
                                                    \left\{   \begin{array}{c}
                                                     -Y_{31}\Sigma 1_{\mathbbm{k}}  \\
                                                     -Y_{32}\Sigma 1_{\mathbbm{k}}    \\
                                                     -Y_{33}\Sigma 1_{\mathbbm{k}}  \\
                                                     0 \,\, \\
                                                     0 \,\,  \\
                                                     0 \,\, \\
                                                   \end{array}\right \}
\end{align*}
we have \begin{align}\label{eqstwo}
\begin{cases}-X_{22}=Y_{22}=-X_{31}=Y_{13}=-Y_{31},\\
 Y_{32}=Y_{33}=X_{33}=X_{23}=0
 \end{cases}
 \end{align}
By (\ref{eqsone}) and (\ref{eqstwo}), we have
\begin{align*}
\begin{cases}
 Y_{22}=-X_{22}=Y_{13}=-X_{31}=Y_{31}\\
-X_{22}=Y_{22}=-X_{31}=Y_{13}=-Y_{31}\\
Y_{32}=Y_{33}=X_{33}=X_{23}=0.
\end{cases}
\end{align*}
Since $\mathrm{char}\,\mathbbm{k}=0$, we get $$Y_{31}=Y_{32}=Y_{33}=Y_{22}=X_{22}=Y_{13}=X_{31}=X_{33}=X_{23}=0.$$
Then $\Omega\not\in \mathrm{GL}_6(\mathbbm{k})$.
One sees that $\sigma $ is not an isomorphism. Then we reach a contradiction. Hence $\mathcal{E}$ is not a symmetric Frobenius algebra.
Finally, we must emphasize that $\mathcal{E}$ is a graded Frobenius algebra. Indeed, there is an isomorphism $$\eta: \mathcal{E}\to \Sigma \Hom_{\mathbbm{k}}(\mathcal{E},\mathbbm{k})$$ of graded left $\mathcal{E}$-modules such that
\begin{align*}
\left(  \begin{array}{c}
                                                    \eta(e_0)   \\
                                                   \eta(e_1)               \\
                                                   \eta(e_2)                  \\
                                                   \eta(e_3)    \\
                                                   \eta(e_4)    \\
                                                   \eta(e_5)
                                                   \end{array}\right)=\left(
                                                                        \begin{array}{cccccc}
                                                                          0 & 0 & 0 & 1 & -1 & 1 \\
                                                                          0 & 0 & 0 & 0 & 1 & 0 \\
                                                                          0 & 0 & 0 & 1 & 0 & 0 \\
                                                                          -1 & 0 & 1 & 0 & 0 & 0 \\
                                                                          1 & 1 & 0 & 0 & 0 & 0 \\
                                                                          -1 & 0 & 0 & 0 & 0 & 0 \\
                                                                        \end{array}
                                                                      \right)
                                                    \left(   \begin{array}{c}
                                                    \Sigma e_0^*   \\
                                                    \Sigma e_1^*               \\
                                                   \Sigma e_2^*                  \\
                                                   \Sigma e_3^*    \\
                                                   \Sigma e_4^*    \\
                                                   \Sigma e_5^*
                                                   \end{array}\right).
\end{align*}
Hence $\mathscr{A}$ is not Calabi-Yau but Gorenstein by Theorem \ref{Gorenstein} and Theorem \ref{cycond}.
\end{proof}

 In \cite[Proposition $6.9$]{MHLX}, one can see a proof for the down-up DG algebra $\mathscr{A}$ in the next proposition to be a non-Koszul Calabi-Yau DG algebra.  The proof there relies on straightforward computation and the definition of Calabi-Yau DG algebras.
  Theorem \ref{cycond} gives us a new way to prove it.
\begin{prop}\label{nonkoszul}
Let $\mathscr{A}$ be a connected cochain DG algebra such tht $$\mathscr{A}^{\#}=\mathbbm{k}\langle x,y\rangle/(x^2y-yx^2,xy^2-y^2x), |x|=|y|=1,\partial_{\mathscr{A}}(x)=y^2, \partial_{\mathscr{A}}(y)=0.$$  Then  the DG algebra $\mathscr{A}$ is not a Koszul DG algebra,  but it is Calabi-Yau.
\end{prop}
\begin{proof}
The trivial DG module ${}_{\mathscr{A}}\mathbbm{k}$ admits a minimal semi-free resolution $F$ with $$F^{\#}=\mathscr{A}^{\#}\oplus \mathscr{A}^{\#}\Sigma e_{y}\oplus \mathscr{A}^{\#}\Sigma e_{z}\oplus \mathscr{A}^{\#}\Sigma e_{x^2}\oplus \mathscr{A}^{\#}\Sigma e_t\oplus \mathscr{A}^{\#}\Sigma e_r, $$
 $\partial_F(\Sigma e_y)=y, \partial_{F}(\Sigma e_z)=x+y\Sigma e_y$, $\partial_F(\Sigma e_{x^2})=x^2, \partial_F(\Sigma e_t)=x^2\Sigma e_y+y\Sigma e_{x^2}$ and
 $\partial_{F}(\Sigma e_r)=y\Sigma e_t+x\Sigma e_{x^2}+x^2\Sigma e_Z$. By the minimality of $F$, we have \begin{align*}
 H(R\Hom_{\mathscr{A}}(\mathbbm{k},\mathbbm{k}))&\cong H(\Hom_{\mathscr{A}}(F,\mathbbm{k}))=\Hom_{\mathscr{A}}(F,\mathbbm{k})\\
 &=\mathbbm{k}1^*\oplus \mathbbm{k}(\Sigma e_y)^*\oplus \mathbbm{k}(\Sigma e_z)^*\oplus \mathbbm{k}(\Sigma e_{x^2})^*\oplus \mathbbm{k}(\Sigma e_{t})^*\oplus \mathbbm{k}(\Sigma e_{r})^*.
 \end{align*}
 It is concentrated in degrees $-1$ and $0$. Hence it is non-Koszul. Since $$\Hom_{\mathscr{A}}(F,F)^{\#}\cong [\mathbbm{k}1^*\oplus \mathbbm{k}(\Sigma e_y)^*\oplus \mathbbm{k}(\Sigma e_z)^*\oplus \mathbbm{k}(\Sigma e_{x^2})^*\oplus \mathbbm{k}(\Sigma e_{t})^*\oplus \mathbbm{k}(\Sigma e_r)^*]\otimes F^{\#}$$ is concentrated in degree $\ge -1$, we have $H^{-1}(\Hom_{\mathscr{A}}(F,F))=Z^{-1}(\Hom_{\mathscr{A}}(F,F))$.
 For any $f\in \Hom_{\mathscr{A}}(F,F)^{-1}$, it is uniquely determined by a matrix
 $$A_f=\left(
     \begin{array}{cccccc}
       0 & 0 & 0 &  0& 0 & 0 \\
       0 & 0 & 0 & 0 & 0 & 0 \\
       0 & 0 & 0 & 0 & 0 & 0 \\
       a_{11} & a_{12} & a_{13} & 0 & 0 & 0 \\
       a_{21} & a_{22} & a_{23} & 0 & 0 & 0 \\
       a_{31} & a_{32} & a_{33} & 0 & 0 & 0 \\
     \end{array}
   \right)\in M_{6}(\mathbbm{k})$$
   such that $$\left(
       \begin{array}{c}
         f(1) \\
         f(\Sigma e_y) \\
         f(\Sigma e_z) \\
         f(\Sigma e_{x^2}) \\
         f(\Sigma e_{t}) \\
         f(\Sigma e_{r}) \\
       \end{array}
     \right)=A_f\left(
                  \begin{array}{c}
                    1 \\
                    \Sigma e_y \\
                    \Sigma e_z \\
                    \Sigma e_{x^2} \\
                    \Sigma e_{t} \\
                    \Sigma e_{r} \\
                  \end{array}
                \right).$$
   Since
\begin{align*}
&\quad \quad \begin{cases}
\partial_F\circ f(1)=0\\
\partial_F\circ f(\Sigma e_y)=0\\
\partial_F\circ f(\Sigma e_z)=0\\
\partial_F\circ f(\Sigma e_{x^2})=a_{12}y+a_{13}(x+y\Sigma e_y)\\
\partial_F\circ f(\Sigma e_{t})=a_{22}y+a_{23}(x+y\Sigma e_y)\\
\partial_F\circ f(\Sigma e_r)=a_{32}y+a_{33}(x+y\Sigma e_y)
\end{cases} \\
& \text{and}\quad \begin{cases}
                                                 f\circ \partial_F(1)=0  \\
                                                   f\circ \partial_F(\Sigma e_y)=0          \\
                                                    f\circ \partial_F(\Sigma e_z)=0         \\
                                                    f\circ \partial_F(\Sigma e_{x^2})=0    \\
                                                     f\circ \partial_F(\Sigma e_{t})=-y(a_{11}+a_{12}\Sigma e_y+a_{13}\Sigma e_z)\\
                                                    f\circ \partial_F(\Sigma e_{r})=-y(a_{21}+a_{22}\Sigma e_y+a_{23}\Sigma e_z)-x(a_{11}+a_{12}\Sigma e_y+a_{13}\Sigma e_z)
                                                   \end{cases},
\end{align*}
 the $\mathscr{A}$-linear map $\partial_{\Hom}(f)=\partial_F\circ f+f\circ \partial_F$ corresponds to the matrix
\begin{align*}
\left(
     \begin{array}{cccccc}
       0 & 0 & 0 &  0& 0 & 0 \\
       0 & 0 & 0 & 0 & 0 & 0 \\
       0 & 0 & 0 & 0 & 0 & 0 \\
       a_{12}y+a_{13}x & a_{13}y & 0 & 0 & 0 & 0 \\
       a_{22}y+a_{23}x-a_{11}y & a_{23}y-a_{12}y & -a_{13}y & 0 & 0 & 0 \\
       a_{32}y+a_{33}x-a_{21}y-a_{11}x & a_{33}y-a_{22}y-a_{12}x & -a_{23}y-a_{13}x & 0 & 0 & 0 \\
     \end{array}
   \right).
\end{align*}
If $f\in Z^{-1}(\Hom_{\mathscr{A}}(F,F))$, then
\begin{align*}
\begin{cases}
a_{12}y+a_{13}x=0\\
a_{13}y=0\\
a_{22}y+a_{23}x-a_{11}y =0\\
a_{23}y-a_{12}y=0\\
a_{32}y+a_{33}x-a_{21}y-a_{11}x =0\\
a_{33}y-a_{22}y-a_{12}x=0 \\
-a_{23}y-a_{13}x=0,
\end{cases}
\end{align*}
which indicate that $a_{12}=a_{13}=0=a_{23}$, $a_{32}=a_{21}$ and $a_{11}=a_{22}=a_{33}$.
Hence $$H^{-1}(\Hom_{\mathscr{A}}(F,F))\cong \left\{\left(
     \begin{array}{cccccc}
       0 & 0 & 0 &  0& 0 & 0 \\
       0 & 0 & 0 & 0 & 0 & 0 \\
       0 & 0 & 0 & 0 & 0 & 0 \\
       a & 0 & 0 & 0 & 0 & 0 \\
       b & a & 0 & 0 & 0 & 0 \\
       c & b & a & 0 & 0 & 0 \\
     \end{array}
   \right)|a,b,c\in \mathbbm{k} \right\}.$$
     Hence we have $\dim_k \Hom_{\mathscr{A}}(F,F)^{-1}=\dim_kZ^{-1}(\Hom_{\mathscr{A}}(F,F))=3$. It implies that $Z^0(\Hom_{\mathscr{A}}(F,F))=H^0(\Hom_{\mathscr{A}}(F,F))$. For any $g\in Z^0(\Hom_{\mathscr{A}}(F,F))$, it is uniquely determined by a
     matrix $$X_g= \left(
     \begin{array}{cccccc}
       b_{11} & b_{12} & b_{13} &  0& 0 & 0 \\
       b_{21} & b_{22} & b_{23} & 0 & 0 & 0 \\
       b_{31} & b_{32} & b_{33} & 0 & 0 & 0 \\
       d_{11} & d_{12} & d_{13} & c_{11} & c_{12} & c_{13} \\
       d_{21} & d_{22} & d_{23} & c_{21} & c_{22} & c_{23} \\
       d_{31} & d_{32} & d_{33} & c_{31} & c_{32} & c_{33} \\
     \end{array}
   \right)=\left(\begin{array}{cc}
   B&0\\
   D&C\\
   \end{array}
   \right)$$
   such that$$\left(
       \begin{array}{c}
         g(1) \\
         g(\Sigma e_y) \\
         g(\Sigma e_z) \\
         g(\Sigma e_{x^2}) \\
         g(\Sigma e_{t}) \\
         g(\Sigma e_r) \\
       \end{array}
     \right)=X_g\left(
                  \begin{array}{c}
                    1 \\
                    \Sigma e_y \\
                    \Sigma e_z \\
                    \Sigma e_{x^2} \\
                    \Sigma e_{t} \\
                    \Sigma e_{r} \\
                  \end{array}
                \right),$$
where $b_{ij},c_{ij}\in \mathbbm{k}$ and $d_{ij}\in \mathscr{A}^1$. For simplicity, we let
$$\left(
      \begin{array}{cccccc}
        0 & 0 & 0 & 0 & 0 & 0 \\
        y & 0 & 0 & 0 & 0 & 0 \\
        x & y & 0 & 0 & 0 & 0 \\
        x^2 & 0 & 0 & 0 & 0 & 0 \\
        0 & x^2 & 0 & y & 0 & 0 \\
        0 & 0 & x^2 & x & y & 0 \\
      \end{array}
    \right)=\left(
      \begin{array}{cc}
      Q_1& 0\\
      Q_2& Q_1\\
      \end{array}\right),$$
      where
      $$Q_1=\left(
      \begin{array}{ccc}
        0 & 0 & 0  \\
        y & 0 & 0 \\
        x & y & 0 \\
\end{array}
\right), Q_2=\left(
      \begin{array}{ccc}
        x^2 & 0 & 0  \\
        0 & x^2 & 0 \\
        0 & 0 & x^2 \\
\end{array}
\right)$$
Since $0=\partial_{\Hom}(g)=\partial_F\circ g-g\circ \partial_F$, we have
\begin{align*}
& \left(
     \begin{array}{cc}
      0 & 0  \\
      \partial_{\mathscr{A}}(D) & 0 \\
     \end{array}
   \right)-\left(
     \begin{array}{cc}
      0 & 0  \\
      DQ_1 & 0 \\
     \end{array}
   \right)+ \left(\begin{array}{cc}
   B&0\\
   0&C\\
   \end{array}
   \right) \left(
      \begin{array}{cc}
        Q_1 & 0 \\
        Q_2 &  Q_1 \\
      \end{array}
    \right)\\
    &=\left(
      \begin{array}{cc}
        Q_1 & 0 \\
        Q_2 &  Q_1 \\
      \end{array}
    \right)\left(\begin{array}{cc}
   B&0\\
   D&C\\
   \end{array}
   \right),
    \end{align*}
    where $\partial_{\mathscr{A}}(D)=\left(
      \begin{array}{ccc}
        \partial_{\mathscr{A}}(d_{11}) &  \partial_{\mathscr{A}}(d_{12})& \partial_{\mathscr{A}}(d_{13})  \\
        \partial_{\mathscr{A}}(d_{21}) & \partial_{\mathscr{A}}(d_{22}) & \partial_{\mathscr{A}}(d_{23}) \\
        \partial_{\mathscr{A}}(d_{31}) & \partial_{\mathscr{A}}(d_{32}) & \partial_{\mathscr{A}}(d_{33}) \\
\end{array}
\right).$
Then
$$\left(
      \begin{array}{cc}
        BQ_1 & 0 \\
      \partial_{\mathscr{A}}(D)-DQ_1+ CQ_2 &  CQ_1 \\
      \end{array}
    \right)=\left(
      \begin{array}{cc}
        Q_1B & 0 \\
        Q_2B+Q_1D &  Q_1C \\
      \end{array}
    \right).$$
Therefore,
\begin{align*}
\left(
      \begin{array}{ccc}
        b_{11} & b_{12} & b_{13}  \\
        b_{21} & b_{22} & b_{23} \\
        b_{31} & b_{32} & b_{33} \\
\end{array}
\right)\left(
      \begin{array}{ccc}
        0 & 0 & 0  \\
        y & 0 & 0 \\
        x & y & 0 \\
\end{array}
\right)=\left(
      \begin{array}{ccc}
        0 & 0 & 0  \\
        y & 0 & 0 \\
        x & y & 0 \\
\end{array}
\right)\left(
      \begin{array}{ccc}
        b_{11} & b_{12} & b_{13}  \\
        b_{21} & b_{22} & b_{23} \\
        b_{31} & b_{32} & b_{33} \\
\end{array}
\right)\\
\left(
      \begin{array}{ccc}
        c_{11} & c_{12} & c_{13}  \\
        c_{21} & c_{22} & c_{23} \\
        c_{31} & c_{32} & c_{33} \\
\end{array}
\right)\left(
      \begin{array}{ccc}
        0 & 0 & 0  \\
        y & 0 & 0 \\
        x & y & 0 \\
\end{array}
\right)=\left(
      \begin{array}{ccc}
        0 & 0 & 0  \\
        y & 0 & 0 \\
        x & y & 0 \\
\end{array}
\right)\left(
      \begin{array}{ccc}
        c_{11} & c_{12} & c_{13}  \\
        c_{21} & c_{22} & c_{23} \\
        c_{31} & c_{32} & c_{33} \\
\end{array}
\right)
\end{align*}
and
\begin{align*}
& \partial_{\mathscr{A}}(D)-\left(
      \begin{array}{ccc}
        d_{11} & d_{12} & d_{13}  \\
        d_{21} & d_{22} & d_{23} \\
        d_{31} & d_{32} & d_{33} \\
\end{array}
\right)\left(
      \begin{array}{ccc}
        0 & 0 & 0  \\
        y & 0 & 0 \\
        x & y & 0 \\
\end{array}
\right)+\left(
      \begin{array}{ccc}
        c_{11} & c_{12} & c_{13}  \\
        c_{21} & c_{22} & c_{23} \\
        c_{31} & c_{32} & c_{33} \\
\end{array}
\right)\left(
      \begin{array}{ccc}
        x^2 & 0 & 0  \\
        0 & x^2 & 0 \\
        0 & 0 & x^2 \\
\end{array}
\right)\\
&=\left(
      \begin{array}{ccc}
        x^2 & 0 & 0  \\
        0 & x^2 & 0 \\
        0 & 0 & x^2 \\
\end{array}
\right)\left(
      \begin{array}{ccc}
        b_{11} & b_{12} & b_{13}  \\
        b_{21} & b_{22} & b_{23} \\
        b_{31} & b_{32} & b_{33} \\
\end{array}
\right)+\left(
      \begin{array}{ccc}
        0 & 0 & 0  \\
        y & 0 & 0 \\
        x & y & 0 \\
\end{array}
\right)\left(
      \begin{array}{ccc}
        d_{11} & d_{12} & d_{13}  \\
        d_{21} & d_{22} & d_{23} \\
        d_{31} & d_{32} & d_{33} \\
\end{array}
\right).
\end{align*}
These equations imply  that
\begin{align*}
\begin{cases}
b_{12}=b_{13}=b_{23}=0\\
b_{11}=b_{22}=b_{33}=c_{11}=c_{22}=c_{33}\\
c_{12}=c_{13}=c_{23}=0\\
b_{21}=b_{32}=c_{21}=c_{32}\\
b_{31}=c_{31}\\
d_{13}=0, d_{12}=ry, d_{11}=rx+sy\\
d_{23}=-ry, d_{22}=-(v+s)y,d_{21}=-vx+uy, \\
d_{33}=-rx+vy, d_{32}=-sx+wy, d_{31}=(u+w)x+\lambda y  \\
\end{cases}
\end{align*}
where $r,s,w,u,v,\lambda \in \mathbbm{k}$.
So $Z^0(\Hom_{\mathscr{A}}(F,F))$ is isomorphic to
$$\left\{ \left(
     \begin{array}{cccccc}
       d & 0 & 0 &  0& 0 & 0 \\
       e & d & 0 & 0 & 0 & 0 \\
       q & e & d & 0 & 0 & 0 \\
       rx+sy & ry & 0 & d &  0 & 0 \\
       -vx+uy & -(s+v)y & -ry & e & d &  0 \\
       (w+u)x+\lambda y & wy-sx & vy-rx & q & e & d \\
     \end{array}
   \right) | d,e,q,r,s,w,u,v,\lambda\in \mathbbm{k}\right\}.$$
For any $g\in Z^0(\Hom_{\mathscr{A}}(F,F))$ such that
$$\left(
                  \begin{array}{c}
                   g(1)\\
                    g(\Sigma e_y) \\
                    g(\Sigma e_z) \\
                    g(\Sigma e_{x^2}) \\
                    g(\Sigma e_{t}) \\
                    g(\Sigma e_{r}) \\
                  \end{array}
                \right)=\left(
     \begin{array}{cccccc}
       0 & 0 & 0 &  0& 0 & 0 \\
       0 & 0 & 0 & 0 & 0 & 0 \\
       0 & 0 & 0 & 0 & 0 & 0 \\
       rx+sy & ry & 0 & 0 &  0 & 0 \\
       -vx+uy & -(s+v)y & -ry & 0 & 0 &  0 \\
       (w+u)x+\lambda y & wy-sx & vy-rx & 0 & 0 & 0 \\
     \end{array}
   \right) \left(
                  \begin{array}{c}
                    1 \\
                    \Sigma e_y \\
                    \Sigma e_z \\
                    \Sigma e_{x^2} \\
                    \Sigma e_{t} \\
                    \Sigma e_{r} \\
                  \end{array}
                \right),$$
we have $\partial_{\Hom}(h)=g$, where $h$ is defined by
$$ \left(
                  \begin{array}{c}
                    h(1) \\
                    h(\Sigma e_y) \\
                    h(\Sigma e_z) \\
                    h(\Sigma e_{x^2}) \\
                    h(\Sigma e_{t}) \\
                    h(\Sigma e_{r}) \\
                  \end{array}
                \right)\left(
     \begin{array}{cccccc}
       0 & 0 & 0 &  0& 0 & 0 \\
       0 & 0 & 0 & 0 & 0 & 0 \\
       0 & 0 & 0 & 0 & 0 & 0 \\
       -u & s & r & 0 &  0 & 0 \\
        -\lambda & 0 & -v & 0 & 0 &  0 \\
       0 & 0 & w & 0 & 0 & 0 \\
     \end{array}
   \right)\left(
                  \begin{array}{c}
                    1 \\
                    \Sigma e_y \\
                    \Sigma e_z \\
                    \Sigma e_{x^2} \\
                    \Sigma e_{t} \\
                    \Sigma e_{r} \\
                  \end{array}
                \right).$$
Therefore $$H^0(\Hom_{\mathscr{A}}(F,F))\cong \left\{ \left(
     \begin{array}{cccccc}
       d & 0 & 0 &  0& 0 & 0 \\
       e & d & 0 & 0 & 0 & 0 \\
       q & e & d & 0 & 0 & 0 \\
       0 & 0 & 0 & d &  0 & 0 \\
       0 & 0 & 0 & e & d &  0 \\
       0 & 0 & 0 & q & e & d \\
     \end{array}
   \right)\quad |\quad d,e,q\in \mathbbm{k}\right\}.$$
From the computations above, the Ext-algebra $$\mathcal{E}=H(\Hom_{\mathscr{A}}(F,F))\cong \left\{ \left(
     \begin{array}{cccccc}
       d & 0 & 0 &  0& 0 & 0 \\
       e & d & 0 & 0 & 0 & 0 \\
       q & e & d & 0 & 0 & 0 \\
       a & 0 & 0 & d &  0 & 0 \\
       b & a & 0 & e & d &  0 \\
       c & b & a & q & e & d \\
     \end{array}
   \right)\quad |\quad a,b,c, d,e,q\in \mathbbm{k}\right\}.$$
It remains to show that the matrix algebra mentioned above
   is a symmetric Frobenius graded algebra. Its identity element $e_0$ is $\sum\limits_{i=1}^6E_{ii}$. Set $e_1=E_{21}+E_{32}+E_{54}+E_{65}$, $e_2=E_{31}+E_{64}$, $e_3=E_{41}+E_{52}+E_{63}$, $e_4=E_{51}+E_{62}$ and $e_5=E_{61}$.
   As a graded vector space, it admits a $\mathbbm{k}$-linear basis
  $ \left\{e_0,e_1,e_2,e_3,e_4,e_5\right \},$
  with $|e_0|=|e_1|=|e_2|=0$ and
  $|e_3|=|e_4|=|e_5|=-1$. Its multiplication is determined by the following tabular\\
\begin{center} \begin{tabular}{l|llllll}
$\cdot$ & $e_0$  & $e_1$ &  $e_2$ & $e_3$ & $e_4$ & $e_5$\\
 \hline
$e_0$    & $e_0$   & $e_1$ & $e_2$  & $e_3$ & $e_4$ & $e_5$ \\
$e_1$  & $e_1$ & $e_2$ & $0$    &$e_4$  & $e_5$ & $0$ \\
$e_2$  & $e_2$ & $0$   & $0$    &$e_5$  & $0$   & $0$\\
$e_3$  &$e_3$  & $e_4$ & $e_5$  &$0$    & $0$   & $0$\\
$e_4$  &$e_4$  & $e_5$ & $0$    &$0$    & $0$   & $0$\\
$e_5$  & $e_5$ & $0$   & $0$    &$0$    & $0$   & $0$
 \\
\end{tabular}.
\end{center}
One sees that the algebra is commutative.
Define a linear map
 \begin{align*}
 \theta: \mathcal{E}\to \Sigma \Hom_{\mathbbm{k}}(\mathcal{E},\mathbbm{k})
 \end{align*}
  by
 $$\left\{   \begin{array}{c}
                                                    e_0   \\
                                                    e_1                \\
                                                     e_2                  \\
                                                     e_3    \\
                                                     e_4    \\
                                                     e_5
                                                   \end{array}\right \}  \begin{array}{c}
                                                   \longrightarrow\\
                                                   \longrightarrow\\
                                                   \longrightarrow\\
                                                   \longrightarrow\\
                                                   \longrightarrow\\
                                                   \longrightarrow
                                                   \end{array}
                                                    \left\{   \begin{array}{c}
                                                     \Sigma e_5^* \,\, \\
                                                     \Sigma e_4^*  \,\,      \\
                                                     \Sigma e_3^* \,\,  \\
                                                     \Sigma e_2^* \,\, \\
                                                     \Sigma e_1^* \,\, \\
                                                     \Sigma e_0^*
                                                   \end{array}\right \}.
                                                     $$
One sees that $\theta$ is bijective. Hence we only need to show that $\theta$ is $\mathcal{E}$-linear.
We have
$\theta(e_ie_0)=\theta(e_i)=\Sigma e_{5-i}^*$ and
$$
e_i\theta(e_0)(\Sigma e_j)=e_i\Sigma e_5^*(\Sigma e_j)=\Sigma e_5^*(\Sigma e_je_i)=\begin{cases}
0, \quad \text{if} \,\,\,\,j\neq 5-i\\
\Sigma 1,\,\, \text{if}\,\,j=5-i.
\end{cases}$$
Hence $e_i\theta(e_0)=\Sigma e_{5-i}^*=\theta(e_ie_0)$. Since
$\theta(e_1e_1)=\theta(e_2)=\Sigma e_3^*$ and
$$e_1\theta(e_1)(\Sigma e_j)=\Sigma e_4^*(\Sigma e_j e_1)=\begin{cases}
0,\quad  \text{if}\,\, j\neq 3 \\
\Sigma 1, \quad \text{if}\,\,j =3
\end{cases},$$
we have $e_1\theta(e_1)=\Sigma e_3^*=\theta(e_1e_1)$. From
 $$e_2\theta(e_1)=e_2\Sigma e_4^*: \left\{   \begin{array}{c}
                                                   \Sigma e_0   \\
                                                    \Sigma e_1                \\
                                                     \Sigma e_2                  \\
                                                     \Sigma e_3    \\
                                                     \Sigma e_4    \\
                                                     \Sigma e_5    \\
                                                   \end{array}\right \}  \begin{array}{c}
                                                   \longrightarrow\\
                                                   \longrightarrow\\
                                                   \longrightarrow\\
                                                   \longrightarrow\\
                                                      \longrightarrow\\
                                                         \longrightarrow\
                                                   \end{array}
                                                    \left\{   \begin{array}{c}
                                                     0 \,\, \\
                                                     0 \,\,      \\
                                                     0 \,\,  \\
                                                     0 \,\, \\
                                                     0 \,\, \\
                                                     0 \,\,
                                                   \end{array}\right \}$$
 and
$$e_5\theta(e_1)=e_5\Sigma e_4^*: \left\{   \begin{array}{c}
                                                   \Sigma  e_0   \\
                                                    \Sigma e_1                \\
                                                      \Sigma e_2                  \\
                                                    \Sigma  e_3    \\
                                                    \Sigma  e_4    \\
                                                     \Sigma e_5    \\
                                                   \end{array}\right \}  \begin{array}{c}
                                                   \longrightarrow\\
                                                   \longrightarrow\\
                                                   \longrightarrow\\
                                                   \longrightarrow\\
                                                      \longrightarrow\\
                                                         \longrightarrow\
                                                   \end{array}
                                                    \left\{   \begin{array}{c}
                                                     0 \,\, \\
                                                     0 \,\,      \\
                                                     0 \,\,  \\
                                                     0 \,\, \\
                                                     0 \,\, \\
                                                     0 \,\,
                                                   \end{array}\right \},$$
we have $e_2\theta(e_1)=0=\theta(0)=\theta(e_2e_1)$ and $e_5\theta(e_1)=0=\theta(0)=\theta(e_5e_1)$.
Since $\theta(e_3e_1)=\theta(e_4)=\Sigma e_1^*$, $\theta(e_4e_1)=\theta(e_5)=\Sigma e_0^*$,
$$e_3\theta(e_1)(\Sigma e_j)=\Sigma e_4^*(\Sigma e_j e_3)=\begin{cases}
0,\quad  \,\, \text{if}\,\, j\neq 1 \\
\Sigma 1, \quad \text{if}\,\,j =1,
\end{cases}$$
and $$e_4\theta(e_1)(\Sigma e_j)=\Sigma e_4^*(\Sigma e_j e_4)=\begin{cases}
0,\quad  \,\, \text{if}\,\, j\neq 0 \\
\Sigma 1, \quad \text{if}\,\,j =0,
\end{cases}$$
we conclude that $e_3\theta(e_1)=\Sigma e_1^*=\theta(e_3e_1)$ and $e_4\theta(e_1)=\Sigma e_0^*=\theta(e_4e_1)$.
From $$e_2\theta(e_2)=e_2\Sigma e_3^*: \left\{   \begin{array}{c}
                                                   \Sigma e_0   \\
                                                    \Sigma e_1                \\
                                                     \Sigma e_2                  \\
                                                     \Sigma e_3    \\
                                                     \Sigma e_4    \\
                                                     \Sigma e_5    \\
                                                   \end{array}\right \}  \begin{array}{c}
                                                   \longrightarrow\\
                                                   \longrightarrow\\
                                                   \longrightarrow\\
                                                   \longrightarrow\\
                                                      \longrightarrow\\
                                                         \longrightarrow\
                                                   \end{array}
                                                    \left\{   \begin{array}{c}
                                                     0 \,\, \\
                                                     0 \,\,      \\
                                                     0 \,\,  \\
                                                     0 \,\, \\
                                                     0 \,\, \\
                                                     0 \,\,
                                                   \end{array}\right \},$$
we have $e_2\theta(e_2)=0=\theta(0)=\theta (e_2e_2)$. Similarly, $e_4\theta(e_2)=0=\theta(0)=\theta(e_4e_2)$ and $e_5\theta(e_2)=0=\theta(0)=\theta(e_5e_2)$. Since
$\theta(e_3e_2)=\theta(e_5)=\Sigma e_0^*$ and
$$e_3\theta(e_2)(\Sigma e_j)=\Sigma e_3^*(\Sigma e_j e_3)=\begin{cases}
0,\quad  \,\,\,\, \text{if}\,\, j\neq 0 \\
\Sigma 1, \quad \text{if}\,\,j =0,
\end{cases}$$
we have $e_3\theta(e_2)=\Sigma e_0^*=\theta(e_3e_2)$. Similarly, we can show
\begin{align*}
\begin{cases}e_i\theta(e_3)=e_i\Sigma e_2^*=0=\theta(0)=\theta(e_ie_3),i=3,4,5\\
e_j\theta(e_4)=e_i\Sigma e_1^*=0=\theta(0)=\theta(e_je_4), j=4,5 \\
e_5\theta(e_5)=e_5\Sigma e_0^*=0=\theta(0)=\theta(e_5e_5).
\end{cases}
\end{align*}
So $\theta$ is an isomorphism of $\mathcal{E}$-module. Since $\mathcal{E}$ is a commutative Frobenius graded algebra, it is a symmetric Frobenius graded algebra. By Theorem \ref{cycond}, $\mathscr{A}$ is a Calabi-Yau DG algebra.
\end{proof}

%%% ------------------------------------------------
\section{applications on dg free algebras}
In \cite{MXYA}, DG free algebra  were introduced and systematically studied. Recall that a connected cochain DG algebra $\mathscr{A}$ is called DG free if its underlying graded algebra $$\mathscr{A}^{\#}=\k\langle x_1,x_2,\cdots, x_n\rangle,\,\, \text{with}\,\, |x_i|=1,\,\, \forall i\in \{1,2,\cdots, n\}.$$
 By \cite[Theorem 2.4]{MXYA},  $\partial_{\mathscr{A}}$ is uniquely determined by a crisscrossed ordered $n$-tuple of $n\times n$ matrixes. For the case $n=2$, it is proved in \cite{MXYA} that all those non-trivial DG free algebras are Koszul and Calabi-Yau. It seems that things will become more complicated when $n$ is larger than $2$. In this section, we pick out several examples of non-trivial DG free algebras generated by $3$ degree one elements. We will apply Theorem \ref{Gorenstein} and Theorem \ref{cycond} to judge whether they are Gorenstein and Calabi-Yau.

\begin{ex}\label{example1}
Let $\mathscr{A}$ be a connected cochain DG algebra with $\mathscr{A}^{\#}=\mathbbm{k}\langle x_1,x_2,x_3\rangle$,
$|x_1|=|x_2|=|x_3|=1$ and
 $\partial_{\mathscr{A}}(x_1)=x_1^2, \partial_{\mathscr{A}}(x_2)=x_2x_1, \partial_{\mathscr{A}}(x_3)=x_1x_3$.
\end{ex}

\begin{prop}\label{ex1}
Let $\mathscr{A}$ be the cochain DG algebra described in Example \ref{example1}.
Then $H(\mathscr{A})=\mathbbm{k}[\lceil x_2x_3\rceil ]$, which implies that $\mathscr{A}$ is homologically smooth and Gorenstein. Furthermore,  the Ext-algebra $H(R\Hom_{\mathscr{A}}(\mathbbm{k},\mathbbm{k}))$ is a graded symmetric Frobenius algebra isomorphic to $\mathbbm{k}[x]/(x^2)$ with $|x|=-1$. Then Theorem \ref{cycond} indicates that $\mathscr{A}$ is a non-Koszul 
Calabi-Yau DG algebra.
\end{prop}
\begin{proof}
Since $\partial_{\mathscr{A}}(x_1)=x_1^2, \partial_{\mathscr{A}}(x_2)=x_2x_1, \partial_{\mathscr{A}}(x_3)=x_1x_3$, we have
$H^0(\mathscr{A})=\mathbbm{k}$ and $H^1(\mathscr{A})=0$. For any $z_2\in Z^2(\mathscr{A})$, it can be written as $z_2=\sum\limits_{i,j=1}^3c_{ij}x_ix_j$.
Then \begin{align*}
0 &=\partial_{\mathscr{A}}(z_2)\\
&=c_{12}[x_1^2x_2-x_1x_2x_1]+c_{22}[x_2x_1x_2-x_2^2x_1]+c_{32}[x_1x_3x_2-x_3x_2x_1]\\
&+c_{31}[x_1x_3x_1-x_3x_1^2]+c_{33}[x_1x_3^2-x_3x_1x_3].
\end{align*}
So $c_{12}=c_{22}=c_{31}=c_{32}=c_{33}=0$ and $z_2=c_{11}x_1^2+c_{13}x_1x_3+c_{21}x_2x_1+c_{23}x_2x_3$. Hence
$Z^2(\mathscr{A})=\mathbbm{k}x_1^2\oplus \mathbbm{k}x_1x_3\oplus \mathbbm{k}x_2x_1\oplus \mathbbm{k}x_2x_3$. We get $H^2(\mathscr{A})=k\lceil x_2x_3\rceil$ since $B^2(\mathscr{A})=\mathbbm{k}x_1^2\oplus \mathbbm{k}x_1x_3\oplus \mathbbm{k}x_2x_1$.
For any $z_3=x_1r_1+x_2r_2+x_3r_3\in Z^3(\mathscr{A})$, we have
\begin{align*}
0=\partial_{\mathscr{A}}(z_3)&=x_1^2r_1-x_1\partial_{\mathscr{A}}(r_1)+x_2x_1r_2-x_2\partial_{\mathscr{A}}(r_2)+x_1x_3r_3-x_3\partial_{\mathscr{A}}(r_3)\\
 &=x_1[x_1r_1-\partial_{\mathscr{A}}(r_1)+x_3r_3]+x_2[x_1r_2-\partial_{\mathscr{A}}(r_2)]-x_3\partial_{\mathscr{A}}(r_3).
\end{align*}
Thus $$\begin{cases}
\partial_{\mathscr{A}}(r_3)=0\\
x_1r_2-\partial_{\mathscr{A}}(r_2)=0\\
x_1r_1-\partial_{\mathscr{A}}(r_1)+x_3r_3=0.
\end{cases}$$
So $r_3=t_{11}x_1^2+t_{13}x_1x_3+t_{21}x_2x_1+t_{23}x_2x_3$, for some $t_{11},t_{13},t_{21},t_{23}\in \mathbbm{k}$. Let
$r_2=c_{11}x_1^2+c_{12}x_1x_2+c_{13}x_1x_3+c_{21}x_2x_1+c_{22}x_2^2+c_{23}x_3^2+c_{31}x_3x_1+c_{32}x_3x_1+c_{33}x_3^2$.
Substituting it into the equation $x_1r_2-\partial_{\mathscr{A}}(r_2)=0$. We get $c_{21}=-c_{12}$ and
$c_{11}=c_{13}=c_{22}=c_{23}=c_{31}=c_{32}=c_{33}=0$. Let
$$r_1=b_{11}x_1^2+b_{12}x_1x_2+b_{13}x_1x_3+b_{21}x_2x_1+b_{22}x_2^2+b_{23}x_3^2+b_{31}x_3x_1+b_{32}x_3x_1+b_{33}x_3^2.$$
Since $\partial_{\mathscr{A}}(r_1)=x_1r_1+x_3r_3$, we have
\begin{align*}
b_{12}x_1^2x_2&-b_{12}x_1x_2x_1+b_{22}x_2x_1x_2-b_{22}x_2^2x_1+b_{31}x_1x_3x_1-b_{31}x_3x_1^2\\
&+b_{32}x_1x_3x_2-b_{32}x_3x_2x_1+b_{33}x_1x_3^2-b_{33}x_3x_1x_3\\
=b_{11}x_1^3 &+b_{12}x_1^2x_2+b_{13}x_1^2x_3+b_{21}x_1x_2x_1+b_{22}x_1x_2^2+b_{23}x_1x_2x_3+b_{31}x_1x_3x_1\\
&+b_{32}x_1x_3x_2+b_{33}x_1x_3^2+t_{11}x_3x_1^2+t_{13}x_3x_1x_3+t_{21}x_3x_2x_1+t_{23}x_3x_2x_3.
\end{align*}
Hence
\begin{align*}
\begin{cases}
b_{11}=b_{13}=b_{22}=b_{23}=t_{23}=0\\
b_{12}=-b_{21},t_{11}=-b_{31},\\
t_{21}=-b_{32},t_{13}=-b_{33}.
\end{cases}
\end{align*}
Hence
\begin{align*}
\begin{cases}
r_1=\beta x_1x_2-\beta x_2x_1-\gamma x_3x_1-\omega x_3x_2-\lambda x_3^2\\
r_2= \theta x_1x_2-\theta x_2x_1\\
r_3=\gamma x_1^2+\lambda x_1x_3+\omega x_2x_1,
\end{cases}
\end{align*}
for some $\beta, \gamma, \lambda, \omega\in k$.
Then
\begin{align*}
x_1r_1+x_2r_2+x_3r_3&=\beta(x_1^2x_2-x_1x_2x_1)-\gamma x_1x_3x_1-\omega x_1x_3x_2-\lambda x_1x_3^2\\
&+\theta(x_2x_1x_2-x_2^2x_1)+\gamma x_3x_1^2+\lambda x_3x_1x_3+\omega x_3x_2x_1\\
&=\partial_{\mathscr{A}}[\beta x_1x_2-\gamma x_3x_1-\lambda x_3^2-\omega x_3x_2+\theta x_2^2]\in B^3(\mathscr{A}).
\end{align*}
Therefore, $H^3(\mathscr{A})=0$.
For any $i\ge 2$,
we want to prove  $H^{2i}=\mathbbm{k}\lceil x_2x_3\rceil^i$ and $H^{2i+1}(\mathscr{A})=0$. We assume inductively that they are correct for the cases $i\le t-1$.
We need to show the case $i=t$.
For any $z_{2t}=x_1q_1+x_2q_2+x_3q_3\in Z^{2t}(\mathscr{A})$, we have
\begin{align*}
0=\partial_{\mathscr{A}}(z_{2t})&=x_1^2q_1-x_1\partial_{\mathscr{A}}(q_1)+x_2x_1q_2-x_2\partial_{\mathscr{A}}(q_2)+x_1x_3q_3-x_3\partial_{\mathscr{A}}(q_3)\\
 &=x_1[x_1q_1-\partial_{\mathscr{A}}(q_1)+x_3q_3]+x_2[x_1q_2-\partial_{\mathscr{A}}(q_2)]-x_3\partial_{\mathscr{A}}(q_3).
\end{align*}
Then $$\begin{cases}
\partial_{\mathscr{A}}(q_3)=0\\
x_1q_2-\partial_{\mathscr{A}}(q_2)=0\\
x_1q_1-\partial_{\mathscr{A}}(q_1)+x_3q_3=0.
\end{cases}$$ Since $H^{2t-1}(\mathscr{A})=0$, we have $q_3=\partial_{\mathscr{A}}(a_3)$ for some $a_3\in \mathscr{A}^{2t-2}$. Then $\partial_{\mathscr{A}}(q_1)=x_1q_1+x_2\partial_{\mathscr{A}}(a_3)$ and hence $q_1=-x_3a_3$. Let $q_2=x_1\beta_1+x_2\beta_2+x_3\beta_3$. Since
$x_1q_2=\partial_{\mathscr{A}}(q_2)$, we have
\begin{align*}
& \quad x_1^2\beta_1+x_1x_2\beta_2+x_1x_3\beta_3 =x_1q_2=\partial_{\mathscr{A}}(q_2)=\partial_{\mathscr{A}}[x_1\beta_1+x_2\beta_2+x_3\beta_3]\\
&=x_1^2\beta_1-x_1\partial_{\mathscr{A}}(\beta_1)+x_2x_1\beta_2-x_2\partial_{\mathscr{A}}(\beta_2)+x_1x_3\beta_3-x_3\partial_{\mathscr{A}}(\beta_3).
\end{align*}
So $x_1x_2\beta_2+x_1\partial_{\mathscr{A}}(\beta_1)-x_2x_1\beta_2+x_2\partial_{\mathscr{A}}(\beta_2)+x_3\partial_{\mathscr{A}}(\beta_3)=0$. Then
\begin{align*}
\begin{cases}
\partial_{\mathscr{A}}(\beta_3)=0\\
x_2\beta_2+\partial_{\mathscr{A}}(\beta_1)=0\\
x_1\beta_2=\partial_{\mathscr{A}}(\beta_2).
\end{cases}
\end{align*}
Since $H^{2t-2}(\mathscr{A})=\mathbbm{k}\lceil x_2x_3\rceil^{t-1}$, we have $\beta_3=\lambda (x_2x_3)^{t-1}+\partial_{\mathscr{A}}(\omega)$ for some $\lambda\in \mathbbm{k}$ and $\omega\in \mathscr{A}^{2t-3}$. By the definition of $\partial_{\mathscr{A}}$,  the equation $\partial_{\mathscr{A}}(\beta_1)=-x_2\beta_2$ implies that $\beta_1=x_2\gamma$ for some $\gamma\in \mathscr{A}^{2t-3}$. Since $-x_2\beta_2=\partial_{\mathscr{A}}(x_2\gamma)=x_2x_1\gamma-x_2\partial_{\mathscr{A}}(\gamma)$, we have $\beta_2=\partial_{\mathscr{A}}(\gamma)-x_1\gamma$. Therefore,
\begin{align*}
z_{2t}&=x_1q_1+x_2q_2+x_3q_3\\
&=x_1[-x_3a_3]+x_2[x_1(x_2\gamma)+x_2(\partial_{\mathscr{A}}(\gamma)-x_1\gamma)+x_3(\lambda(x_2x_3)^{t-1}+\partial_{\mathscr{A}}(\omega))]\\
&\quad\quad +x_3\partial_{\mathscr{A}}(a_3)\\
&=\partial_{\mathscr{A}}[-x_3a_3+x_2^2\gamma +x_2x_3\omega ]+\lambda(x_2x_3)^{t}.
\end{align*}
Then we have $H^{2t}(\mathscr{A})=\mathbbm{k}\lceil x_2x_3\rceil^t$. It remains to show that $H^{2t+1}(\mathscr{A})=0$.

For any $z_{2t+1}=x_1s_1+x_2s_2+x_3s_3\in Z^{2t+1}(\mathscr{A})$, we have
\begin{align*}
0=\partial_{\mathscr{A}}(z_{2t+1})&=x_1^2s_1-x_1\partial_{\mathscr{A}}(s_1)+x_2x_1s_2-x_2\partial_{\mathscr{A}}(s_2)+x_1x_3s_3-x_3\partial_{\mathscr{A}}(s_3)\\
 &=x_1[x_1s_1-\partial_{\mathscr{A}}(s_1)+x_3s_3]+x_2[x_1s_2-\partial_{\mathscr{A}}(s_2)]-x_3\partial_{\mathscr{A}}(s_3).
\end{align*}
Then $$\begin{cases}
\partial_{\mathscr{A}}(s_3)=0\\
x_1s_2-\partial_{\mathscr{A}}(s_2)=0\\
x_1s_1-\partial_{\mathscr{A}}(s_1)+x_3s_3=0.
\end{cases}$$
Since $H^{2t}(\mathscr{A})=k\lceil x_2x_3\rceil^t$, we have $s_3=b(x_2x_3)^t+\partial_{\mathscr{A}}(c)$ for some $b\in k$ and $c\in \mathscr{A}^{2t-1}$. Then $\partial_{\mathscr{A}}(s_1)=x_1s_1+x_3s_3=x_1s_1+x_3[b(x_2x_3)^t+\partial_{\mathscr{A}}(c)]$. It implies that
$b=0$ and $s_1=-x_3c$. Let $s_2=x_1\xi_1+x_2\xi_2+x_3\xi_3$.
Since $\partial_{\mathscr{A}}(s_2)=x_1s_2$, we have
\begin{align*}
x_1^2\xi_1+x_1x_2\xi_2&+x_1x_3\xi_3=x_1s_2=\partial_{\mathscr{A}}(s_2)=\partial_{\mathscr{A}}[x_1\xi_1+x_2\xi_2+x_3\xi_3]\\
&=x_1^2\xi_1-x_1\partial_{\mathscr{A}}(\xi_1)+x_2x_1\xi_2-x_2\partial_{\mathscr{A}}(\xi_2)+x_1x_3\xi_3-x_3\partial_{\mathscr{A}}(\xi_3).
\end{align*}
Then $x_1x_2\xi_2+x_1\partial_{\mathscr{A}}(\xi_1)-x_2x_1\xi_2+x_2\partial_{\mathscr{A}}(\xi_2)+x_3\partial_{\mathscr{A}}(\xi_3)=0$ and hence
\begin{align*}
\begin{cases}
x_2\xi_2+\partial_{\mathscr{A}}(\xi_1)=0 \\
\partial_{\mathscr{A}}(\xi_2)=x_1\xi_2 \\
\partial_{\mathscr{A}}(\xi_3)=0.
\end{cases}
\end{align*}
Since $H^{2t-1}(\mathscr{A})=0$, we have $\xi_3=\partial_{\mathscr{A}}(\chi)$ for some $\chi\in \mathscr{A}^{2t-2}$. Since
$\partial_{\mathscr{A}}(\xi_1)=-x_2\xi_2$, we have $\xi_1=x_2\eta$ for some $\eta\in \mathscr{A}^{2t-2}$, and $\xi_2=\partial_{\mathscr{A}}(\eta)-x_1\eta$. Then $\partial_{\mathscr{A}}(\xi_2)=-x_1^2\eta+x_1\partial_{\mathscr{A}}(\eta)=x_1\xi_2$.
Thus $$s_2=x_1\xi_1+x_2\xi_2+x_3\xi_3=x_1x_2\eta+x_2\partial_{\mathscr{A}}(\eta)-x_2x_1\eta+x_3\partial_{\mathscr{A}}(\chi).$$
We have \begin{align*}
z_{2t+1}&=x_1s_1+x_2s_2+x_3s_3\\
&=x_1(-x_3c)+x_1[x_1x_2\eta+x_2\partial_{\mathscr{A}}(\eta)-x_2x_1\eta+x_3\partial_{\mathscr{A}}(\chi)]+x_3\partial_{\mathscr{A}}(c)\\
&=\partial_{\mathscr{A}}[-x_3c+x_1x_2\eta+x_3\partial_{\mathscr{A}}(\chi)]\in B^{2t+1}(\mathscr{A}).
\end{align*}
So $H^{2t+1}(\mathscr{A})=0$.
By the induction above, we conclude that $H(\mathscr{A})=\mathbbm{k}[\lceil x_2x_3\rceil ]$. Since $H(\mathscr{A})$ is a graded Gorenstein algebra,  the DG algebra is Gorenstein. The graded $H(\mathscr{A})$-module $\mathbbm{k}$ admits a minimal free resolution
$$0\to H(\mathscr{A})e\stackrel{d_1}{\to} H(\mathscr{A})\stackrel{\varepsilon}{\to} \mathbbm{k}\to 0,$$
where $d_1(e)=\lceil x_2x_3\rceil$ and $\varepsilon$ is the augmentation map of $H(\mathscr{A})$. From it, we can construct an Eilenberg-Moore resolution $F$ of ${}_{\mathscr{A}}\mathbbm{k}$ with $F^{\#}=\mathscr{A}^{\#}\oplus \mathscr{A}^{\#}\Sigma e$ such that $\partial_F(\Sigma e)=x_2x_3$.
Obviously, $F$ is minimal and hence $$H(R\Hom_{\mathscr{A}}(\mathbbm{k},\mathbbm{k}))=\Hom_{\mathscr{A}}(F,\mathbbm{k})=\mathbbm{k}1^*\oplus \mathbbm{k}(\Sigma e)^*.$$
So $\mathscr{A}$ is homologically smooth. We need to figure out the algebra structure of the Ext-algebra $H(R\Hom_{\mathscr{A}}(\mathbbm{k},\mathbbm{k}))$.
 For this, we consider the DG algebra
$\Hom_{\mathscr{A}}(F,F)$ since $H(\Hom_{\mathscr{A}}(F,F))=H(R\Hom_{\mathscr{A}}(\mathbbm{k},\mathbbm{k}))$. Note that $$\Hom_{\mathscr{A}}(F,F)^{\#}\cong [k\cdot 1^*\oplus \mathbbm{k}(\Sigma e)^*]\otimes F^{\#}$$ is concentrated in degrees $\ge -1$. Hence $\Hom_{\mathscr{A}}(F,F))=Z^{-1}(\Hom_{\mathscr{A}}(F,F))$. For any $f\in \Hom_{\mathscr{A}}(F,F)^{-1}$, it is determined by a matrix
$\left(
  \begin{array}{cc}
    0 & 0 \\
    c & 0 \\
  \end{array}
\right)$ such that $$\left(
                      \begin{array}{c}
                        f(1) \\
                        f(\Sigma e) \\
                      \end{array}
                    \right)=\left(
                              \begin{array}{cc}
                                0 & 0 \\
                                c & 0 \\
                              \end{array}
                            \right)\left(
                      \begin{array}{c}
                        1 \\
                        \Sigma e \\
                      \end{array}
                    \right),$$ for some $c\in \mathbbm{k}$. For convenience,  we write $f_c$ instead of $f$.
                     Since $\partial_{\Hom}(f_c)=\partial_F\circ f_c+f_c\circ \partial_F$, we have
 \begin{align*}\partial_{\Hom}(f_c)(1)&= \partial_F\circ f_c(1)+f_c\circ \partial_F(1)=0 \,\,\, \text{and}\\
 \partial_{\Hom}(f_c)(\Sigma e)&=\partial_F\circ f_c(\Sigma e)+ f_c\circ \partial_F(\Sigma e)=0.
 \end{align*}
 So $Z^{-1}(\Hom_{\mathscr{A}}(F,F))=H^{-1}(\Hom_{\mathscr{A}}(F,F))=\Hom_{\mathscr{A}}(F,F).$ Furthermore, we have $H^0(\Hom_{\mathscr{A}}(F,F))=Z^0(\Hom_{\mathscr{A}}(F,F))$. We have $H^{-1}(\Hom_{\mathscr{A}}(F,F))=\mathbbm{k}\lceil f_1\rceil$ and $H^0(\Hom_{\mathscr{A}}(F,F))=\mathbbm{k}\lceil \mathrm{id}_F\rceil$. Since $f_1\circ f_1=0$, we have $\lceil f_1\rceil^2=0$. Clearly, $H(\Hom_{\mathscr{A}}(F,F))\cong \mathbbm{k}[x]/(x^2)$ with $|x|=-1$. It is a graded symmetric Frobenius algebra. By Theorem \ref{cycond}, $\mathscr{A}$ is Calabi-Yau.
\end{proof}
It is natural to ask whether any non-trivial DG free algebras generated in three degree one elements are Calabi-Yau.
The following example indicates it is not right. We even can't ensure that they are Gorenstein. One can see this from the following example.

\begin{ex}\label{ex3}
Let $\mathscr{A}$ be a connected cochain DG algebra with $\mathscr{A}^{\#}=\mathbbm{k}\langle x_1,x_2,x_3\rangle$, $|x_1|=|x_2|=|x_3|=1$ and
$\partial_{\mathscr{A}}(x_1)=x_2x_3,\,\, \partial_{\mathscr{A}}(x_2)=0 \,\, \partial_{\mathscr{A}}(x_3)=0$.

\end{ex}
\begin{prop}
Let $\mathscr{A}$ be the DG algebra illustrated in Example \ref{ex3}. Then $$H(\mathscr{A})=\mathbbm{k}\langle \lceil x_2\rceil, \lceil x_3\rceil \rangle/(\lceil x_2\rceil\cdot \lceil x_3\rceil)$$ and $\mathscr{A}$ is a Koszul, homologically smooth and non-Gorenstein DG algebra.
\end{prop}
\begin{proof}
Clearly, $H^0(\mathscr{A})=\mathbbm{k}$ and $H^1(\mathscr{A})=\mathbbm{k}\lceil x_2\rceil\oplus \mathbbm{k}\lceil x_3\rceil$. For any $z\in Z^2(\mathscr{A})$,
we have $$z=c_{11}x_1^2+c_{12}x_1x_2+c_{13}x_1x_3+c_{21}x_2x_1+c_{22}x_2^2+c_{23}x_2x_3+c_{31}x_3x_1+c_{32}x_3x_2+c_{33}x_3^2,$$ for some $c_{ij}\in \mathbbm{k}$. Then
\begin{align*}
0&=\partial_{\mathscr{A}}(z)\\
&=\partial_{\mathscr{A}}[c_{11}x_1^2+c_{12}x_1x_2+c_{13}x_1x_3+c_{21}x_2x_1+c_{31}x_3x_1]\\
&=c_{11}x_2x_3x_1-c_{11}x_1x_2x_3+c_{12}x_2x_3x_2+c_{13}x_2x_3^2-c_{21}x_2^2x_3-c_{31}x_3x_2x_3.
\end{align*}
It implies that $c_{11}=0$, $c_{12}=c_{13}=c_{21}=c_{31}=0$. Hence $$Z^2(\mathscr{A})=\mathbbm{k}x_2^2\oplus \mathbbm{k}x_2x_3\oplus \mathbbm{k}x_3x_2\oplus \mathbbm{k}x_3^2.$$
From the definition of $\partial_{\mathscr{A}}$, one sees that $B^2(\mathscr{A})=\mathbbm{k} x_2x_3$. So $$H^2(\mathscr{A})=\mathbbm{k}\lceil x_2\rceil^2 \oplus \mathbbm{k}\lceil x_3\rceil \lceil x_2\rceil \oplus \mathbbm{k}\lceil x_3\rceil^2.$$
Suppose inductively that we have proved $$H^i(\mathscr{A})=\bigoplus\limits_{j=0}^i \mathbbm{k}\lceil x_3\rceil^{i-j}\lceil x_2\rceil^j, j\le n-1.$$
For any $\omega \in Z^n(\mathscr{A})$, we may write
$\omega= x_1b_1+x_2b_2+x_3b_3$, for some $b_i\in \mathscr{A}^{n-1}$.
Then
\begin{align*}
0&=\partial_{\mathscr{A}}(\omega)=x_2x_3b_1-x_2\partial_{\mathscr{A}}(b_2)-x_3\partial_{\mathscr{A}}(b_3)\\
 &=x_2[x_3b_1-\partial_{\mathscr{A}}(b_2)]-x_3\partial_{\mathscr{A}}(b_3).
\end{align*}
It implies that $\partial_{\mathscr{A}}(b_2)=x_3b_1$ and $\partial_{\mathscr{A}}(b_3)=0$. Therefore, $b_1\in B^{n-1}(\mathscr{A})$ and $b_3\in Z^{n-1}(\mathscr{A})$. We have $b_1=\partial_{\mathscr{A}}(c_1)$ and $b_3=\sum\limits_{j=0}^{n-1}t_jx_3^{n-1-j}x_2^j+\partial_{\mathscr{A}}(c_3)$, for some $t_j\in \mathbbm{k}$ and $c_1, c_3\in \mathscr{A}^{n-2}$. Then $b_2=-x_3c_1+\sum\limits_{j=0}^{n-1}r_jx_3^{n-1-j}x_2^j+\partial_{\mathscr{A}}(c_2)$, for some $r_j\in \mathbbm{k}$ and $c_2\in \mathscr{A}^{n-2}$. Hence
\begin{align*}
\omega &= x_1b_1+x_2b_2+x_3b_3\\
       &= x_1\partial_{\mathscr{A}}(c_1)+x_2[-x_3c_1+\sum\limits_{j=0}^{n-1}r_jx_3^{n-1-j}x_2^j +\partial_{\mathscr{A}}(c_2)]\\
       &\quad+x_3[\sum\limits_{j=0}^{n-1}t_jx_3^{n-1-j}x_2^j+\partial_{\mathscr{A}}(c_3)]\\
       &=\partial_{\mathscr{A}}(-x_1c_1-x_2c_2-x_3c_3+\sum\limits_{j=0}^{n-2}r_jx_1x_3^{n-j-2}x_2^j)+r_{n-1}x_2^n+\sum\limits_{j=0}^{n-1}t_jx_3^{n-j}x_2^j.
\end{align*}
It means that $$H^n(\mathscr{A})=\bigoplus\limits_{j=0}^n\mathbbm{k} \lceil x_3\rceil^{n-j}\lceil x_2\rceil^j.$$
By the induction above, we conclude that $$H(\mathscr{A})=\mathbbm{k} \langle \lceil x_2\rceil, \lceil x_3\rceil\rangle/(\lceil x_2\rceil\cdot\lceil x_3\rceil).$$
The graded $H(\mathscr{A})$-module $\mathbbm{k}$ admits a minimal free resolution:
$$0\to H(\mathscr{A})e_r\stackrel{d_2}{\to} H(\mathscr{A})e_{x_2}\oplus H(\mathscr{A})e_{x_3}\stackrel{d_1}{\to} H(\mathscr{A})\stackrel{\epsilon}{\to} \mathbbm{k}\to 0,$$
where the differentials $d_2$ and $d_1$ are $H(\mathscr{A})$-linear maps defined by
$$d_2(e_r)=\lceil x_2\rceil e_{x_3}, d_1(e_{x_2})=\lceil x_2\rceil,
d_1(e_{x_3})=\lceil x_3\rceil. $$
By the construction of Eilengberg-Moore resolution, we can get a semi-free resolution $F$ of the DG $\mathscr{A}$-module $\mathbbm{k}$ from the free resolution above. According to the constructing procedure,
$$F^{\#}\cong \mathscr{A}^{\#}\oplus \mathscr{A}^{\#}\Sigma e_{x_2}\oplus \mathscr{A}^{\#}\Sigma e_{x_3}\oplus \mathscr{A}^{\#}\Sigma^2 e_r$$
and the differential $\partial_F$ is defined by $$\partial_F(1)=0, \partial_F(\Sigma e_{x_2})=x_2, \partial_F(\Sigma e_{x_3})=x_3, \partial_F(\Sigma^2 e_r)=x_2\Sigma e_{x_3}.$$
Note that $F$ has a semi-basis $\{1,\Sigma e_{x_2},\Sigma e_{x_3},\Sigma^2 e_r\}$ concentrated in degree $0$. Hence $\mathscr{A}$ is a Koszul and homologically smooth DG algebra. It remains to show that $\mathscr{A}$ is not Gorenstein. It suffices to prove that its Ext-algebra is not Frobenius by Theorem \ref{Gorenstein}.
By the minimality of $F$, the differential of $\Hom_{\mathscr{A}}(F,\mathbbm{k})$ is zero. We have
\begin{align*}
H(\Hom_{\mathscr{A}}(F,\mathbbm{k}))=\Hom_{\mathscr{A}}(F,\mathbbm{k})=\mathbbm{k}1^*\oplus \mathbbm{k}(\Sigma e_{x_2})^*\oplus \mathbbm{k}(\Sigma e_{x_3})^*\oplus \mathbbm{k}(\Sigma^2 e_r)^*.
\end{align*}
So the Ext-algebra $E=H(\Hom_{\mathscr{A}}(F,F))\cong H(\Hom_{\mathscr{A}}(F,\mathbbm{k}))$ is concentrated in degree $0$. On the other hand, $$\Hom_{\mathscr{A}}(F,F)^{\#}\cong \{\mathbbm{k} 1^*\oplus \mathbbm{k}(\Sigma e_{x_2})^*\oplus \mathbbm{k}(\Sigma e_{x_3})^*\oplus \mathbbm{k}(\Sigma^2 e_r)^*\}\otimes_{\mathbbm{k}} F^{\#}$$ is concentrated in degrees $\ge 0$ since $|1^*|=|(\Sigma e_{x_2})^*|=|(\Sigma e_{x_3})^*|=|(\Sigma^2 e_r)^*|=0$ and $F$ is concentrated in degrees $\ge 0$. This implies that $E= Z^0(\Hom_{\mathscr{A}}(F,F))$. Since $F^{\#}$ is a free graded $\mathscr{A}^{\#}$-module with a basis
$\{1, \Sigma e_{x_2}, \Sigma e_{x_3},\Sigma^2 e_r \}$ concentrated in degree $0$,
  the elements in  $\Hom_{\mathscr{A}}(F,F)^0$ is one to one correspondence with the matrixes in $M_4(\mathbbm{k})$. Indeed, any $f\in \Hom_{\mathscr{A}}(F,F)^0$ is uniquely determined by
  a matrix $A_f=(a_{ij})_{4\times 4}\in M_4(\mathbbm{k})$ with
$$\left(
                         \begin{array}{c}
                          f(1) \\
                          f(\Sigma e_{x_2})\\
                           f(\Sigma e_{x_3})\\
                            f(\Sigma^2 e_r)
                         \end{array}
                       \right) =      A_f \cdot \left(
                         \begin{array}{c}
                          1 \\
                          \Sigma e_{x_2}\\
                           \Sigma e_{x_3}\\
                            \Sigma^2 e_r
                         \end{array}
                       \right).  $$
                       And $f\in  Z^0[\Hom_{\mathcal{A}}(F,F)]$ if and only if $\partial_{F}\circ f=f\circ \partial_{F}$, if and only if

 $$ A_f \left(
            \begin{array}{cccc}
              0  & 0  & 0 & 0 \\
              x_2 & 0 & 0 & 0 \\
              x_3 & 0 & 0 & 0 \\
              0   & 0 & x_2 & 0\\
            \end{array}
          \right) = \left(
            \begin{array}{cccc}
              0 & 0 & 0 & 0 \\
              x_2 & 0& 0 & 0 \\
              x_3 & 0 & 0 & 0 \\
              0 &  0& x_2 & 0\\
            \end{array}
          \right) A_f,$$
 which is also equivalent to
                       $$\begin{cases}
                       a_{12}=a_{13}=a_{14}=a_{23}=a_{24}=a_{32}=a_{34}=a_{43}=0\\
                       a_{11}=a_{22}=a_{33}= a_{44}\\
                       a_{42}=a_{31}\\
                       \end{cases}$$
by direct computations. Hence the the Ext-algebra $$H(\Hom_{\mathscr{A}}(F,F))\cong \mathcal{E}=\left\{\left(
                                                             \begin{array}{cccc}
                                                               a & 0 & 0 & 0  \\
                                                               b & a & 0 & 0 \\
                                                               c & 0 & a & 0  \\
                                                               d & c & 0 & a
                                                             \end{array}
                                                           \right)
\quad | \quad a,b,c,d\in \mathbbm{k} \right\}.$$
Let $e_0=\sum\limits_{i=1}^4E_{ii}$, $e_1=E_{21}$, $e_2=E_{31}+E_{42}$ and $e_3=E_{41}$. They constitute a basis of $\mathcal{E}$. The multiplication of $\mathcal{E}$ is defined by the following tabular: \\
\begin{center} \begin{tabular}{l|llll}
$\cdot$   & $e_0$   & $e_1$ &  $e_2$ & $e_3$ \\
 \hline
$e_0$     & $e_0$   & $e_1$ & $e_2$  & $e_3$ \\
$e_1$     & $e_1$   & $0$   & $0$    & $0$  \\
$e_2$     & $e_2$   & $e_3$ & $0$    & $0$ \\
$e_3$     & $e_3$   & $0$   & $0$    & $0$  \\
\end{tabular}.
\end{center}
If $\mathcal{E}$ is a Frobenius algebra, then there is an isomorphism $\sigma: \mathcal{E}\to \Hom_{\mathbbm{k}}(\mathcal{E},\mathbbm{k})$ of left $\mathcal{E}$-modules.
There exists a matrix $M=(m_{ij})_{4\times 4}\in \mathrm{GL}_4(k)$ such that
\begin{align*}
\left(  \begin{array}{c}
                                                    \sigma(e_0)   \\
                                                   \sigma (e_1)               \\
                                                   \sigma(e_2)                  \\
                                                   \sigma(e_3)    \\
                                                   \end{array}\right)=M \left(   \begin{array}{c}
                                                    e_0^*   \\
                                                    e_1^*               \\
                                                    e_2^*                  \\
                                                    e_3^*    \\
                                                   \end{array}\right).
\end{align*}
Since
\begin{align*}
 \sigma(e_2\cdot e_1)=\sigma(e_3)= m_{41}e_0^*+m_{42} e_1^* + m_{43} e_2^*+m_{44}e_3^* : \left\{   \begin{array}{c}
                                                    e_0   \\
                                                   e_1             \\
                                               e_2               \\
                                               e_3  \\
                                                   \end{array}\right \}  \begin{array}{c}
                                                   \longrightarrow\\
                                                   \longrightarrow\\
                                                   \longrightarrow\\
                                                   \longrightarrow\\
                                                   \end{array}
                                                    \left\{   \begin{array}{c}
                                                     m_{41} \,\, \\
                                                     m_{42} \,\,      \\
                                                     m_{43} \,\,  \\
                                                     m_{44} \,\, \\
                                                   \end{array}\right \}
 \end{align*}
and
\begin{align*}
 e_2\sigma(\sigma e_1)=e_2[m_{21}e_0^*+m_{22} e_1^* + m_{23} e_2^*+m_{24}e_3^*]: \left\{   \begin{array}{c}
                                                    e_0   \\
                                                    e_1                \\
                                                    e_2                  \\
                                                    e_3  \\
                                                   \end{array}\right \}  \begin{array}{c}
                                                   \longrightarrow\\
                                                   \longrightarrow\\
                                                   \longrightarrow\\
                                                   \longrightarrow\\
                                                   \end{array}
                                                    \left\{   \begin{array}{c}
                                                     m_{23} \,\, \\
                                                     0 \,\,      \\
                                                     0 \,\,  \\
                                                     0 \,\, \\
                                                   \end{array}\right \},
 \end{align*}
 we conclude that $m_{41}=m_{23},m_{42}=m_{43}=m_{44}=0$ by $\sigma(e_2e_1)=e_2\sigma(e_1)$.
 Similarly, we can get $m_{32}=m_{34}=0$ from $0=\sigma(0)=\sigma(e_1e_2)=e_1\sigma (e_2)$, since
 \begin{align*}
 e_1\sigma(e_2)=e_1[m_{31}e_0^*+m_{32} e_1^* + m_{33} e_2^*+m_{34}e_3^*]: \left\{   \begin{array}{c}
                                                    e_0   \\
                                                    e_1                \\
                                                    e_2                  \\
                                                    e_3  \\
                                                   \end{array}\right \}  \begin{array}{c}
                                                   \longrightarrow\\
                                                   \longrightarrow\\
                                                   \longrightarrow\\
                                                   \longrightarrow\\
                                                   \end{array}
                                                    \left\{   \begin{array}{c}
                                                     m_{32} \,\, \\
                                                     0 \,\,      \\
                                                     m_{34} \,\,  \\
                                                     0 \,\, \\
                                                   \end{array}\right \}.
 \end{align*}
 By $0=\sigma(0)=\sigma(e_2e_2)=e_2\sigma (e_2)$ and
 \begin{align*}
 e_2\sigma(e_2)=e_2[m_{31}e_0^*+m_{32} e_1^* + m_{33} e_2^*+m_{34}e_3^*]: \left\{   \begin{array}{c}
                                                    e_0   \\
                                                    e_1                \\
                                                    e_2                  \\
                                                    e_3  \\
                                                   \end{array}\right \}  \begin{array}{c}
                                                   \longrightarrow\\
                                                   \longrightarrow\\
                                                   \longrightarrow\\
                                                   \longrightarrow\\
                                                   \end{array}
                                                    \left\{   \begin{array}{c}
                                                     m_{33} \,\, \\
                                                     0 \,\,      \\
                                                     0 \,\,  \\
                                                     0 \,\, \\
                                                   \end{array}\right \},
 \end{align*}
 we get $m_{33}=0$.
 This contradicts with $M\in \mathrm{GL}_4(\mathbbm{k})$. Hence $\mathcal{E}$ is not a Frobenius algebra.
By Theorem \ref{Gorenstein}, $\mathscr{A}$ is not Gorenstein.
\end{proof}

\begin{rem}
Note that the DG free algebra in Example \ref{example1} is a non-Koszul Calabi-Yau DG algebra, while the algebra in Example \ref{ex3} is not Gorenstein but Koszul and homologically smooth. These examples indicate that the DG free algebras of the case $n=3$ are quite different from those of the case $n=2$. In spite of this, there are many Koszul Calabi-Yau non-trivial DG free algebras generated in three degree one elements. In the rest of this section, we will give some examples to explain it.
\end{rem}

\begin{ex}\label{ex2}
Let $\mathcal{A}$ be a connected cochain DG algebra such that
\begin{align*}
\mathscr{A}^{\#}=\mathbbm{k}\langle x_1,x_2,x_3\rangle,\,\,
\partial_{\mathscr{A}}(x_1)=x_3^2,\,\, \partial_{\mathscr{A}}(x_2)=x_1x_3+x_3x_1 \,\, \partial_{\mathscr{A}}(x_3)=0.
\end{align*}
\end{ex}
\begin{prop}
Let $\mathscr{A}$ be the DG algebra in Example \ref{ex2}. Then
$$H(\mathscr{A})=\mathbbm{k}[\lceil x_3\rceil,\lceil x_1^2+x_2x_3+x_3x_2\rceil ]/(\lceil x_3\rceil^2)$$ and
it is a Koszul and Calabi-Yau DG algebra.
\end{prop}
\begin{proof}
Obviously, we have $H^1(\mathscr{A})=\mathbbm{k}$ and $H^1(\mathscr{A})=\mathbbm{k}\lceil x_3\rceil$. For any $z\in Z^2(\mathscr{A})$, it can be written by
$$z=c_{11}x_1^2+c_{12}x_1x_2+c_{13}x_1x_3+c_{21}x_2x_1+c_{22}x_2^2+c_{23}x_2x_3+c_{31}x_3x_1+c_{32}x_3x_2+c_{33}x_3^2,$$
with each $c_{ij}\in \mathbbm{k}$.
We have
\begin{align*}
0&=\partial_{\mathscr{A}}(z)=c_{11}(x_3^2x_1-x_1x_3^2)+c_{12}x_3^2x_2-c_{12}x_1(x_1x_3+x_3x_1)+c_{13}x_3^3\\
&\quad +c_{21}(x_1x_3+x_3x_1)x_1-c_{21}x_2x_3^2+c_{22}(x_1x_3+x_3x_1)x_2-c_{22}x_2(x_1x_3+x_3x_1)\\
&\quad +c_{23}(x_1x_3+x_3x_1)x_3-c_{31}x_3^3-c_{32}x_3(x_1x_3+x_3x_1)\\
&=(c_{11}-c_{32})x_3^2x_1+(c_{23}-c_{11})x_1x_3^2+c_{12}x_3^2x_2-c_{12}x_1^2x_3+(c_{21}-c_{12})x_1x_3x_1\\
&\quad +(c_{13}-c_{31})x_3^3+c_{21}x_3x_1^2-c_{21}x_2x_3^2+c_{22}x_1x_3x_2+c_{22}x_3x_1x_2\\
&\quad -c_{22}x_2x_1x_3-c_{22}x_2x_3x_1+(c_{23}-c_{32})x_3x_1x_3.
\end{align*}
It implies that $c_{11}=c_{32}=c_{23}$, $c_{12}=0=c_{21}=c_{22}$ and $c_{13}=c_{31}$. Hence $$Z^2(\mathscr{A})=\mathbbm{k}x_3^2\oplus \mathbbm{k}(x_1x_3+x_3x_1)\oplus \mathbbm{k}(x_1^2+x_2x_3+x_3x_2).$$ Since $B^2(\mathscr{A})=\mathbbm{k}x_3^2\oplus \mathbbm{k}(x_1x_3+x_3x_1)$, we have $H^2(\mathscr{A})=\mathbbm{k}\lceil x_1^2+x_2x_3+x_3x_2\rceil$.
Note that \begin{align*}
x_3(x_1^2+x_2x_3+x_3x_2)-(x_1^2+x_2x_3+x_3x_2)x_3&=x_3x_1^2-x_1^2x_3+x_3^2x_2-x_2x_3^2\\
&=\partial_{\mathscr{A}}(x_1x_2+x_2x_1).
\end{align*}
Thus $\lceil x_3\rceil \lceil x_1^2+x_2x_3+x_3x_2\rceil =\lceil x_1^2+x_2x_3+x_3x_2\rceil\lceil x_3\rceil$ in $H(\mathscr{A})$.
We assume inductively that we have proved  $$H^{2i-1}(\mathscr{A})=\mathbbm{k} \lceil x_3\rceil \lceil x_1^2+x_2x_3+x_3x_2\rceil^{i-1}\quad\text{and}\quad H^{2i}(\mathscr{A})=\mathbbm{k}\lceil x_1^2+x_2x_3+x_3x_2\rceil^{i},$$
for $i\le n-1$.
For any $x_1a_1+x_2a_2+x_3a_3\in Z^{2n-1}(\mathscr{A})$, we have
\begin{align*}
0&=\partial_{\mathscr{A}}(x_1a_1+x_2a_2+x_3a_3)\\
&=x_3^2a_1-x_1\partial_{\mathscr{A}}(a_1)+(x_1x_3+x_3x_1)a_2-x_2\partial_{\mathscr{A}}(a_2)-x_3\partial_{\mathscr{A}}(a_3)\\
&=x_1(x_3a_2-\partial_{\mathscr{A}}(a_1))-x_2\partial_{\mathscr{A}}(a_2) +x_3(x_3a_1+x_1a_2-\partial_{\mathscr{A}}(a_3)).\\
\end{align*}
So $$\begin{cases}
\partial_{\mathscr{A}}(a_2)=0\\
\partial_{\mathscr{A}}(a_1)=x_3a_2\\
\partial_{\mathscr{A}}(a_3)=x_3a_1+x_1a_2.
\end{cases}$$
Then $a_2=l_2(x_1^2+x_2x_3+x_3x_2)^{n-1}+\partial_{\mathscr{A}}(b_2)$, for some $l_2\in k$ and $b_2\in \mathscr{A}^{2n-3}$.
Since $\partial_{\mathscr{A}}(a_1)=l_2x_3(x_1^2+x_2x_3+x_3x_2)^{n-1}+x_3\partial_{\mathscr{A}}(b_2)$, we have $l_2=0$ and $$a_1=-x_3 b_2+l_1(x_1^2+x_2x_3+x_3x_2)^{n-1}+\partial_{\mathscr{A}}(b_1),$$ for some $l_1\in \mathbbm{k}$ and $b_1\in \mathscr{A}^{2n-3}$. Then
$$\partial_{\mathscr{A}}(a_3)=-x_3^2b_2+l_1x_3(x_1^2+x_2x_3+x_3x_2)^{n-1}+x_3\partial_{\mathscr{A}}(b_1)+x_1\partial_{\mathscr{A}}(b_2).$$
It implies that $l_1=0$ and $a_3=-x_1b_2-x_3b_1+l_3(x_1^2+x_2x_3+x_3x_2)^{n-1}+\partial_{\mathscr{A}}(b_3)$ for some $l_3\in \mathbbm{k}$ and $b_3\in  (x_1^2+x_2x_3+x_3x_2)^{n-1}$. Then
\begin{align*}
x_1a_1+x_2a_2+x_3a_3&=-x_1x_3b_2+x_1\partial_{\mathscr{A}}(b_1)+x_2\partial_{\mathscr{A}}(b_2)\\
&+x_3[-x_1b_2-x_3b_1+l_3(x_1^2+x_2x_3+x_3x_2)^{n-1}+\partial_{\mathscr{A}}(b_3)]\\
&=\partial_{\mathscr{A}}(-x_2b_2-x_1b_1-x_3b_3)+l_3x_3(x_1^2+x_2x_3+x_3x_2)^{n-1}.
\end{align*}
So $H^{2n-1}(\mathscr{A})=\mathbbm{k}(\lceil x_3\rceil \lceil x_1^2+x_2x_3+x_3x_2\rceil^{n-1})$.

For any $x_1r_1+x_2r_2+x_3r_3\in Z^{2n}(\mathscr{A})$, we have
\begin{align*}
0&=\partial_{\mathscr{A}}(x_1r_1+x_2r_2+x_3r_3)\\
&=x_3^2r_1-x_1\partial_{\mathscr{A}}(r_1)+(x_1x_3+x_3x_1)r_2-x_2\partial_{\mathscr{A}}(r_2)-x_3\partial_{\mathscr{A}}(r_3)\\
&=x_1(x_3r_2-\partial_{\mathscr{A}}(r_1))-x_2\partial_{\mathscr{A}}(r_2) +x_3(x_3r_1+x_1r_2-\partial_{\mathscr{A}}(r_3)).\\
\end{align*}
Thus $$\begin{cases}
\partial_{\mathscr{A}}(r_2)=0\\
\partial_{\mathscr{A}}(r_1)=x_3r_2\\
\partial_{\mathscr{A}}(r_3)=x_3r_1+x_1r_2.
\end{cases}$$
So $r_2=t_2x_3(x_1^2+x_2x_3+x_3x_2)^{n-1}+\partial_{\mathscr{A}}(s_2)$ for some $t_2\in \mathbbm{k}$ and $s_2\in \mathscr{A}^{2n-2}$.
Then $\partial_{\mathscr{A}}(r_1)=t_2x_3^2(x_1^2+x_2x_3+x_3x_2)^{n-1}+x_3\partial_{\mathscr{A}}(s_2)$. It implies that
$$r_1=t_2x_1(x_1^2+x_2x_3+x_3x_2)^{n-1}-x_3s_2+t_1x_3(x_1^2+x_2x_3+x_3x_2)^{n-1}+\partial_{\mathscr{A}}(s_1),$$ for some $t_1\in \mathbbm{k}$ and $s_1\in \mathscr{A}^{2n-2}$. Then
\begin{align*}
& \partial_{\mathscr{A}}(r_3)=x_3r_1+x_1r_2 \\ &=t_2x_3x_1(x_1^2+x_2x_3+x_3x_2)^{n-1}-x_3^2s_2+t_1x_3^2(x_1^2+x_2x_3+x_3x_2)^{n-1}\\
& \quad +x_3\partial_{\mathscr{A}}(s_1)+t_2x_1x_3(x_1^2+x_2x_3+x_3x_2)^{n-1}+x_1\partial_{\mathscr{A}}(s_2)\\
&=\partial_{\mathscr{A}}[t_2x_2(x_1^2+x_2x_3+x_3x_2)^{n-1}-x_1s_2+t_1x_1(x_1^2+x_2x_3+x_3x_2)^{n-1}-x_3s_1].
\end{align*}
It implies that \begin{align*}
r_3 &=t_2x_2(x_1^2+x_2x_3+x_3x_2)^{n-1}-x_1s_2+t_1x_1(x_1^2+x_2x_3+x_3x_2)^{n-1}-x_3s_1\\
&\quad\quad +t_3x_3(x_1^2+x_2x_3+x_3x_2)^{n-1}+\partial_{\mathscr{A}}(s_3),
\end{align*}
for some $t_3\in \mathbbm{k}$ and $s_3\in \mathscr{A}^{2n-2}$. Then
\begin{align*}
&x_1r_1+x_2r_2+x_3r_3\\
&=x_1[t_2x_1(x_1^2+x_2x_3+x_3x_2)^{n-1}-x_3s_2+t_1x_3(x_1^2+x_2x_3+x_3x_2)^{n-1}+\partial_{\mathscr{A}}(s_1)]\\
&+x_2[t_2x_3(x_1^2+x_2x_3+x_3x_2)^{n-1}+\partial_{\mathscr{A}}(s_2)]+x_3[t_2x_2(x_1^2+x_2x_3+x_3x_2)^{n-1}-x_1s_2]\\
&+x_3[t_1x_1(x_1^2+x_2x_3+x_3x_2)^{n-1}-x_3s_1+t_3x_3(x_1^2+x_2x_3+x_3x_2)^{n-1}+\partial_{\mathscr{A}}(s_3)]\\
&=t_2(x_1^2+x_2x_3+x_3x_2)^{n}+t_1(x_1x_3+x_3x_1)(x_1^2+x_2x_3+x_3x_2)^{n-1}\\
&-(x_1x_3+x_3x_1)s_2+x_2\partial_{\mathscr{A}}(s_2)+x_1\partial_{\mathscr{A}}(s_1)-x_3^2s_1+x_3\partial_{\mathscr{A}}(s_3)\\
&=t_2(x_1^2+x_2x_3+x_3x_2)^{n}+\partial_{\mathscr{A}}[t_1x_2(x_1^2+x_2x_3+x_3x_2)^{n-1}-x_2s_2-x_1s_1-x_3s_3].
\end{align*}
This indicates that $H^{2n}(\mathscr{A})=\mathbbm{k}\lceil x_1^2+x_2x_3+x_3x_2\rceil^n$.
By the induction above, we conclude that $$H(\mathscr{A})=\mathbbm{k}[\lceil x_3\rceil, \lceil x_1^2+x_2x_3+x_3x_2\rceil ]/(\lceil x_3\rceil^2).$$
Applying the constructing procedure of minimal semi-free resolution described in the proof of \cite[Proposition 2.4]{MW1}, we can construct a semi-free DG $\mathscr{A}$-module $F$ such that
$$F^{\#}=\mathscr{A}^{\#}\oplus \mathscr{A}^{\#}\Sigma e_{x_3}\oplus \mathscr{A}^{\#}\Sigma e_{z}\oplus \mathscr{A}^{\#}\Sigma e_{r}$$ and the differential
$\partial_{F}$ is defined by
$$\left(
    \begin{array}{c}
      \partial_F(1) \\
      \partial_F(\Sigma e_{x_3}) \\
      \partial_F(\Sigma e_z) \\
       \partial_F(\Sigma e_r)\\
    \end{array}
  \right)=\left(
            \begin{array}{cccc}
              0   & 0   & 0   & 0 \\
              x_3 & 0   & 0   & 0 \\
              x_1 & x_3 & 0   & 0 \\
              x_2 & x_1 & x_3 & 0\\
            \end{array}
          \right)\left(
                   \begin{array}{c}
                     1 \\
                     \Sigma e_{x_3} \\
                     \Sigma e_z \\
                     \Sigma e_r \\
                   \end{array}
                 \right).$$
Note that $F$ admits a semi-basis $\{1, \Sigma e_{x_3},\Sigma e_z, \Sigma e_r\}$ concentrated in degree $0$. Hence it is a Koszul and homoloigcally smooth DG algebra.
By the minimality of $F$, the differential of $\Hom_{\mathscr{A}}(F,\mathbbm{k})$ is zero. We have
\begin{align*}
H(\Hom_{\mathscr{A}}(F,\mathbbm{k}))=\Hom_{\mathscr{A}}(F,k)=\mathbbm{k}1^*\oplus \mathbbm{k}(\Sigma e_{x_3})^*\oplus \mathbbm{k}(\Sigma e_z)^*\oplus \mathbbm{k}(\Sigma e_r)^*.
\end{align*}
So the Ext-algebra $E=H(\Hom_{\mathscr{A}}(F,F))\cong H(\Hom_{\mathscr{A}}(F,k))$ is concentrated in degree $0$. On the other hand, $$\Hom_{\mathscr{A}}(F,F)^{\#}\cong \{\mathbbm{k}1^*\oplus \mathbbm{k}(\Sigma e_{x_3})^*\oplus \mathbbm{k}(\Sigma e_z)^*\oplus \mathbbm{k}(\Sigma e_r)^*\}\otimes_\mathbbm{k} F^{\#}$$ is concentrated in degrees $\ge 0$ since $|1^*|=|(\Sigma e_{x_3})^*|=|(\Sigma e_{z})^*|=|(\Sigma e_r)^*|=0$ and $F$ is concentrated in degrees $\ge 0$. This implies that $E= Z^0(\Hom_{\mathscr{A}}(F,F))$.
Since $F^{\#}$ is a free graded $\mathscr{A}^{\#}$-module with a basis
$\{1, \Sigma e_{x_3}, \Sigma e_{z},\Sigma e_r \}$ concentrated in degree $0$,
  the elements in  $\Hom_{\mathscr{A}}(F,F)^0$ is one to one correspondence with the matrixes in $M_4(\mathbbm{k})$. Indeed, any $f\in \Hom_{\mathscr{A}}(F,F)^0$ is uniquely determined by
  a matrix $A_f=(a_{ij})_{4\times 4}\in M_4(\mathbbm{k})$ with
$$\left(
                         \begin{array}{c}
                          f(1) \\
                          f(\Sigma e_{x_3})\\
                           f(\Sigma e_z)\\
                            f(\Sigma e_r)
                         \end{array}
                       \right) =      A_f \cdot \left(
                         \begin{array}{c}
                          1 \\
                          \Sigma e_{x_3}\\
                           \Sigma e_z\\
                            \Sigma e_r
                         \end{array}
                       \right).  $$
                       And $f\in  Z^0[\Hom_{\mathcal{A}}(F,F)]$ if and only if $\partial_{F}\circ f=f\circ \partial_{F}$, if and only if

 $$ A_f \left(
            \begin{array}{cccc}
              0 & 0 & 0 & 0 \\
              x_3 & 0& 0 & 0 \\
              x_1 & x_3 & 0 & 0 \\
              x_2 & x_1 & x_3 & 0\\
            \end{array}
          \right) = \left(
            \begin{array}{cccc}
              0 & 0 & 0 & 0 \\
              x_3 & 0& 0 & 0 \\
              x_1 & x_3 & 0 & 0 \\
              x_2 & x_1 & x_3 & 0\\
            \end{array}
          \right) A_f,$$
 which is also equivalent to
                       $$\begin{cases}
                       a_{ij}=0, \forall i<j\\
                       a_{11}=a_{22}=a_{33}= a_{44}\\
                       a_{32}=a_{21}=a_{43} \\
                       a_{31}=a_{42}
                       \end{cases}$$
by direct computations. Hence the the Ext-algebra $$ E\cong \left\{\left(
                                                             \begin{array}{cccc}
                                                               a & 0 & 0 & 0  \\
                                                               b & a & 0 & 0 \\
                                                               c & b & a & 0  \\
                                                               d & c & b & a
                                                             \end{array}
                                                           \right)
\quad | \quad a,b,c,d\in \mathbbm{k} \right\}\cong \mathbbm{k}[x]/(x^4)$$
is a symmetric Frobenius graded algebra. By Theorem \ref{cycond}, $\mathscr{A}$ is a Koszul and Calabi-Yau DG algebra.
\end{proof}

\begin{ex}\label{ex5}
Let $\mathcal{A}$ be a connected cochain DG algebra such that
\begin{align*}
\mathscr{A}^{\#}=\mathbbm{k}\langle x_1,x_2,x_3\rangle,\,\,
\partial_{\mathscr{A}}(x_1)=x_2x_3+x_3x_2,\,\, \partial_{\mathscr{A}}(x_2)=0 \,\, \partial_{\mathscr{A}}(x_3)=0.
\end{align*}
\end{ex}

\begin{rem}
Let $\mathscr{A}$ be the DG algebra in Example \ref{ex5}. By definition, we have
$H^0(\mathscr{A})=\mathbbm{k}, H^1(\mathscr{A})=\mathbbm{k}\lceil x_2\rceil\oplus \mathbbm{k}\lceil x_3\rceil$  and $ H^2(\mathscr{A})=\mathbbm{k}\lceil x_2\rceil^2\oplus \mathbbm{k}\lceil x_2x_3\rceil\oplus \mathbbm{k}\lceil x_3\rceil^2$.
By tedious computation, we have
$H(\mathscr{A})=\mathbbm{k}\langle \lceil x_2\rceil,\lceil x_3\rceil \rangle/(\lceil x_2\rceil\cdot \lceil x_3\rceil+\lceil x_3\rceil\cdot \lceil x_2\rceil).$
By \cite[Theorem 2.6]{MH}, $\mathscr{A}$ is a Koszul Calabi-Yau DG algebra.
\end{rem}

%%% ------------------------------------------------
\section*{Acknowledgments}
The work is supported by the National Natural Science Foundation of
China (No. 11871326), by Key Disciplines of Shanghai Municipality(No. S30104), and the Innovation Program of
Shanghai Municipal Education Commission (No. 12YZ031).

%%% ------------------------------------------------.

\def\refname{References}

\end{document}